\documentclass{amsart}

\usepackage{stmaryrd}
\usepackage{enumerate}
\usepackage[pdftex]{graphicx}
\usepackage{amsmath}
\usepackage{amsthm}
\usepackage{multirow}

\newtheorem*{stirling'sformula}{Stirling's Formula}
\newtheorem{thm}{Theorem}[section]
\newtheorem{prop}[thm]{Proposition}
\newtheorem{lem}[thm]{Lemma}
\newtheorem{cor}[thm]{Corollary}

\theoremstyle{definition}
\newtheorem{definition}[thm]{Definition}

\theoremstyle{remark}
\newtheorem{notation}[thm]{Notation}

\numberwithin{equation}{section}

\newcommand{\Apc}{\pazocal{A}} 
\newcommand{\Bpc}{\pazocal{B}} 
\newcommand{\Cpc}{\pazocal{C}} 
\newcommand{\Dpc}{\pazocal{D}} 
\newcommand{\Epc}{\pazocal{E}} 
\newcommand{\Fpc}{\pazocal{F}}

\newcommand{\Ipc}{\pazocal{I}}

\newcommand{\Mpc}{\pazocal{M}} 
\newcommand{\Npc}{\pazocal{N}} 
 
\newcommand{\Ppc}{\pazocal{P}} 
\newcommand{\Qpc}{\pazocal{Q}} 
 
\newcommand{\Spc}{\pazocal{S}} 
\newcommand{\Tpc}{\pazocal{T}} 
 
\newcommand{\Vpc}{\pazocal{V}} 
\newcommand{\Wpc}{\pazocal{W}} 
\newcommand{\Xpc}{\pazocal{X}} 
 
\newcommand{\Zpc}{\pazocal{Z}}

\newcommand{\Fmc}{\mathcal{F}}

\newcommand{\Dbbb}{\mathbb{D}}

\newcommand{\Pbbb}{\mathbb{P}} 
 
\newcommand{\Rbbb}{\mathbb{R}}

\newcommand{\Zbbb}{\mathbb{Z}} 

\newcommand{\amf}{\mathfrak{a}}

\newcommand{\lmf}{\mathfrak{l}}

\newcommand{\smf}{\mathfrak{s}}

\newcommand{\Std}{\widetilde{S}}

\newcommand{\btd}{\widetilde{b}}

\newcommand{\etd}{\widetilde{e}}

\newcommand{\ptd}{\widetilde{p}}

\newcommand{\std}{\widetilde{s}}

\newcommand{\ztd}{\widetilde{z}}

\DeclareMathAlphabet{\pazocal}{OMS}{zplm}{m}{n}

\DeclareMathOperator{\suc}{suc}

\DeclareMathOperator{\Sec}{Sec}

\DeclareMathOperator{\Span}{Span}

\DeclareMathOperator{\Vol}{Vol}

\DeclareMathOperator{\id}{id}

\DeclareMathOperator{\Hom}{Hom}
\DeclareMathOperator{\htop}{h_{top}}

\newcommand{\sub}[2]{L^{(#2)}_{#1}} 

\begin{document}

\title{Degeneration of Hitchin representations along internal sequences}

\author{Tengren Zhang}

\thanks{The author was partially supported by U.S. National Science Foundation grants DMS - 1006298, DMS - 1306992, DMS - 1307164 and DMS - 1107452, 1107263, 1107367 "RNMS: Geometric structures And Representation varieties" (the GEAR Network).}

\begin{abstract}
We first give a coordinate system on the $PSL(n,\Rbbb)$ Hitchin component that is a direct analogue of the Fenchel-Nielsen coordinates on Teichm\"uller space. Using these coordinates, we consider a class of sequences in the Hitchin component, called internal sequences, along which most length functions grow to infinity and the topological entropy of the associated geodesic flow converges to $0$.
\end{abstract}

\maketitle

%%%%%%%%%%%%%%%%%%%%%%%%%%%%%%%%%%%%%%%%%%%%%%%%%%%%%%%%%%%%%%%%%%%%%%%%%%%%%%%%%%%%%%%%%%%%%%%%%%%%
\section{Introduction} 
%%%%%%%%%%%%%%%%%%%%%%%%%%%%%%%%%%%%%%%%%%%%%%%%%%%%%%%%%%%%%%%%%%%%%%%%%%%%%%%%%%%%%%%%%%%%%%%%%%%%
Let $S$ be a closed, oriented topological surface of genus at least $2$, and denote its fundamental group by $\Gamma$. The Teichm\"uller space of $S$, denoted by $\Tpc(S)$, is the space of marked hyperbolic metrics on $S$. From a representation theoretic point of view, one can think of $\Tpc(S)$ as a connected component of the space of conjugacy classes of discrete, faithful representations from $\Gamma$ to $PSL(2,\Rbbb)$. An advantage of taking this point of view is that it allows one to define a higher rank generalization of $\Tpc(S)$, which was introduced by Hitchin \cite{Hit1}. Presently, this ``higher Teichm\"uller space" is known as the {\em Hitchin component}, and can be defined as follows. Let $\iota:PSL(2,\Rbbb)\to PSL(n,\Rbbb)$ be the unique (up to conjugation) irreducible representation. This induces, via post-composition, an embedding $i$ from $\Tpc(S)$ into the character variety
\[\Xpc_n(S):=\Hom(\Gamma,PSL(n,\Rbbb))/ PSL(n,\Rbbb).\]
The {\em $n$-th Hitchin component of $S$}, denoted $Hit_n(S)$, can then be defined to be the connected component of $\Xpc_n(S)$ that contains the image of $i$, which is also known as the \emph{Fuchsian locus}. 

It is well-known that $\Tpc(S)=Hit_2(S)$. Also, by the work of Choi-Goldman \cite{ChoGol1} and Guichard-Wienhard \cite{GuiWie1}, we know respectively that $Hit_3(S)$ is the space of marked convex $\Rbbb\Pbbb^2$ structures on $S$ and that $Hit_4(S)$ is the space of marked convex foliated $\Rbbb\Pbbb^3$ structures on $T^1S$. These realizations of the lower rank Hitchin components as deformations spaces of geometric structures associated to $S$ provide a strong motivation for studying Hitchin representations from a geometric point of view. Guichard-Wienhard \cite{GuiWie2} also constructed domains of discontinuities for the image of Hitchin representations in $Hit_n(S)$ for all $n$. 

Interestingly, the Hitchin components have many of the desirable properties that $\Tpc(S)$ possess.  Hitchin \cite{Hit1} proved using Higgs bundle techniques that $Hit_n(S)$ is a cell of real dimension $(n^2-1)(2g-2)$, where $g$ is the genus of $S$. By understanding the dynamics of the $\Gamma$-action  induced by Hitchin representations on the space of complete flags in $\Rbbb^n$, Labourie \cite{Lab1} proved that they are discrete, faithful, and their images consist only of diagonalizable elements with eigenvalues that have pairwise distinct norms. Using the work of Fock-Goncharov \cite{FocGon1}, Bonahon-Dreyer \cite{BonDre1} gave a real-analytic parameterization of $Hit_n(S)$ that is a generalization of Thurston's shear coordinates in $\Tpc(S)$. There is also a parameterization of $Hit_3(S)$ by Goldman \cite{Gol1} which generalizes the Fenchel-Nielsen coordinates on $\Tpc(S)$.

By taking a special case of the parameterization by Bonahon-Dreyer \cite{BonDre1} and performing a linear reparameterization, one can obtain another parameterization that is explicitly analogous to the Fenchel-Nielsen coordinates on $\Tpc(S)$. More specifically, if we choose an oriented pants decomposition $\Ppc$ of $S$, then $Hit_n(S)$ can be parameterized by the following:
\begin{itemize}
\item A \emph{boundary invariant} taking values in $\amf^+$ for each curve in $\Ppc$.
\item $n-1$ \emph{gluing parameters} taking values in $\Rbbb$ for each curve in $\Ppc$.
\item $(n-1)(n-2)$ \emph{internal parameters} taking values in $\Rbbb$ for each pair of pants given by $\Ppc$.
\end{itemize}
Here, $\amf^+$ is the positive Weyl chamber of the Lie algebra $\smf\lmf(n,\Rbbb)$, which one can think of as the set of traceless $n\times n$ diagonal matrices with diagonal entries that are strictly decreasing down the diagonal. In fact, the boundary invariant for a curve in $\Ppc$ is the image of the corresponding group element under the Jordan projection. Here, one should think of the boundary invariant and gluing parameters as analogs of the Fenchel-Nielsen length and twist coordinates respectively. We call this parameterization of $Hit_n(S)$ the \emph{modified shear-triangle parameterization}.

In view of this parameterization, one can ask if there is any geometric meaning behind deforming the internal parameters. One way to approach this question is to study sequences $\{\rho_i\}_{i=1}^\infty$ in $Hit_n(S)$, along which the boundary invariants are held bounded away from the walls of $\amf^+$, while the $(n-1)(n-2)$ internal parameters for each pair of pants escape every compact set in the cell where they take values. Such sequences are called \emph{internal sequences} (see Definition \ref{internal sequence}). More informally, the internal sequences are those where we do not change the boundary invariants by much while deforming the internal parameters as much as possible. 

One can also study Hitchin representations by considering the induced flows on $T^1S$. For any Hitchin representation $\rho$, define the length function $l_\rho:\Gamma\to\Rbbb$ by 
\[l_\rho(X)=\log\bigg|\frac{\omega_n(\rho(X))}{\omega_1(\rho(X))}\bigg|,\]
where $\omega_n(\rho(X))$ and $\omega_1(\rho(X))$ are the eigenvalues of $\rho(X)$ with largest and smallest norm respectively. These are related to the length functions studied by Dreyer \cite{Dre1}. When $n=2$, $l_\rho(X)$ is the length of the closed geodesic in $S$ corresponding to $X$ in $\Gamma$, measured in the hyperbolic metric on $S$ corresponding to the representation $\rho$. 

Labourie \cite{Lab2} constructed, for each $\rho$, a unique (up to H\"older time preserving equivalence) H\"older reparameterization $(\phi_\rho)_t$ of the geodesic flow on $T^1S$, so that the closed orbit of $(\phi_\rho)_t$ corresponding to the conjugacy class $[X]$ in $[\Gamma]$ has period $l_\rho(X)$. In the case when $\rho$ lies in $\Tpc(S)$, $(\phi_\rho)_t$ is conjugate to the geodesic flow of the hyperbolic metric corresponding to $\rho$ via a H\"older time-preserving homeomorphism. This allows us to define the \emph{topological entropy of $\rho$}, $\htop(\rho)$, to be the topological entropy of the flow $(\phi_\rho)_t$. It then follows from the work of Bowen \cite{Bow1} and Pollicot \cite{Pol1} that 
\[\htop(\rho)=\limsup_{T\to\infty}\frac{1}{T}\log|\{[X]\in[\Gamma]:l_\rho(X)<T\}|.\]

In this paper, we study how the dynamics of these induced flows degenerate along internal sequences. The main goal is to prove the following theorem.

\begin{thm}\label{main theorem}
There exists a continuous function $\Theta:Hit_n(S)\to\Rbbb^+$ with the following properties:
\begin{enumerate}
\item If $X$ in $\Gamma$ does not correspond to a curve homotopic to a multiple of a curve in $\Ppc$, then $\Theta(\rho)\leq l_\rho(X)$. 
\item If $\{\rho_i\}_{i=1}^\infty$ is an internal sequence, then 
\[\lim_{i\to\infty}\Theta(\rho_i)=\infty.\]
\end{enumerate}
Furthermore,
\[\lim_{i\to\infty}\htop(\rho_i)=0.\]\end{thm}

This theorem is a generalization of the results in \cite{Zha1}, where the author proved the same statement for $Hit_3(S)$ using the Goldman parameterization. Nie (\cite{Nie1}, \cite{Nie2}) also has some related results. In addition to the consequences mentioned in \cite{Zha1}, the main theorem has several other interesting geometric corollaries. Let $M$ be the symmetric space of $SL(n,\Rbbb)$. Define the \emph{critical exponent of $\rho$} by
\[h_M(\rho):=\limsup_{T\to\infty}\frac{1}{T}\log|\{X\in\Gamma:d_M(o,\rho(X)\cdot o)<T\}|,\]
where $d_M$ is the distance function on $M$ induced by the Riemannian metric, and $o$ is any point in $M$. This quantity is the exponential growth rate of the number of points in the $\Gamma$-orbit of $o$ that are contained in a ball of growing radius centered about $o$, and is in fact independent of the choice of $o$. Theorem \ref{main theorem} then allows us to deduce how the critical exponent degenerates along internal sequences.

\begin{cor}\label{first corollary}
Let $\{\rho_i\}_{i=1}^\infty$ be an internal sequence. Then
\[\lim_{i\to\infty}h_M(\rho_i)=0.\]
\end{cor}

\begin{proof}
Let $o\in M$ be the point stabilized by $PSO(n)$ in $PSL(n,\Rbbb)$ and $\mu:G\to\overline{\amf^+}$ the Cartan projection corresponding to $PSO(n)$ (See Section \ref{Cartan and Jordan projections}). One can verify that for any $g$ in $PSL(n,\Rbbb)$,
\begin{equation*}
d_M(o,g\cdot o)=||\mu(g)||:=c_n\sqrt{\sum_{i=1}^n\mu_i(g)^2}
\end{equation*}
where $\mu_1(g)\leq\mu_2(g)\leq\dots\leq\mu_n(g)$ are the eigenvalues of $\mu(g)$ and $c_n$ is a constant depending only on $n$. Also, by Corollary 4.4 of Sambarino \cite{Sam2},
\[\htop(\rho)=\lim_{T\to\infty}\frac{1}{T}\log|\{X\in\Gamma:\mu_n(\rho(X))-\mu_1(\rho(X))<T\}|\]
for any $\rho$ in $Hit_n(S)$. Using the fact that $\displaystyle\sum_{i=1}^n\mu_i(g)=0$ for all $g$ in $PSL(n,\Rbbb)$, we can deduce that
\[||\mu(g)||\geq\frac{c_n}{n}\big(\mu_n(g)-\mu_1(g)\big),\]
which implies that for any $\rho$ in $Hit_n(S)$,
\begin{eqnarray*}
\htop(\rho)&=&\frac{c_n}{n}\limsup_{T\to\infty}\frac{1}{T}\log\Big|\Big\{X\in\Gamma:\frac{c_n}{n}\Big(\mu_n\big(\rho(X)\big)-\mu_1\big(\rho(X)\big)\Big)<T\Big\}\Big|\\
&\geq&\frac{c_n}{n}\limsup_{T\to\infty}\frac{1}{T}\log|\{X\in\Gamma:||\mu(\rho(X))||<T\}|\\
&=&\frac{c_n}{n}h_M(\rho).
\end{eqnarray*}
The corollary follows immediately from Theorem \ref{main theorem}.
\end{proof}

This corollary has further implications for the minimal immersions that arise from Hitchin representations. For any $\rho$ in $Hit_n(S)$ and any conformal structure $\Sigma$ on $S$, a special case of the work of Corlette \cite{Cor1} or Eells-Sampson \cite{EelSam1} implies the existence of a unique (up to $PSL(n,\Rbbb)$ action) harmonic map 
\[f:\Sigma\to\rho(\Gamma)\backslash M.\] 
Labourie \cite{Lab2} then proved that for every $\rho$ in $Hit_n(S)$, there are conformal structures $\Sigma$ on $S$ so that $f$ is a branched minimal immersion. Recently, Sanders \cite{San1} showed that these harmonic maps $f$ are in fact always immersions, and they satisfy the following inequality:
\[\frac{1}{\Vol(f^*m)}\int_\Sigma\sqrt{-\Sec(T_{f(p)}f(\Sigma))+\frac{1}{2}||B_f(p)||^2}\,\mathrm{d}V(p)\leq h_M(\rho).\]
Here, $\Sec$ is the sectional curvature in $\rho(\Gamma)\backslash M$, $m$ is the Riemannian metric on $\rho(\Gamma)\backslash M$, $\mathrm{d}V$ is the volume measure of $f^*m$ and $B_f$ is the second fundamental form of $f$. By the Gauss equation, one sees (see Section 6.1 of Sanders \cite{San1}) that if $f$ is a minimal immersion, then $\Sec(T_{f(p)}f(\Sigma))\leq 0$ for all points $p$ in $\Sigma$. This, together with Corollary \ref{first corollary}, allows us to conclude the following.

\begin{cor}
Let $\{\rho_i\}_{i=1}^\infty$ be an internal sequence in $Hit_n(S)$, and let $\Sigma_i$ be a conformal structure on $S$ for which the harmonic immersion $f_i:\Sigma_i\to\rho_i(\Gamma)\backslash M$ is minimal. Then
\[\lim_{i\to\infty}\frac{1}{\Vol(f_i^*m_i)}\int_{\Sigma_i}\sqrt{-\Sec_i(T_{f_i(p)}f_i(\Sigma_i))}\,\mathrm{d}V_i(p)=0\]
and
\[\lim_{i\to\infty}\frac{1}{\Vol(f_i^*m_i)}\int_{\Sigma_i}||B_{f_i}(p)||\,\mathrm{d}V_i(p)=0.\]
Here, $\Sec_i$ is the sectional curvature in $\rho_i(\Gamma)\backslash M$, $m_i$ is the Riemannian metric on $\rho_i(\Gamma)\backslash M$, $dV_i$ is the volume measure of $f_i^*m_i$ and $B_{f_i}$ is the second fundamental form of $f_i$.
\end{cor}
More informally, this corollary says that the minimal immersions corresponding to the Hitchin representations along any internal sequence are on average becoming flatter and more totally geodesic as we move along the sequence. Collier-Li \cite{ColLi1} also have results that are similar in flavor to this corollary.

We will now give a sketch of the proof of the main theorem in three main steps. Choose a hyperbolic metric on $S$ and consider the ideal triangulation on $S$ that is obtained by further subdividing each pair of pants given by $\Ppc$ into two ideal triangles. Also, fix $\rho$ in $Hit_n(S)$. For the first step, we obtain a combinatorial description (which depends on $\rho$) of every oriented closed geodesic $\gamma$ on $S$ using the intersection pattern of $\gamma$ with the ideal triangulation. Roughly, this combinatorial description keeps track of how many times $\gamma$ ``winds around" a collar neighborhood of a simple closed curve in $\Ppc$, and how many times $\gamma$ ``crosses between" these collar neighborhoods. By design, two oriented geodesics on $S$ have the same combinatorial description if and only if they are the same oriented geodesic (Proposition \ref{combinatorial prop}).

In the second step, we find, for any Hitchin representation $\rho$ and any $X$ in $\Gamma$, a lower bound for $l_\rho(X)$ of the form
\begin{equation}\label{inequality intro}
l_\rho(X)\geq r(X)\cdot K(\rho)+s_\rho(X)\cdot L(\rho).
\end{equation}
(Theorem \ref{length lower bound}). Here, $K,L:Hit_n(S)\to\Rbbb^+$ are continuous functions, $s_\rho(X)$ is the number of times the oriented closed geodesic $\gamma$ corresponding to $X$ ``winds around" collar neighborhoods of the simple closed curves in $\Ppc$, and $r(X)$ is the number of times $\gamma$ ``crosses between" these collar neighborhoods. The quantities $r(X)$ and $s_\rho(X)$ are non-negative integers that depend only on the combinatorial description of $X$ in the first step. Also, $L(\rho)$ is a multiple of the length of the shortest curve in $\Ppc$. Informally, $K(\rho)$ is the length picked up by $\gamma$ whenever it ``crosses between'' collar neighborhoods of the curves in $\Ppc$. 

Finally, in the third step, we show that $\displaystyle\lim_{i\to\infty}K(\rho_i)=\infty$ for any internal sequence $\{\rho_i\}_{i=1}^\infty$ in $Hit_n(S)$. This fact, combined with an elementary counting argument demonstrated in \cite{Zha1}, proves the main theorem.

The second and third steps of the proof rely heavily on the work of Labourie \cite{Lab1}. He proved that if a representation $\rho$ in $\Xpc_n(S)$ lies in the Hitchin component, then there exists a $\rho$-equivariant Frenet curve $\xi:\partial_\infty\Gamma\to\Fpc(\Rbbb^n)$.  Understanding the way the cross ratio interacts with $\xi$ is of central importance to the arguments used in the second and third steps.

The rest of the paper is structured as follows. In Section \ref{Configurations of flags}, we investigate how the cross ratio interacts with the $\rho$-equivariant Frenet curve and develop the tools required to perform the second and third steps. Then, in Section \ref{Shear-triangle parameterization}, we give a detailed explanation of the modified shear-triangle parameterization, and how to derive it from the work of Bonahon-Dreyer and Fock-Goncharov. The first step of the proof, i.e. the combinatorial description of the oriented closed curves on $S$, and also the proof of the lower bound (\ref{inequality intro}), are described in Section \ref{Lower bound for lengths of closed curves}. Finally, we combine all the proof ingredients and execute the second and third step of the proof of the main theorem in Section \ref{Degeneration along internal sequences}.

{\bf Acknowledgements:}
This work has benefitted from conversations with Andrew Sanders and Andres Sambarino. Much of the content in Section \ref{Parameterizing the Hitchin component} regarding the parameterization of the Hitchin component came from discussions with Sara Maloni, for which the author is very grateful. The author also especially wishes to thank Richard Canary for the invaluable help he provided, both with the mathematics and the exposition in this paper. Finally, the author would like to thank the referee for carefully reading this paper and providing many useful comments. 

\tableofcontents

%%%%%%%%%%%%%%%%%%%%%%%%%%%%%%%%%%%%%%%%%%%%%%%%%%%%%%%%%%%%%%%%%%%%%%%%%%%%%%%%%%%%%%%%%%%%%%%%%%%%
\section{Configurations of flags}\label{Configurations of flags}
%%%%%%%%%%%%%%%%%%%%%%%%%%%%%%%%%%%%%%%%%%%%%%%%%%%%%%%%%%%%%%%%%%%%%%%%%%%%%%%%%%%%%%%%%%%%%%%%%%%%

%%%%%%%%%%%%%%%%%%%%%%%%%%%%%%%%%%%%%%%%%%%%%%%%%%
\subsection{Flags and the Hitchin component}\label{Flags and the Hitchin component}
%%%%%%%%%%%%%%%%%%%%%%%%%%%%%%%%%%%%%%%%%%%%%%%%%%

Let $\Fmc(\Rbbb^n)$ be the space of complete flags in $\Rbbb^n$, i.e. a nested sequence of $n$ linear subspaces in $\Rbbb^n$, each properly contained in its predecessor. We start by defining a special embedding of the circle into $\Fpc(\Rbbb^n)$.

\begin{notation}
For any $F$ in $\Fpc(\Rbbb^n)$, let $F^{(l)}$ be the $l$-dimensional subspace for $F$. 
\end{notation}

\begin{definition}
A closed curve $\xi:S^1\to\Fmc(\Rbbb^n)$ is \emph{Frenet} if the following two conditions are satisfied:
\begin{enumerate}
\item Let $x_1,\dots,x_k$ be pairwise distinct points in $S^1$ and let $n_1,\dots,n_k$ be positive integers so that $\displaystyle\sum_{i=1}^k n_i=n$. Then 
\[\sum_{i=1}^k\xi(x_i)^{(n_i)}=\Rbbb^n.\]
\item Let $x_1,\dots,x_k$ be pairwise distinct points in $S^1$ and let $n_1,\dots,n_k$ be positive integers so that $\displaystyle m:=\sum_{i=1}^k n_i\leq n$. Then for any $x\in S^1$, 
\[\underset{x_i\neq x_j,\forall i\neq j}{\lim_{x_i\to x,\forall i}}\sum_{i=1}^k\xi(x_i)^{(n_i)}=\xi(x)^{(m)}.\]
\end{enumerate}
\end{definition}

One should think of Frenet curves as having the property that points along the curve are ``maximally transverse". Often, we will also refer to the image of the Frenet curve $\xi$ by $\xi$.

{\bf For the rest of the paper, let $S=S_g$ be a closed oriented smooth surface of genus $g>1$ and let $\Gamma:=\pi_1(S)$.} It is well-known that $\Gamma$ is Gromov hyperbolic, so the Cayley graph of $\Gamma$ has a natural boundary, denoted by $\partial_\infty\Gamma$. The action of $\Gamma$ on its Cayley graph also extends to an action on $\partial_\infty\Gamma$. Moreover, the fact that $S$ can be equipped with a complete hyperbolic metric allows us to identify $\partial_\infty\Gamma$ with the boundary of the Poincar\'e disc, so $\partial_\infty\Gamma$ is topologically a circle.

Let $\Pi:\Hom(\Gamma,PSL(n,\Rbbb))/PSL(n,\Rbbb)\to\Hom(\Gamma, PGL(n,\Rbbb))/PGL(n,\Rbbb)$ be the obvious projection. When $n$ is odd, $\Pi^{-1}(\Pi(Hit_n(S)))=Hit_n(S)$. On the other hand, when $n$ is even, $\Pi^{-1}(\Pi(Hit_n(S)))$ is a union of two homeomorphic connected components, one of which is $Hit_n(S)$. By the work of Labourie \cite{Lab1} and Guichard \cite{Gui1}, we have the following useful characterization of the representations in $Hit_n(S)$.

\begin{thm}[Guichard, Labourie] \label{Guichard, Labourie}
A representation $\rho$ in the character variety 
\[\Hom(\Gamma,PSL(n,\Rbbb))/ PSL(n,\Rbbb)\]
lies in $\Pi^{-1}(\Pi(Hit_n(S)))$ if and only if there exists a $\rho$-equivariant Frenet curve $\xi:\partial_\infty\Gamma\to\Fmc(\Rbbb^n)$. If $\xi$ exists, then it is uniquely determined (up to $PSL(n,\Rbbb)$) by $\rho$.
\end{thm}

This theorem thus allows us to study any Hitchin representation $\rho$ via its corresponding Frenet curve $\xi$. 

%%%%%%%%%%%%%%%%%%%%%%%%%%%%%%%%%%%%%%%%%%%%%%%%%%
\subsection{Projections of the Frenet curve}\label{Projections of the Frenet curve}
%%%%%%%%%%%%%%%%%%%%%%%%%%%%%%%%%%%%%%%%%%%%%%%%%%

One way we can study a Frenet curve is by considering its projections onto some special projective lines. To that end, we develop the following notation.

\begin{notation}\label{moving subspaces} 
Let $M_1,\dots,M_k$ be pairwise distinct points along a Frenet curve $\xi$ and let $n_1,\dots,n_k$ be positive integers. For any positive integer $\displaystyle m\leq n-\sum_{i=1}^kn_i$ and for any $E$ on $\xi$, define $\sub{E}{m}$ as follows:
\begin{itemize}
\item if $E\neq M_i$ for all $i=1,\dots, k$ then $\sub{E}{m}=E^{(m)}$
\item if $E=M_i$ for some $i=1,\dots,k$, then $\sub{E}{m}$ is a choice of an $m$-dimensional subspace in $M_i^{(n_i+m)}$ that is transverse to $M_i^{(n_i)}$. 
\end{itemize}
\end{notation}
Any statement we make involving the above notation is true for all possible choices of the $m$-dimensional subspaces in $M_i^{(n_i+m)}$ that is transverse to $M_i^{(n_i)}$. 

\begin{lem}\label{important lemma}
Let $A,B$ be distinct points along a Frenet curve $\xi$, let $M_1,\dots,M_k$ be pairwise distinct points along $\xi$ and let $n_1,\dots,n_{k+1}$ be positive integers such that $\displaystyle\sum_{i=1}^{k+1} n_i=n-1$. 
\begin{enumerate}
\item If $n_{k+1}=1$, the map
\[f_1:\xi(S^1)\to\Pbbb(\sub{A}{1}+\sub{B}{1})\]
given by 
\[f_1(E)=\Pbbb\bigg(\sum_{i=1}^k M_i^{(n_i)}+\sub{E}{1}\bigg)\cap\Pbbb(\sub{A}{1}+\sub{B}{1})\]
is a homeomorphism with $f_1(A)=\sub{A}{1}$ and $f_1(B)=\sub{B}{1}$.
\item Let $s$ be a closed subsegment of $\xi$ with endpoints $A$ and $B$.  Then there exists a closed subsegment $t$ of $\Pbbb(\sub{A}{1}+\sub{B}{1})$ with endpoints $\sub{A}{1}$ and $\sub{B}{1}$, such that the map
\[f_{n_{k+1}}:s\to t\]
given by
\[f_{n_{k+1}}(E)=\Pbbb\bigg(\sum_{i=1}^k M_i^{(n_i)}+\sub{E}{n_{k+1}}\bigg)\cap\Pbbb(\sub{A}{1}+\sub{B}{1})\]
is a homeomorphism with $f_{n_{k+1}}(A)=\sub{A}{1}$ and $f_{n_{k+1}}(B)=\sub{B}{1}$.
\end{enumerate}
\end{lem}

\begin{proof}
Proof of (1). The continuity and well-definedness of $f_1$ is clear by the definition of the Frenet curve. Suppose for contradiction that there exist $E\neq E'$ such that $f_1(E)=f_1(E')$. Since $\xi$ is a Frenet curve, we have
\begin{eqnarray*}
\sum_{i=1}^k M_i^{(n_i)}+\sub{E}{1}&=&\sum_{i=1}^k M_i^{(n_i)}+f_1(E)\\
&=&\sum_{i=1}^k M_i^{(n_i)}+f_1(E')\\
&=&\sum_{i=1}^k M_i^{(n_i)}+\sub{E'}{1}.
\end{eqnarray*}
In particular, $\displaystyle\sum_{i=1}^k M_i^{(n_i)}+\sub{E}{1}+\sub{E'}{1}\neq\Rbbb^n$, which contradicts the fact that $\xi$ is a Frenet curve. This proves that $f_1$ is an injective continuous map between two spaces homeomorphic to $S^1$, so $f_1$ has to be a homeomorphism. It is easy to verify that $f_1(A)=\sub{A}{1}$ and $f_1(B)=\sub{B}{1}$.

Proof of (2). As before, the continuity of $f_{n_{k+1}}$ is clear. We will prove that $f_{n_{k+1}}$ is a homeomorphism by induction. The base case when $n_{k+1}=1$ follows from (1). For the inductive step, consider the case when $n_{k+1}=m+1$. Pick any pair of distinct points $E_0$ and $E_1$ in the interior of $s$, and assume without loss of generality that $E_1$ lies between $E_0$ and $B$ on $s$. Since $\xi$ is a Frenet curve, $f_{m+1}(E_0)\neq\sub{A}{1},\sub{B}{1}$, so there is a unique subsegment of $\Pbbb(\sub{A}{1}+\sub{B}{1})$ with endpoints $\sub{A}{1},\sub{B}{1}$ that contains $f_{m+1}(E_0)$. Let this subsegment be $t$.

By the inductive hypothesis, the map
\[f_m:s\to t\]
is a homeomorphism. Hence, the point 
\[f_m(E_1)=\Pbbb\bigg(\sum_{i=1}^k M_i^{(n_i)}+\sub{E_0}{1}+\sub{E_1}{m}\bigg)\cap\Pbbb(\sub{A}{1}+\sub{B}{1})\] 
lies on $t$, strictly between the points 
\[f_m(E_0)=\Pbbb\bigg(\sum_{i=1}^k M_i^{(n_i)}+\sub{E_0}{m+1}\bigg)\cap\Pbbb(\sub{A}{1}+\sub{B}{1})\text{ and }f_m(B)=\sub{B}{1}.\] 
By the base case, the map
\[f_1:s\to t\]
is a homeomorphism. Thus, we can conclude that the point 
\[f_1(E_1)=\Pbbb\bigg(\sum_{i=1}^k M_i^{(n_i)}+\sub{E_1}{m+1}\bigg)\cap\Pbbb(\sub{A}{1}+\sub{B}{1})\]
lies on $t$, strictly between the points 
\[f_1(E_0)=\Pbbb\bigg(\sum_{i=1}^k M_i^{(n_i)}+\sub{E_1}{m}+\sub{E_0}{1}\bigg)\cap\Pbbb(\sub{A}{1}+\sub{B}{1})\text{ and }f_1(B)=\sub{B}{1}.\] 
Since $f_m(E_1)=f_1(E_0)$, $f_{m+1}(E_0)=f_m(E_0)$ and $f_{m+1}(E_1)=f_1(E_1)$, we see in particular that $f_{m+1}(E_0)\neq f_{m+1}(E_1)$, so $f_{m+1}$ is injective. It is clear that $f_{m+1}(A)=\sub{A}{1}$ and $f_{m+1}(B)=\sub{B}{1}$, so the continuity of $f_{m+1}$ implies that it is surjective. This finishes the inductive step.
\end{proof}

The homeomorphisms $f_{n_{k+1}}$ should be thought of as projections of subsegments of $\xi$ (or all of $\xi$ in the case when $n_{k+1}=1$) onto the projective line $\Pbbb(\sub{A}{1}+\sub{B}{1})$ via the ``base" $(n-1-n_{k+1})$-dimensional subspace $\displaystyle\sum_{i=1}^kM_i^{(n_i)}$ of $\Rbbb^n$.

%%%%%%%%%%%%%%%%%%%%%%%%%%%%%%%%%%%%%%%%%%%%%%%%%%
\subsection{Cartan and Jordan projections}\label{Cartan and Jordan projections}
%%%%%%%%%%%%%%%%%%%%%%%%%%%%%%%%%%%%%%%%%%%%%%%%%%

The Cartan and Jordan projections are well-known maps which can be defined for any reductive Lie group of noncompact type. The former is useful for capturing translation distances in the corresponding symmetric space while the latter captures eigenvalues. We will now describe a special case of these two projections in the case when the Lie group is $PSL(n,\Rbbb)$, starting with the Cartan projection. For more details, see Benoist \cite{Ben1} and Chapter 2 of Eberlein \cite{Ebe1}.

Let $\amf^+$ be the set of traceless $n\times n$ diagonal matrices where the diagonal entries are (strictly) decreasing down the diagonal. This is a choice of a positive Weyl chamber in $\smf\lmf(n,\Rbbb)$ that we make once and for all in this paper. Then let $\overline{\amf^+}$ be the closure of $\amf^+$ in $\smf\lmf(n,\Rbbb)$. It is well-known that for any $g\in PSL(n,\Rbbb)$, there is a unique element $\mu(g)\in\overline{\amf^+}$ so that $g=k\cdot\exp(\mu(g))\cdot l$ for some $k,l$ in $PSO(n)$. The \emph{Cartan projection} (corresponding to $PSO(n)$) is then the map $\mu:PSL(n,\Rbbb)\to\overline{\amf^+}$ defined by $\mu:g\mapsto \mu(g)$. An easily verified but important property of this Cartan projection is the following: if $o$ is the point in $M$, the symmetric space of $PSL(n,\Rbbb)$, whose stabilizer is $PSO(n)$, then for any $g$ in $PSL(n,\Rbbb)$,
\begin{equation*}
d_M(o,g\cdot o)=||\mu(g)||:=c_n\sqrt{\sum_{i=1}^n\mu_i(g)^2}
\end{equation*}
where $\mu_i(g)$ is the $(i,i)$-th entry of $\mu(g)$, $c_n$ is a constant depending only on $n$, and $d_M$ is the distance in $M$ equipped with the Riemannian metric. 

The Jordan projection, sometimes also called the Lyapunov projection, is similar, except that we use the Jordan decomposition in place of the Cartan decomposition. The Jordan decomposition theorem states that for any $g\in PSL(n,\Rbbb)$, there are unique hyperbolic, elliptic and unipotent elements $g_h,g_e,g_u\in PSL(n,\Rbbb)$ respectively, so that $g=g_h\cdot g_e\cdot g_u$. Furthermore, the conjugacy class of $g_h$ intersects $\exp(\overline{\amf^+})$ at a unique point, which we denote by $\exp(\lambda(g))$. Since $\exp$ is injective when restricted to $\overline{\mathfrak{a}^+}$, this allows us to define the \emph{Jordan projection} $\lambda:PSL(n,\Rbbb)\to\overline{\amf^+}$ by $\lambda:g\mapsto\lambda(g)$. In the case when $g$ is a diagonalizable matrix, $\lambda(g)$ is then the matrix whose diagonal entries are the logarithms of the absolute values of the eigenvalues of $g$, listed in decreasing order down the diagonal. We will also denote the $(i,i)$-th entry of $\lambda(g)$ by $\lambda_i(g)$.

%%%%%%%%%%%%%%%%%%%%%%%%%%%%%%%%%%%%%%%%%%%%%%%%%%
\subsection{Cross ratio}\label{Cross ratio}
%%%%%%%%%%%%%%%%%%%%%%%%%%%%%%%%%%%%%%%%%%%%%%%%%%

We will now describe the main tool used to prove the main theorem (Theorem \ref{main theorem}) and establish some of its properties. Similar ideas were used to study surface group representations by Labourie \cite{Lab3}, Fock-Goncharov \cite{FocGon1}, Bonahon-Dreyer \cite{BonDre1}, Burger-Iozzi-Wienhard \cite{BurIozWie1} and many others. Here is its definition.

\begin{definition}\label{cross ratio definition}
Let $L_1=[l_1],\dots,L_4=[l_4]$ be four lines in $\Rbbb^n$ through the origin, and let $M=\Span\{m_1,\dots,m_{n-2}\}$ be a $(n-2)$-dimensional subspace of $\Rbbb^n$ not containing $L_i$ for any $i=1,\dots,4$, so that no three of the four $(n-1)$-dimensional subspaces $M+L_i$ agree. Define the \emph{cross ratio} of the lines $L_1,L_2,L_3,L_4$ based at $M$ by
\[(L_1,L_2,L_3,L_4)_M:=\frac{m_1\wedge\dots\wedge m_{n-2}\wedge l_1\wedge l_3\cdot m_1\wedge\dots\wedge m_{n-2}\wedge l_4\wedge l_2}{m_1\wedge\dots\wedge m_{n-2}\wedge l_1\wedge l_2\cdot m_1\wedge\dots\wedge m_{n-2}\wedge l_4\wedge l_3}.\]
\end{definition}

Here, we use the determinant map $\det:\bigwedge^n(\Rbbb^n)\to\Rbbb$ to evaluate the expression on the right as a number in the one point compactification $\Rbbb\cup\{\infty\}$ of $\Rbbb$. It is easy to verify that the cross ratio depends neither on the choice of basis $\{m_1,\dots,m_{n-2}\}$ for $M$, nor the choice of representatives $l_i$ for $L_i$. 

If we let $c_i$ be a linear functional on $\Rbbb^n$ with kernel $M+L_i$ for $i=1,\dots,4$, then this cross ratio takes a more familiar guise 
\[(L_1,L_2,L_3,L_4)_M=(c_1,c_4;l_2,l_3)=\frac{c_1(l_3)\cdot c_4(l_2)}{c_1(l_2)\cdot c_4(l_3)},\]
which is a well-studied object in projective geometry. 

The next proposition summarizes some basic properties of this cross ratio.

\begin{prop}\label{basic cross ratio}
Let $L_1,\dots,L_5$ be pairwise distinct lines in $\Rbbb^n$ through the origin. Let $M$, $M'$ be $(n-2)$-dimensional subspaces of $\Rbbb^n$ not containing $L_i$ for any $i=1,\dots,5$, so that no three of the five $(n-1)$-dimensional subspaces $M+L_i$ agree and no three of the five $(n-1)$-dimensional subspaces $M'+L_i$ agree.
\begin{enumerate}
\item For any $g$ in $PSL(n,\Rbbb)$, $(g\cdot L_1,\dots, g\cdot L_4)_{g\cdot M}=(L_1,\dots L_4)_M$.
\item For all $i$, let $L_i'$ be a line in $\Rbbb^n$ such that $L_i'\subset M+L_i$ and $L_i'\not\subset M$. Then 
\[(L_1',L_2',L_3',L_4')_M=(L_1,L_2,L_3,L_4)_M.\]
\item Suppose $L_1,L_2,L_3,L_4$ lie in a plane. Then 
\[(L_1,L_2,L_3,L_4)_M=(L_1,L_2,L_3,L_4)_{M'}.\]
\item If $M+L_1$, $M+L_2$, $M+L_3$ are pairwise distinct, then 
\[(L_1,L_1,L_2,L_3)_M=(L_1,L_2,L_3,L_3)_M=\infty.\] 
\item If $M+L_1$, $M+L_2$, $M+L_3$ are pairwise distinct, then 
\[(L_1,L_2,L_2,L_3)_M=(L_1,L_2,L_3,L_1)_M=1.\]
\item $(L_1,L_2,L_3,L_4)_M=(L_4,L_3,L_2,L_1)_M$.
\item $(L_1,L_2,L_3,L_4)_M=1-(L_2,L_1,L_3,L_4)_M$.
\item $(L_1,L_2,L_3,L_5)_M\cdot (L_1,L_3,L_4,L_5)_M=(L_1,L_2,L_4,L_5)_M$.
\end{enumerate}
\end{prop}

\begin{proof}
It is clear that (1) holds because $g$ preserves the volume form. (2) follows from the observation that replacing any of the $l_i$ in Definition \ref{cross ratio definition} with a linear combination of $m_1,\dots,m_{n-2}$, $l_i$ so that the coefficient of $l_i$ is non-zero does not change the cross ratio. To prove (3), note that we can choose $g$ in $PSL(n,\Rbbb)$ so that $g\cdot M=M'$ and $g$ fixes $L_1$, $L_2$, $L_3$. Since $L_1$, $L_2$, $L_3$ and $L_4$ lie in a plane, $g$ also fixes $L_4$. The $PSL(n,\Rbbb)$-invariance of the cross ratio stated in (1) then proves (3). Parts (4), (5), (6) and (8) are immediate from the formula in Definition \ref{cross ratio definition}, and (7) can be checked via choosing a normalization and performing a simple computation.
\end{proof}

Part (2) of Proposition \ref{basic cross ratio} allows one to think of the cross ratio as a projective invariant associated to the four $(n-1)$-dimensional subspaces of $\Rbbb^n$ that intersect along a $(n-2)$-dimensional subspace. This is thus a natural generalization of the classical cross ratio of four lines intersecting at a point in $\Rbbb\Pbbb^2$.

\begin{notation}
In view of (3) of Proposition \ref{basic cross ratio}, we will denote $(L_1,L_2,L_3,L_4)_M$ by $(L_1,L_2,L_3,L_4)$ in the case when $L_1$, $L_2$, $L_3$, $L_4$ lie in the same plane.
\end{notation}

There will be two main ways we use the cross ratio. The first is a well-known method to capture the eigenvalue data of $g$ in $PSL(n,\Rbbb)$. This is described in the next proposition, whose proof is a simple computation which we omit. 

\begin{prop}\label{cross ratio and length}
Let $g$ in $PSL(n,\Rbbb)$ be diagonalizable. For any $i< j$, let $V_i$ and $V_j$ be the eigenspaces corresponding to $e^{\lambda_i(g)}$ and $e^{\lambda_j(g)}$ respectively, and let $M$ be an $n-2$-dimensional subspace that is invariant under $g$ and complementary to $V_i+V_j$. (Recall that $\lambda_i(g)$ is the $(i,i)$-th entry of $\lambda(g)$, the image of $g$ under the Jordan projection.) Then for any line $L$ in $\Rbbb^n$ through the origin such that $L\not\subset M+V_i$ and $L\not\subset M+V_j$, we have
\[(V_j,L,g\cdot L,V_i)_M=e^{\lambda_i(g)-\lambda_j(g)}.\]
\end{prop}

Given three pairwise distinct $(n-1)$-dimensional subspaces in $\Rbbb^n$ that intersect along a common $(n-2)$-dimensional subspace $M$, we can also use this cross ratio to parameterize the set of $(n-1)$-dimensional subspaces in $\Rbbb^n$ that contain $M$. More precisely, we have the following standard proposition, whose proof we also omit.

\begin{prop}\label{cross ratio configuration}
Let $M$ be any $(n-2)$-dimensional subspace of $\Rbbb^n$, and let $N_1$, $N_2$, $N_3$ be pairwise distinct $(n-1)$-dimensional subspaces in $\Rbbb^n$ that contain $M$. For $i=1,2,3$, let $L_i$ be a line through the origin in $N_i$ that does not lie in $M$. Denote the space of $(n-1)$-dimensional subspaces of $\Rbbb^n$ containing $M$ by $\Spc$, and for any $N$ in $\Spc$, let $L_N$ be any line through the origin in $N$ but not in $M$. Then the map
\[f:\Spc\to\Rbbb\cup\{\infty\}\]
given by 
\[f(N)=(L_1,L_N,L_2,L_3)_M\]
is a homeomorphism. Moreover, $f(N_1)=\infty$, $f(N_2)=1$ and $f(N_3)=0$.
\end{prop}

Next, we discuss how the cross ratio interacts with a Frenet curve. For that purpose, we introduce the following notation.

\begin{notation}\label{cross ratio notation}
Let $A,B,C,D$ be pairwise distinct points along a Frenet curve $\xi$. Let $M_1,\dots, M_k$ be another set of pairwise distinct points along $\xi$ and let $n_1,\dots n_k$ be positive integers such that $\displaystyle\sum_{i=1}^k n_i=n-2$. Let $\displaystyle M:={\sum_{i=1}^kM_i^{(n_i)}}$, and denote
\[(A,B,C,D)_M:=(\sub{A}{1},\sub{B}{1},\sub{C}{1},\sub{D}{1})_M,\]
where $\sub{A}{1}$, $\sub{B}{1}$, $\sub{C}{1}$, $\sub{D}{1}$ are defined as in Notation \ref{moving subspaces}.
\end{notation}

By (2) of Proposition \ref{basic cross ratio}, the cross ratio $(A,B,C,D)_M$ is independent of the choices (if any) made to define $\sub{A}{1}$, $\sub{B}{1}$, $\sub{C}{1}$ or $\sub{D}{1}$. Using this notation, we can state the following proposition, which is a collection of useful inequalities involving the cross ratio and the Frenet curve.

\begin{prop}\label{useful cross ratio inequalities}
Let $A,U,B,C,V,D$ be pairwise distinct points along $\xi$, in that order. Let $M_1,\dots,M_k$ be another collection of pairwise distinct points along $\xi$, let $n_1,\dots,n_k$ be positive integers such that $\displaystyle\sum_{i=1}^k n_i=n-2$ and let $\displaystyle M:=\sum_{i=1}^kM_i^{(n_i)}$. Then the following inequalities hold:
\begin{enumerate}
\item $(A,B,C,D)_M>1$.
\item $(A,B,C,D)_M<(U,B,C,D)_M$
\item $(A,B,C,D)_M<(A,U,C,D)_M$
\item $(A,B,C,D)_M<(A,B,V,D)_M$
\item $(A,B,C,D)_M<(A,B,C,V)_M$
\end{enumerate}
\end{prop}

By (6) of Proposition \ref{basic cross ratio}, it does not matter if, in the above proposition, $A,U,B,C,V,D$ lie in clockwise or anti-clockwise order along $\xi$.

\begin{proof}
Part (1) follows from (1) of Lemma \ref{important lemma} and Proposition \ref{cross ratio configuration}. Since the proofs for (2) to (5) are very similar, we will only show the proof for (2).

Proof of (2). Consider the lines
\begin{eqnarray*}
L_1&:=&(\sub{U}{1}+M)\cap(\sub{B}{1}+\sub{C}{1})\\
L_2&:=&(\sub{A}{1}+M)\cap(\sub{B}{1}+\sub{C}{1})\\
L_3&:=&(\sub{D}{1}+M)\cap(\sub{B}{1}+\sub{C}{1})
\end{eqnarray*}
By (1) of Lemma \ref{important lemma}, we can choose vectors $l_1,l_2,l_3,l_B,l_C$ in $\Rbbb^n$ such that $[l_i]=L_i$ for $i=1,2,3$, $[l_B]=\sub{B}{1}$, $[l_C]=\sub{C}{1}$ and $l_1=a\cdot l_B+(1-a)\cdot l_C$, $l_2=b\cdot l_B+(1-b)\cdot l_C$, $l_3=c\cdot l_B+(1-c)\cdot l_C$ for $0<c<b<a<1$. Then
\begin{eqnarray*}
(A,B,C,D)_M&=&(L_2,\sub{B}{1},\sub{C}{1},L_3)_M\\
&=&\frac{(1-c)b}{(1-b)c}\\
&<&\frac{(1-c)a}{(1-a)c}\\
&=&(L_1,\sub{B}{1},\sub{C}{1},L_3)_M\\
&=&(U,B,C,D)_M.
\end{eqnarray*}
\end{proof}

%%%%%%%%%%%%%%%%%%%%%%%%%%%%%%%%%%%%%%%%%%%%%%%%%%
\subsection{Triple ratio}\label{Triple ratio}
%%%%%%%%%%%%%%%%%%%%%%%%%%%%%%%%%%%%%%%%%%%%%%%%%%

Aside from the cross ratio, there is another less well-known collection of projective invariants, called \emph{triple ratios}, that one can associate to any triple of generic flags. These were used by Fock-Goncharov and Bonahon-Dreyer to parameterize the Hitchin component. Understanding what these triple ratios mean geometrically is crucial to our proof, so we will devote this section to describe the triple ratios. Before we do so, we need the following definition.

\begin{definition}
A triple of flags $(F,G,H)$ in $\Fpc(\Rbbb^n)^3$ is \emph{generic} if for any triple of non-negative integers $(a,b,c)$ with $a+b+c=n$, we have that $F^{(a)}+G^{(b)}+H^{(c)}=\Rbbb^n$.
\end{definition}

With this, we can define the triple ratios associated to a generic triple of flags.

\begin{notation}\label{Apc}
Let $\Apc:=\{(x,y,z)\in(\Zbbb^+)^3:x+y+z=n\}.$
\end{notation}

\begin{definition}\label{triple ratio}
Let $(F,G,H)$ be a generic triple of flags in $\Fpc(\Rbbb^n)^3$. Then for any $(x,y,z)$ in $\Apc$, define the \emph{triple ratio} to be the quantity
\begin{align*}
&T_{x,y,z}(F,G,H):=\frac{F^{(x)}\wedge G^{(y-1)}\wedge H^{(z+1)}}{F^{(x)}\wedge G^{(y+1)}\wedge H^{(z-1)}}\cdot\frac{F^{(x+1)}\wedge G^{(y)}\wedge H^{(z-1)}}{F^{(x-1)}\wedge G^{(y)}\wedge H^{(z+1)}}\cdot\\
&\hspace{8cm}\frac{F^{(x-1)}\wedge G^{(y+1)}\wedge H^{(z)}}{F^{(x+1)}\wedge G^{(y-1)}\wedge H^{(z)}}.
\end{align*}
\end{definition}

We will now explain the notation used in the formula above. For any triple of flags $(F, G, H)$ in $\Fpc(\Rbbb^n)^3$, choose three bases $\{f_1,\dots,f_n\}$, $\{g_1,\dots,g_n\}$, $\{h_1,\dots,h_n\}$ of $\Rbbb^n$ so that for all $k=1,\dots,n$,
\[F^{(k)}=\Span\{f_1,\dots,f_k\},\] 
\[G^{(k)}=\Span\{g_1,\dots,g_k\},\] 
\[H^{(k)}=\Span\{h_1,\dots,h_k\}.\] 
Then for any triple $(x,y,z)$ in $(\Zbbb_{\geq 0})^3$ such that $x+y+z=n$, define
\[F^{(x)}\wedge G^{(y)}\wedge H^{(z)}:=f_1\wedge\dots\wedge f_x\wedge g_1\wedge\dots\wedge g_y\wedge h_1\wedge\dots\wedge h_z.\]
The determinant map $\det:\bigwedge^n(\Rbbb^n)\to\Rbbb$, again allows us to treat $F^{(x)}\wedge G^{(y)}\wedge H^{(z)}$ as a real number. One can check that $T_{x,y,z}(F,G,H)$ is well-defined and $PSL(n,\Rbbb)$-invariant.

From the definition of a triple ratio, one can immediately observe the next lemma.

\begin{lem}\label{symmetry}
Let $(F,G,H)$ be a triple of flags in $\Fpc(\Rbbb^n)^3$ and $(x,y,z)\in\Apc$. Then
\[T_{x,y,z}(F,G,H)=T_{y,z,x}(G,H,F)=T_{z,x,y}(H,F,G).\]
\end{lem}

Fock-Goncharov proved that the collection of triple ratios 
\[\{T_{x,y,z}(F,G,H):(x,y,z)\in\Apc\}\] 
determine the generic triple of flags $(F,G,H)$ up to the action of $PSL(n,\Rbbb)$. (Section 9.7 and 9.8 of Fock-Goncharov \cite{FocGon1}). In fact, we have the following lemma.

\begin{lem}\label{triangle lemma}
Let $(F,G,H)$ and $(F',G',H')$ be generic triples in $\Fpc(\Rbbb^n)^3$ so that $F=F'$, $H=H'$ and $G^{(i)}=G'^{(i)}$ for all $i=1,\dots,y_0$, where $y_0=1,\dots,n-2$. If 
\[T_{x,y_0,z}(F,G,H)=T_{x,y_0,z}(F',G',H')\]
for all $(x,y_0,z)$ in $\Apc$, then
then $G^{(y_0+1)}=G'^{(y_0+1)}$.
\end{lem}

\begin{proof}
Let $\{f_1,\dots,f_n\}$, $\{g_1,\dots,g_n\}$, $\{h_1,\dots,h_n\}$ and $\{g'_1,\dots,g'_n\}$ be bases for $\Rbbb^n$ such that for $k=1,\dots,n$, $F'^{(k)}=F^{(k)}=\Span\{f_1,\dots,f_k\}$, $H'^{(k)}=H^{(k)}=\Span\{h_1,\dots,h_k\}$, $G^{(k)}=\Span\{g_1,\dots,g_k\}$ and $G'^{(k)}=\Span\{g'_1,\dots,g'_k\}$. Since $G^{(i)}=G'^{(i)}$ for all $i=1,\dots,y_0$, we can assume without loss of generality that $g'_i=g_i$ for all $i=1,\dots,y_0$. Also, the genericity of the triple $(F,G,H)$ implies that 
\[\{f_1,\dots,f_x,g_1,\dots,g_{y_0},h_1,\dots,h_z\}\] 
is a basis for $\Rbbb^n$ for any triple $(x,y_0,z)$ in $\Apc$. Hence, for any such triple $(x,y_0,z)$ in $\Apc$, we can write
\[f_{x+1}:=\sum_{i=1}^x\alpha_if_i+\sum_{i=1}^{y_0}\alpha_{x+i}g_i+\sum_{i=1}^z\alpha_{x+y_0+i}h_i,\] 
\[g_{y_0+1}:=\sum_{i=1}^x\beta_if_i+\sum_{i=1}^{y_0}\beta_{x+i}g_i+\sum_{i=1}^z\beta_{x+y_0+i}h_i,\] 
\[g'_{y_0+1}:=\sum_{i=1}^x\beta'_if_i+\sum_{i=1}^{y_0}\beta'_{x+i}g_i+\sum_{i=1}^z\beta'_{x+y_0+i}h_i,\] 
\[h_{z+1}:=\sum_{i=1}^x\gamma_if_i+\sum_{i=1}^{y_0}\gamma_{x+i}g_i+\sum_{i=1}^z\gamma_{x+y_0+i}h_i,\] 
for some real numbers $\alpha_i$, $\beta_i$, $\beta'_i$, $\gamma_i$. 

By an easy computation, we have that 
\begin{eqnarray*}
\frac{\bigwedge_{i=1}^{x+1}f_i\wedge\bigwedge_{i=1}^{y_0} g_i\wedge\bigwedge_{i=1}^{z-1} h_i}{\bigwedge_{i=1}^{x+1} f_i\wedge\bigwedge_{i=1}^{y_0-1} g_i\wedge\bigwedge_{i=1}^z h_i}&=&(-1)^z\cdot\frac{\alpha_n}{\alpha_{x+y_0}},\\
\frac{\bigwedge_{i=1}^{x-1} f_i\wedge\bigwedge_{i=1}^{y_0+1} g_i\wedge\bigwedge_{i=1}^z h_i}{\bigwedge_{i=1}^xf_i\wedge\bigwedge_{i=1}^{y_0+1} g_i\wedge\bigwedge_{i=1}^{z-1} h_i}&=&(-1)^{y_0-z+1}\cdot\frac{\beta_x}{\beta_n},\\
\frac{\bigwedge_{i=1}^{x-1} f_i\wedge\bigwedge_{i=1}^{y_0+1} g'_i\wedge\bigwedge_{i=1}^z h_i}{\bigwedge_{i=1}^xf_i\wedge\bigwedge_{i=1}^{y_0+1} g'_i\wedge\bigwedge_{i=1}^{z-1} h_i}&=&(-1)^{y_0-z+1}\cdot\frac{\beta'_x}{\beta'_n},\\
\frac{\bigwedge_{i=1}^xf_i\wedge\bigwedge_{i=1}^{y_0-1} g_i\wedge\bigwedge_{i=1}^{z+1} h_i}{\bigwedge_{i=1}^{x-1} f_i\wedge\bigwedge_{i=1}^{y_0} g_i\wedge\bigwedge_{i=1}^{z+1} h_i}&=&(-1)^{-y_0}\cdot\frac{\gamma_{x+y_0}}{\gamma_x}.
\end{eqnarray*}
Combining the definition of the triple ratio with the computation above yields
\begin{eqnarray}
-\frac{\alpha_n}{\alpha_{x+y_0}}\cdot\frac{\beta'_x}{\beta'_n}\cdot\frac{\gamma_{x+y_0}}{\gamma_x}&=&T_{x,y_0,z}(F',G',H')\label{triple ratio equation}\\
&=&T_{x,y_0,z}(F,G,H)\nonumber\\
&=&-\frac{\alpha_n}{\alpha_{x+y_0}}\cdot\frac{\beta_x}{\beta_n}\cdot\frac{\gamma_{x+y_0}}{\gamma_x}.\nonumber
\end{eqnarray}
The assumption that $(F,G,H)$ is a generic triple implies that $T_{x,y_0,z}(F,G,H)$ is a non-zero real number, so $\displaystyle\frac{\beta'_x}{\beta'_n}=\frac{\beta_x}{\beta_n}$. Since the $(n-1)$-dimensional subspaces $F^{(x-1)}+G^{(y_0+1)}+H^{(z-1)}$ and $F^{(x-1)}+G'^{(y_0+1)}+H^{(z-1)}$ are spanned by 
\[\{f_1,\dots,f_{x-1},g_1,\dots,g_{y_0},h_1,\dots,h_{z-1},\beta_xf_x+\beta_nh_z\}\]
and
\[\{f_1,\dots,f_{x-1},g_1,\dots,g_{y_0},h_1,\dots,h_{z-1},\beta'_xf_x+\beta'_nh_z\}\]
respectively, we see that $F^{(x-1)}+G^{(y_0+1)}+H^{(z-1)}=F^{(x-1)}+G'^{(y_0+1)}+H^{(z-1)}$ for all $(x,y_0,z)$ in $\Apc$. The genericity of the triples $(F,G,H)$ and $(F',G',H')$ then imply that
\[G^{(y_0+1)}=\bigcap_{(x,y_0,z)\in\Apc}(F^{(x-1)}+G^{(y_0+1)}+H^{(z-1)})\]
and
\[G'^{(y_0+1)}=\bigcap_{(x,y_0,z)\in\Apc}(F^{(x-1)}+G'^{(y_0+1)}+H^{(z-1)}),\]
so the lemma holds.
\end{proof}

Moreover, as a consequence of Lemma \ref{triangle lemma}, we can make the following observation, which we record as Proposition \ref{triple ratio escaping 1}. 

\begin{prop}\label{triple ratio escaping 1}
Let $\{(F_i,G_i,H_i)\}_{i=1}^\infty$ be a sequence of generic triples of flags in $\Fpc(\Rbbb^n)$ such that for all positive integers $i,j$, $F_i=F_j$, $H_i=H_j$ and  $G_i^{(1)}=G_j^{(1)}$. Suppose that there is some $y_0=1,\dots,n-2$ so that for any $(x,y,z)$ in $\Apc$ with $y<y_0$, $\displaystyle\lim_{i\to\infty}T_{x,y,z}(F_i,G_i,H_i)$ is a non-zero real number. Then for any integers $x_0$, $z_0$ such that $(x_0,y_0,z_0)$ is in $\Apc$,
\begin{enumerate}
\item $\displaystyle\lim_{i\to\infty}T_{x_0,y_0,z_0}(F_i,G_i,H_i)=\infty$ if and only if
\[\lim_{i\to\infty}F_i^{(x_0-1)}+G_i^{(y_0+1)}+H_i^{(z_0-1)}=\lim_{i\to\infty}F_i^{(x_0)}+G_i^{(y_0)}+H_i^{(z_0-1)}.\]
\item $\displaystyle\lim_{i\to\infty}T_{x_0,y_0,z_0}(F_i,G_i,H_i)=0$ if and only if
\[\lim_{i\to\infty}F_i^{(x_0-1)}+G_i^{(y_0+1)}+H_i^{(z_0-1)}=\lim_{i\to\infty}F_i^{(x_0-1)}+G_i^{(y_0)}+H_i^{(z_0)}.\]
\end{enumerate}
\end{prop}

\begin{proof}
Let $F^{(k)}:=F_i^{(k)}$ and $H^{(k)}:=H_i^{(k)}$ for all $k=1,\dots,n$. Also, for all positive integers $i$, let $G$ be the flag in $\Fpc(\Rbbb^n)$ so that the following hold: 
\begin{itemize}
\item $(F,G,H)$ is a generic triple of flags
\item $G^{(1)}=G_i^{(1)}$ for all positive integers $i$,
\item $\displaystyle T_{x,y,z}(F,G,H)=\lim_{i\to\infty}T_{x,y,z}(F_i,G_i,H_i)$
for all $(x,y,z)$ in $\Apc$ with $y<y_0$,
\item $T_{x,y,z}(F,G,H)=1$
for all $(x,y,z)$ in $\Apc$ with $y\geq y_0$.
\end{itemize}
The flag $G$ exist by Section 9.7 and 9.8 of Fock-Goncharov \cite{FocGon1} (alternatively, see Lemma 2.3.7 of \cite{Zha2}). Similarly, for any positive integer $i$, let $G'_i$ be the flag in $\Fpc(\Rbbb^n)$ so that the following hold: 
\begin{itemize}
\item $(F_i,G'_i,H_i)$ is a generic triple of flags,
\item ${G'_i}^{(1)}=G^{(1)}$ for all positive integers $i$,
\item $\displaystyle T_{x,y,z}(F_i,G'_i,H_i)=T_{x,y,z}(F_i,G_i,H_i)$
for all $(x,y,z)$ in $\Apc$ with $y<y_0$,
\item $T_{x,y,z}(F_i,G'_i,H_i)=1$
for all $(x,y,z)$ in $\Apc$ with $y\geq y_0$.
\end{itemize}

By Lemma \ref{triangle lemma}, we know that $G_i'^{(y)}=G_i^{(y)}$ for $y=1,\dots, y_0$, so the continuity of the triple ratio ensures that
\[G^{(y)}=\lim_{i\to\infty}G_i^{(y)}.\]
As in the proof of Lemma \ref{triangle lemma}, choose a basis
\[\{f_1,\dots,f_{x_0},g_1,\dots,g_{y_0},h_1,\dots,h_{z_0}\}\]
for $\Rbbb^n$ so that 
\begin{align*}
F^{(k)}&=\Span\{f_1,\dots,f_k\}\text{ for all }k=1,\dots,x_0,\\ 
G^{(k)}&=\Span\{g_1,\dots,g_k\}\text{ for all }k=1,\dots,y_0,\\ 
H^{(k)}&=\Span\{h_1,\dots,h_k\}\text{ for all }k=1,\dots,z_0.
\end{align*} 
Also, let $g_{i,1},\dots,g_{i,y_0}$ be vectors in $\Rbbb^n$ so that $G_i^{(k)}=\Span\{g_{i,1},\dots,g_{i,k}\}$ for all $k=1,\dots, y_0$ and $\displaystyle\lim_{i\to\infty}g_{i,j}=g_j$ for $j=1,\dots,y_0$. Observe that for all $i$, 
\[\{f_1,\dots,f_{x_0},g_{i,1},\dots,g_{i,y_0},h_1,\dots,h_{z_0}\}\]
is also a basis for $\Rbbb^n$. 

Let $f_{x_0+1}$, $g_{i,y_0+1}$, $h_{z_0+1}$ be vectors so that $F^{(x_0+1)}=\Span\{f_1,\dots,f_{x_0+1}\}$, $G_i^{(y_0+1)}=\Span\{g_{i,1},\dots,g_{i,y_0+1}\}$, $H^{(z_0+1)}=\Span\{h_1,\dots,h_{z_0+1}\}$. For all positive integers $i$ and for all $j=1,\dots,n$, let $\alpha_j$, $\gamma_j$ be real numbers so that
\[f_{x_0+1}:=\sum_{j=1}^{x_0}\alpha_jf_j+\sum_{j=1}^{y_0}\alpha_{x_0+j}g_j+\sum_{j=1}^{z_0}\alpha_{x_0+y_0+j}h_j\] 
\[h_{z_0+1}:=\sum_{j=1}^{x_0}\gamma_jf_j+\sum_{j=1}^{y_0}\gamma_{x_0+j}g_j+\sum_{j=1}^{z_0}\gamma_{x_0+y_0+j}h_j.\]
Also, let $\alpha_{i,j}$, $\beta_{i,j}$, $\gamma_{i,j}$ be real numbers so that
\[f_{x_0+1}:=\sum_{j=1}^{x_0}\alpha_{i,j}f_j+\sum_{j=1}^{y_0}\alpha_{i,x_0+j}g_{i,j}+\sum_{j=1}^{z_0}\alpha_{i,x_0+y_0+j}h_j\] 
\[g_{i,y_0+1}:=\sum_{j=1}^{x_0}\beta_{i,j}f_j+\sum_{j=1}^{y_0}\beta_{i,x_0+j}g_{i,j}+\sum_{j=1}^{z_0}\beta_{i,x_0+y_0+j}h_j\] 
\[h_{z_0+1}:=\sum_{j=1}^{x_0}\gamma_{i,j}f_j+\sum_{j=1}^{y_0}\gamma_{i,x_0+j}g_{i,j}+\sum_{j=1}^{z_0}\gamma_{i,x_0+y_0+j}h_{i,j},\]

Using Equation (\ref{triple ratio equation}), we have that 
\begin{equation*}
T_{x_0,y_0,z_0}(F_i,G_i,H_i)=-\frac{\alpha_{i,n}}{\alpha_{i,x_0+y_0}}\cdot\frac{\beta_{i,x_0}}{\beta_{i,n}}\cdot\frac{\gamma_{i,x_0+y_0}}{\gamma_{i,x_0}}.
\end{equation*}
Since $\displaystyle\lim_{i\to\infty} g_{i,j}=g_j$ for $j=1,\dots,y_0$, we see that $\displaystyle\lim_{i\to\infty}\alpha_{i,j}=\alpha_j$ and $\displaystyle\lim_{i\to\infty}\gamma_{i,j}=\gamma_j$ for all $j=1,\dots,n$. Note that $\alpha_n$, $\alpha_{x_0+y_0}$, $\gamma_{x_0+y_0}$ and $\gamma_{x_0}$ are all non-zero real numbers, so $\displaystyle\lim_{i\to\infty}T_{x_0,y_0,z_0}(F_i,G_i,H_i)=\infty$ if and only if $\displaystyle\lim_{i\to\infty}\bigg|\frac{\beta_{i,x_0}}{\beta_{i,n}}\bigg|=\infty$ and $\displaystyle\lim_{i\to\infty}T_{x_0,y_0,z_0}(F_i,G_i,H_i)=0$ if and only if $\displaystyle\lim_{i\to\infty}\bigg|\frac{\beta_{i,x_0}}{\beta_{i,n}}\bigg|=0$. This implies the proposition.
\end{proof}

%\begin{prop}\label{triple ratio escaping 2}
%Let $\{(F_i,G_i,H_i)\}_{i=1}^\infty$ be a sequence of generic triples of flags in $\Fpc(\Rbbb^n)$ such that 
%for all positive integers $i,j$, $F_i=F_j$, $G_i=G_j$ and $H_i^{(1)}=H_j^{(1)}$. Suppose that there is some $z_0=1,\dots,n-2$ so that for any $(x,y,z)$ in $\Apc$ with $z<z_0$, $\displaystyle\lim_{i\to\infty}T_{x,y,z}(F_i,G_i,H_i)$ is a non-zero real number. Then for any $x_0$, $y_0$ such that $(x_0,y_0,z_0)$ is in $\Apc$,
%\begin{enumerate}
%\item $\displaystyle\lim_{i\to\infty}T_{x_0,y_0,z_0}(F_i,G_i,H_i)=\infty$ if and only if
%\[\lim_{i\to\infty}F_i^{(x_0-1)}+G_i^{(y_0-1)}+H_i^{(z_0+1)}=\lim_{i\to\infty}F_i^{(x_0-1)}+G_i^{(y_0)}+H_i^{(z_0)}.\]
%\item $\displaystyle\lim_{i\to\infty}T_{x_0,y_0,z_0}(F_i,G_i,H_i)=0$ if and only if
%\[\lim_{i\to\infty}F_i^{(x_0-1)}+G_i^{(y_0-1)}+H_i^{(z_0+1)}=\lim_{i\to\infty}F_i^{(x_0)}+G_i^{(y_0-1)}+H_i^{(z_0)}.\]
%\end{enumerate}
%\end{prop}

%%%%%%%%%%%%%%%%%%%%%%%%%%%%%%%%%%%%%%%%%%%%%%%%%%%%%%%%%%%%%%%%%%%%%%%%%%%%%%%%%%%%%%%%%%%%%%%%%%%%
\section{Shear-triangle parameterization}\label{Shear-triangle parameterization}
%%%%%%%%%%%%%%%%%%%%%%%%%%%%%%%%%%%%%%%%%%%%%%%%%%%%%%%%%%%%%%%%%%%%%%%%%%%%%%%%%%%%%%%%%%%%%%%%%%%%

In the first parts of this section, we will briefly describe a particular case of what we call the shear-triangle parameterization of $Hit_n(S)$ given by Bonahon-Dreyer \cite{BonDre1}. A version of this parameterization can also be found in the monumental work of Fock-Goncharov \cite{FocGon1}, though in a much less explicit form. We will also give a geometric interpretation of the parameters in terms of flags. After that, we slightly modify this parameterization to obtain a parameterization of $Hit_n(S)$ that is more explicitly analogous to the Fenchel-Nielsen coordinates on $Hit_2(S)$ and the Goldman parameterization \cite{Gol1} on $Hit_3(S)$. 

%%%%%%%%%%%%%%%%%%%%%%%%%%%%%%%%%%%%%%%%%%%%%%%%%%
\subsection{Ideal triangulations}\label{Ideal triangulations}
%%%%%%%%%%%%%%%%%%%%%%%%%%%%%%%%%%%%%%%%%%%%%%%%%%

One main ingredient needed to describe the shear-triangle parameterization is an ideal triangulation of the surface $S$. Here, we will give a description of ideal triangulations of $S$ in terms of $\partial_\infty\Gamma$.

\begin{notation}
Denote by $\partial_\infty\Gamma^{(2)}$ and $\partial_\infty\Gamma^{[2]}$ the set of ordered and unordered pairs of distinct points in $\partial_\infty\Gamma$ respectively.
\end{notation}

We say $\{a,b\}$ and $\{c,d\}$ in $\partial_\infty\Gamma^{[2]}$ \emph{intersect} if neither of the closed subsegments of $\partial_\infty\Gamma$ with endpoints $a$, $b$ contain both $c$ and $d$.

\begin{definition}
An (undirected) \emph{ideal triangulation} of the universal cover $\Std$ of $S$ is a maximal $\Gamma$-invariant subset $\widetilde{\Tpc}$ of $\partial_\infty\Gamma^{[2]}$ such that the following hold:
\begin{enumerate}
\item Any two pairs $\{a,b\}$, $\{c,d\}$ in $\widetilde{\Tpc}$ do not intersect.
\item For any $\{a,b\}$ in $\widetilde{\Tpc}$, either one of the following must hold:
\begin{itemize}
\item There is some $c$ in $\partial_\infty\Gamma$ such that $\{b,c\}$ and $\{c,a\}$ both lie in $\widetilde{\Tpc}$.
\item There is some $X$ in $\Gamma$ such that $\{a,b\}$ is the set of fixed points of $X$.
\end{itemize}
\end{enumerate}
An \emph{ideal triangulation} of $S$ is the quotient of an ideal triangulation of $\Std$ by $\Gamma$. If $\widetilde{\Tpc}$ is an ideal triangulation of $\Std$, then we denote by $\Tpc$ its quotient by $\Gamma$.
\end{definition}

If $\{a,b\}$ in $\widetilde{\Tpc}$ is the set of fixed points of some $X$ in $\Gamma$, then we call $\{a,b\}$ a \emph{closed leaf}. Also, $[a,b]$ in $\Tpc$ is called a \emph{closed leaf} if some (or equivalently, all) of its representatives in $\widetilde{\Tpc}$ are closed leaves. By a \emph{triangle} in $\widetilde{\Tpc}$, we mean a subset of $\widetilde{\Tpc}$ that is of the form $\{\{a,b\},\{b,c\},\{c,a\}\}$, where $a,b,c$ are points in $\partial_\infty\Gamma$. Each of the three pairs in any triangle is called an \emph{edge} of that triangle, and a point in any edge is called a \emph{vertex} of that edge. Also, we say that two triangles in $\widetilde{\Tpc}$ are \emph{adjacent} if they share a common edge. We will denote by $\widetilde{\Delta}=\widetilde{\Delta}_{\widetilde{\Tpc}}$ the set of triangles in $\widetilde{\Tpc}$. There is an obvious $\Gamma$-action on $\widetilde{\Delta}$, so we can consider
\[\Delta=\Delta_\Tpc:=\widetilde{\Delta}/\Gamma\]
and call any element in $\Delta$ a \emph{triangle} in $\Tpc$. An \emph{edge} of a triangle $T$ in $\Delta$ is an element $e$ in $\Tpc$ so that $T$ has a representative (in $\widetilde{\Delta}$) which has a representative (in $\widetilde{\Tpc}$) of $e$ as an edge. As before, we say two triangles in $\Delta$ are \emph{adjacent} if they share an edge, or equivalently, if they have adjacent representatives in $\widetilde{\Delta}$.

If we choose a marked complete hyperbolic structure on $S$, then there is a canonical $\Gamma$-equivariant identification of $\partial_\infty\Gamma$ with the boundary of the Poincar\'e disc, $\partial\Dbbb$. The ideal triangulation $\widetilde{\Tpc}$ then gives us an ideal triangulation of $\Dbbb$ (in the classical sense) by assigning to each pair $\{a,b\}$ in $\widetilde{\Tpc}\subset(\partial\Dbbb\times\partial\Dbbb)/\Zbbb_2$ the unique geodesic in $\Dbbb$ between $a$ and $b$. Moreover, this ideal triangulation is $\Gamma$-invariant, so $\Tpc$ can be thought of as an ideal triangulation (in the classical sense) of $S$ equipped with the marked hyperbolic structure. 

%%%%%%%%%%%%%%%%%%%%%%%%%%%%%%%%%%%%%%%%%%%%%%%%%%
\subsection{A special ideal triangulation}\label{A special ideal triangulation}
%%%%%%%%%%%%%%%%%%%%%%%%%%%%%%%%%%%%%%%%%%%%%%%%%%

In \cite{BonDre1}, Bonahon-Dreyer gave an explicit description of the shear-triangle parameterization of $Hit_n(S)$ for any closed surface $S$ using the Frenet curve $\xi$ guaranteed by Theorem \ref{Guichard, Labourie} and the notion of positive configurations of flags developed by Fock-Goncharov \cite{FocGon1}. To specify this parameterization, one needs to first choose an ideal triangulation of the surface $S$. For our purposes, we will only be considering this parameterization for a particular ideal triangulation, which we will now describe.

{\bf For the rest of this paper, fix a pants decomposition $\Ppc$ for $S$, i.e. a maximal, pairwise disjoint, pairwise non-homotopic collection of simple closed curves in $S$.} This pants decomposition cuts $S$ into finitely many pairs of pants, $P_1,\dots,P_{2g-2}$. For each of these pairs of pants $P_j$, consider primitive conjugacy classes $[A_j]$, $[B_j]$ and $[C_j]$ in $\Gamma$ that correspond to the three boundary components  of $P_j$, oriented so that $P_j$ lies on the left of each of these boundary components. For each $j$, consider $A_j\in[A_j]$, $B_j\in[B_j]$ and $C_j=A_j^{-1}B_j^{-1}\in[C_j]$. 

The action of any non-identity element $X$ in $\Gamma$ on $\partial_\infty\Gamma$ has a repelling and attracting fixed point. Let $a_j^-$, $b_j^-$, $c_j^-$ be the repelling fixed points and $a_j^+$, $b_j^+$, $c_j^+$ be the attracting fixed points of $A_j$, $B_j$, $C_j$ respectively. Let $\widetilde{\Qpc}_j$ and $\widetilde{\Ppc}_j$ be the subsets  of $\partial_\infty\Gamma^{[2]}$ defined by
\begin{align*}
&\widetilde{\Qpc}_j:=\bigcup_{X\in\Gamma}\{X\cdot\{b_j^-,a_j^-\}, X\cdot\{a_j^-,c_j^-\}, X\cdot\{c_j^-,b_j^-\}\},\\
&\widetilde{\Ppc}_j:=\bigcup_{X\in\Gamma}\{X\cdot\{a_j^-,a_j^+\},X\cdot\{b_j^-,b_j^+\},X\cdot\{c_j^-,c_j^+\}\},
\end{align*}
and let
\begin{align*}
&\widetilde{\Qpc}:=\bigcup_{j=1}^{2g-2}\widetilde{\Qpc}_j,\\
&\widetilde{\Ppc}:=\bigcup_{j=1}^{2g-2}\widetilde{\Ppc}_j.
\end{align*}
One can check that $\widetilde{\Qpc}$ and $\widetilde{\Ppc}$ are disjoint, and $\widetilde{\Qpc}\cup\widetilde{\Ppc}$ is an ideal triangulation of $\Std$. {\bf For the rest of this paper, we denote this particular ideal triangulation by $\widetilde{\Tpc}$ and let $\Tpc$ be the quotient of $\widetilde{\Tpc}$ by $\Gamma$.} Since $\widetilde{\Qpc}$ and $\widetilde{\Ppc}$ are $\Gamma$-invariant, we can define $\Ppc:=\widetilde{\Ppc}/\Gamma$ and $\Qpc:=\widetilde{\Qpc}/\Gamma$. It is easy to see that $\Ppc$ is the pants decomposition we chose for $S$, $\Tpc=\Ppc\cup\Qpc$ and $\Ppc$ is exactly the set of closed leaves in $\Tpc$. (See Figure \ref{Q_e} for a picture of $\Tpc$ restricted to a pair of pants given by $\Ppc$.)

We will now state some easily verified properties of $\Tpc$. Since we can realize $\Tpc$ as an ideal triangulation (in the classical sense) of $S$ by choosing a hyperbolic metric on $S$, the Gauss-Bonnet theorem tells us that the cardinalities of $\Delta$, $\Tpc$, $\Ppc$ and $\Qpc$ are $4g-4$, $9g-9$, $3g-3$ and $6g-6$ respectively. 

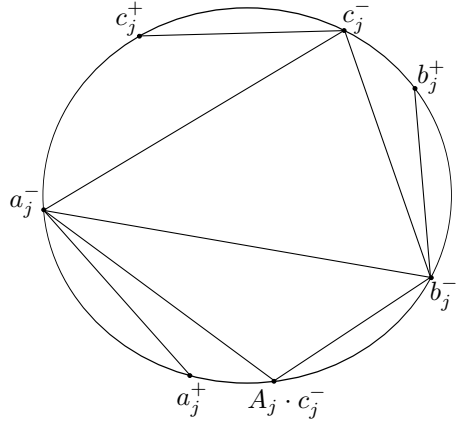
\begin{figure}
\includegraphics[scale=0.5]{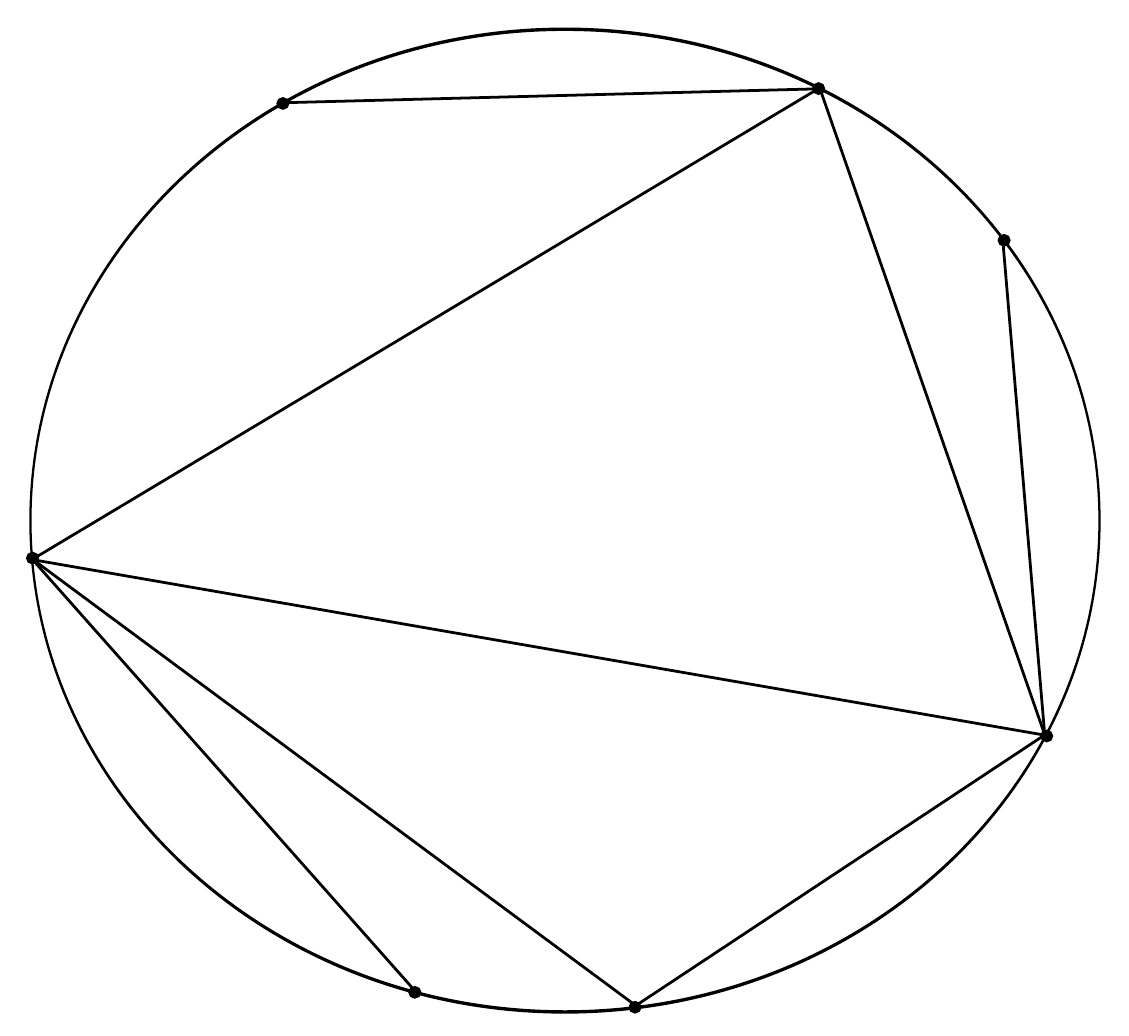}
\put (-369, 139){\makebox[0.7\textwidth][r]{$c_j^+$ }}
\put (-283, 140){\makebox[0.7\textwidth][r]{$c_j^-$ }}
\put (-255, 117){\makebox[0.7\textwidth][r]{$b_j^+$ }}
\put (-250, 35){\makebox[0.7\textwidth][r]{$b_j^-$ }}
\put (-300, -6){\makebox[0.7\textwidth][r]{$A_j\cdot c_j^-$ }}
\put (-345, -5){\makebox[0.7\textwidth][r]{$a_j^+$ }}
\put (-408, 70){\makebox[0.7\textwidth][r]{$a_j^-$ }}
\caption{$a_j^-$, $b_j^-$, $c_j^-$, $a_j^+$, $b_j^+$, $c_j^+$, $A_j\cdot c_j^-$ in $\partial_\infty\Gamma$.}
\label{standard}
\end{figure}

Define $\Qpc_j:=\widetilde{\Qpc}_j/\Gamma$ and observe that for any $j=1,\dots,2g-2$, $\Qpc_j$ has exactly three elements, which are the $\Gamma$ orbits of $\{b_j^-,a_j^-\}$, $\{a_j^-,c_j^-\}$ and $\{c_j^-,b_j^-\}$. Moreover, a triangle in $\Delta$ has an edge in $\Qpc_j$ if and only if all its edges lie in $\Qpc_j$. Furthermore, there are exactly two triangles in $\Delta$ with edges in $\Qpc_j$, and we can describe them explicitly. One of them, denoted $T_j$, is the $\Gamma$-orbit of the triangle $\{\{b_j^-,a_j^-\}, \{a_j^-,c_j^-\}, \{c_j^-,b_j^-\}\}$ and the other, denoted $T_j'$, is the $\Gamma$-orbit triangle $\{\{b_j^-,a_j^-\},\{a_j^-,A_j\cdot c_j^-\},\{A_j\cdot c_j^-,b_j^-\}\}$. (See Figure \ref{standard}.) The triangles $T_j$ and $T_j'$ share all their edges, and any adjacent pair of triangles in $\Delta$ is the pair $T_j$, $T_j'$ for some $j$. 

%%%%%%%%%%%%%%%%%%%%%%%%%%%%%%%%%%%%%%%%%%%%%%%%%%
\subsection{Triangle invariants}\label{Triangle invariants}
%%%%%%%%%%%%%%%%%%%%%%%%%%%%%%%%%%%%%%%%%%%%%%%%%%

In their description of the shear-triangle parameterization, Bonahon-Dreyer \cite{BonDre1} used two different kinds of parameters which they call the \emph{triangle invariants} and the \emph{shear invariants}. We will spend the next two subsections studying these invariants in detail. In this subsection, we describe the triangle invariants and what they mean geometrically. 

Choose any $\rho$ in $Hit_n(S)$ and let $\xi$ be the Frenet curve for $\rho$. For each $j=1,\dots,2g-2$, assign to each $(x,y,z)$ in $\Apc$ the real numbers $\tau_{(x,y,z),j}(\rho)$ and $\tau'_{(x,y,z),j}(\rho)$ given by the formulas
\begin{align*}
&\tau_{(x,y,z),j}(\rho):=\log\bigg(T_{x,z,y}(\xi(a_j^-),\xi(c_j^-),\xi(b_j^-))\bigg),\\
&\tau'_{(x,y,z),j}(\rho):=\log\bigg(T_{x,y,z}(\xi(a_j^-),\xi(b_j^-),\xi(A_j\cdot c_j^-))\bigg).
\end{align*}
The positivity condition on these flags described in Section 7.1 and 7.2 of Fock-Goncharov \cite{FocGon1} (alternatively, see Lemma 8.4.2 of Labourie-McShane \cite{LabMcS1} or Proposition 2.4.7 of \cite{Zha2}) ensures that 
\[T_{x,z,y}(\xi(a_j^-),\xi(c_j^-),\xi(b_j^-))\text{ and }T_{x,y,z}(\xi(a_j^-),\xi(b_j^-),\xi(A_j\cdot c_j^-))\] 
are positive, so $\tau'_{(x,y,z),j}(\rho)$ and $\tau_{(x,y,z),j}(\rho)$ are well-defined real numbers.

We can do this for every $\rho$ in $Hit_n(S)$, so $\tau'_{(x,y,z),j}$ and $\tau_{(x,y,z),j}$ can be viewed as functions 
\[\tau'_{(x,y,z),j},\tau_{(x,y,z),j}:Hit_n(S)\to\Rbbb.\]
Henceforth, we will call 
\[\Apc_j:=\{\tau_{(x,y,z),j}:(x,y,z)\in\Apc\}\text{ and }\Apc'_j:=\{\tau'_{(x,y,z),j}:(x,y,z)\in\Apc\}\] 
the set of \emph{triangle invariants for $T_j$ and $T_j'$} respectively. Also, we will call $\Apc_j\cup\Apc'_j$ the set of \emph{triangle invariants for $P_j$}. By Section 9.7 and 9.8 of Fock-Goncharov \cite{FocGon1} (see Lemma 2.3.7 of \cite{Zha2} for an alternative proof), $\Apc_j$ determines the triple of flags $(\xi(a_j^-), \xi(c_j^-), \xi(b_j^-))$ and $\Apc'_j$ determines the triple of flags $(\xi(a_j^-), \xi(b_j^-), \xi(A_j\cdot c_j^-))$ up to the action of $PSL(n,\Rbbb)$.

%%%%%%%%%%%%%%%%%%%%%%%%%%%%%%%%%%%%%%%%%%%%%%%%%%
\subsection{Shear invariants}\label{Shear invariants}
%%%%%%%%%%%%%%%%%%%%%%%%%%%%%%%%%%%%%%%%%%%%%%%%%%
The goal of this subsection is to describe the shear invariants, which are the second kind of invariants used in the shear-triangle parameterization. As before, let $\rho$ be a representation in $Hit_n(S)$ and let $\xi$ be the corresponding Frenet curve. We will use the set 
\[\Cpc:=\{(x,y,z)\in(\Zbbb_{\geq 0})^3:x+y+z=n\text{ and exactly one of }x,y,z\text{ is }0\}\]
to label the shear invariants. 

For each $j=1,\dots,2g-2$, assign to each vertex $(x,y,z)$ in $\Cpc$ the number $\sigma_{(x,y,z),j}(\rho)$ given by the formulas
\begin{align*}
&\sigma_{(x,y,0),j}(\rho):=\log\bigg(-(\xi(a_j^-),\xi(c_j^-),\xi(A_j\cdot c_j^-),\xi(b_j^-))_{\xi(a_j^-)^{(x-1)}+\xi(b_j^-)^{(y-1)}}\bigg),\\
&\sigma_{(x,0,z),j}(\rho):=\log\bigg(-(\xi(c_j^-),\xi(b_j^-),\xi(C_j\cdot b_j^-),\xi(a_j^-))_{\xi(c_j^-)^{(z-1)}+\xi(a_j^-)^{(x-1)}}\bigg),\\
&\sigma_{(0,y,z),j}(\rho):=\log\bigg(-(\xi(b_j^-),\xi(a_j^-),\xi(B_j\cdot a_j^-),\xi(c_j^-))_{\xi(b_j^-)^{(y-1)}+\xi(c_j^-)^{(z-1)}}\bigg),
\end{align*} 
if $z=0$, $y=0$ or $x=0$ respectively.
The cross ratio notation used in these formulas are as described in Notation \ref{cross ratio notation}. These invariants are called the \emph{shear invariants}. Just as in the case of the triangle invariants, Fock-Goncharov proved in Sections 7.1 and 7.2 of \cite{FocGon1} (alternatively, Lemma 8.4.2 of Labourie-McShane \cite{LabMcS1} or Proposition 2.4.7 of \cite{Zha2}) that 
\begin{align*}
&-(\xi(a_j^-),\xi(c_j^-),\xi(A_j\cdot c_j^-),\xi(b_j^-))_{\xi(a_j^-)^{(x-1)}+\xi(b_j^-)^{(y-1)}},\\
&-(\xi(c_j^-),\xi(b_j^-),\xi(C_j\cdot b_j^-),\xi(a_j^-))_{\xi(c_j^-)^{(z-1)}+\xi(a_j^-)^{(x-1)}},\\
&-(\xi(b_j^-),\xi(a_j^-),\xi(B_j\cdot a_j^-),\xi(c_j^-))_{\xi(b_j^-)^{(y-1)}+\xi(c_j^-)^{(z-1)}}
\end{align*}
are positive, so the shear invariants are well-defined. By allowing $\rho$ to vary over $Hit_n(S)$, we can view each shear invariant as a real valued function on $Hit_n(S)$.

We should think of the sets of invariants $\Cpc_{z,j}:=\{\sigma_{(x,y,0),j}:x+y=n\}$, $\Cpc_{y,j}:=\{\sigma_{(x,0,z),j}:x+z=n\}$ and $\Cpc_{x,j}:=\{\sigma_{(0,y,z),j}:y+z=n\}$ as being associated to the edges $[a_j^-,b_j^-]$, $[a_j^-,c_j^-]$, $[b_j^-,c_j^-]$ respectively. Geometrically, $\Cpc_{z,j}$ determines the pair of flags $\xi(a_j^-)$, $\xi(b_j^-)$ and the lines $\xi(c_j^-)^{(1)}$, $\xi(A_j\cdot c_j^-)^{(1)}$ up to $PSL(n,\Rbbb)$. To see this, choose the normalization so that 
\begin{itemize}
\item for all $k=1,\dots,n$, $\xi(a_j^-)^{(k)}=\Span\{e_1,\dots,e_k\}$, 
\item for all $k=1,\dots,n$, $\xi(b_j^-)^{(k)}=\Span\{e_n,\dots,e_{n-k+1}\}$ 
\item $\displaystyle\xi(c_j^-)^{(1)}=\Big[\sum_{i=1}^ne_i\Big]$,
\end{itemize} 
where $\{e_1,\dots,e_n\}$ is the standard basis of $\Rbbb^n$. We thus need to show that we can recover $\xi(A_j\cdot c_j^-)^{(1)}$ from $\Cpc_{z,j}$.

By Proposition \ref{cross ratio configuration}, the shear invariant $\sigma_{x,y,0}$ determines the $(n-1)$-dimensional subspace $\xi(a_j^-)^{(x-1)}+\xi(b_j^-)^{(y-1)}+\xi(A_j\cdot c_j^-)^{(1)}$ of $\Rbbb^n$. Thus, the $n-1$ elements in $\Cpc_{z,j}$ determine $n-1$ different $(n-1)$-dimensional subspaces of $\Rbbb^n$, all of which contain $\xi(A_j\cdot c_j^-)^{(1)}$. The Frenet property of $\xi$ ensures that the intersection of these $n-1$ subspaces is $1$-dimensional, so it has to be $\xi(A_j\cdot c_j^-)^{(1)}$. Similarly, $\Cpc_{y,j}$ determines the flags $\xi(c_j^-)$, $\xi(a_j^-)$ and the lines $\xi(b_j^-)^{(1)}$, $\xi(C_j\cdot b_j^-)^{(1)}$ up to $PSL(n,\Rbbb)$ while $\Cpc_{x,j}$ determines the flags $\xi(b_j^-)$, $\xi(c_j^-)$ and the lines $\xi(a_j^-)^{(1)}$, $\xi(B_j\cdot a_j^-)^{(1)}$ up to $PSL(n,\Rbbb)$.

%%%%%%%%%%%%%%%%%%%%%%%%%%%%%%%%%%%%%%%%%%%%%%%%%%
\subsection{Parameterizing the Hitchin component}\label{Parameterizing the Hitchin component}
%%%%%%%%%%%%%%%%%%%%%%%%%%%%%%%%%%%%%%%%%%%%%%%%%%

To obtain the shear-triangle parameterization, Bonahon-Dreyer found linear relations between the shear and triangle invariants,which we will now describe. 

As before, let $\rho$ be a representation in $Hit_n(S)$ and $\xi$ the corresponding Frenet curve. By a theorem of Labourie (Theorem 1.5 of \cite{Lab1}), we know that for every non-identity element $X$ in $\Gamma$, $\rho(X)$ is diagonalizable with real eigenvalues that have pairwise distinct norms. In particular, $\lambda(\rho(X))$, which is the image of $\rho(X)$ under the Jordan projection, lies in $\amf^+$. For any $k=1,\dots,n-1$ and any $j=1,\dots,2g-2$, Bonahon-Dreyer then established the following equalities (see Proposition 13 of \cite{BonDre1}):
\begin{eqnarray}
&&\lambda_k(A_j)-\lambda_{k+1}(A_j)\label{alpha equation}\\
&&\hspace{1.8cm}=\sigma_{(n-k,k,0),j}+\sigma_{(n-k,0,k),j}+\sum_{i=1}^{k-1}(\tau_{(n-k,i,k-i),j}+\tau'_{(n-k,i,k-i),j}),\nonumber\\
&&\lambda_k(B_j)-\lambda_{k+1}(B_j)\label{beta equation}\\
&&\hspace{1.8cm}=\sigma_{(0,n-k,k),j}+\sigma_{(k,n-k,0),j}+\sum_{i=1}^{k-1}(\tau_{(k-i,n-k,i),j}+\tau'_{(k-i,n-k,i),j}),\nonumber\\
&&\lambda_k(B_j)-\lambda_{k+1}(B_j)\label{gamma equation}\\
&&\hspace{1.8cm}=\sigma_{(k,0,n-k),j}+\sigma_{(0,k,n-k),j}+\sum_{i=1}^{k-1}(\tau_{(i,k-i,n-k),j}+\tau'_{(i,k-i,n-k),j}).\nonumber
\end{eqnarray}
The sum on the right hand side of Equations (\ref{alpha equation}), (\ref{beta equation}) and (\ref{gamma equation}) is the sum of the numbers assigned to all the points in $\Apc\cup\Cpc$ that lie on the $x=k$, $y=k$ and $z=k$ plane respectively. (See Figure \ref{triangle}.) These equations immediately imply that the sums on the right hand side have to be positive. Doing this over every pair of pants given by the pants decomposition gives us $(3n-3)(2g-2)$ linear inequalities involving the shear and triangle invariants. These inequalities are called the \emph{closed leaf inequalities}.

\begin{figure}
\includegraphics[scale=0.65]{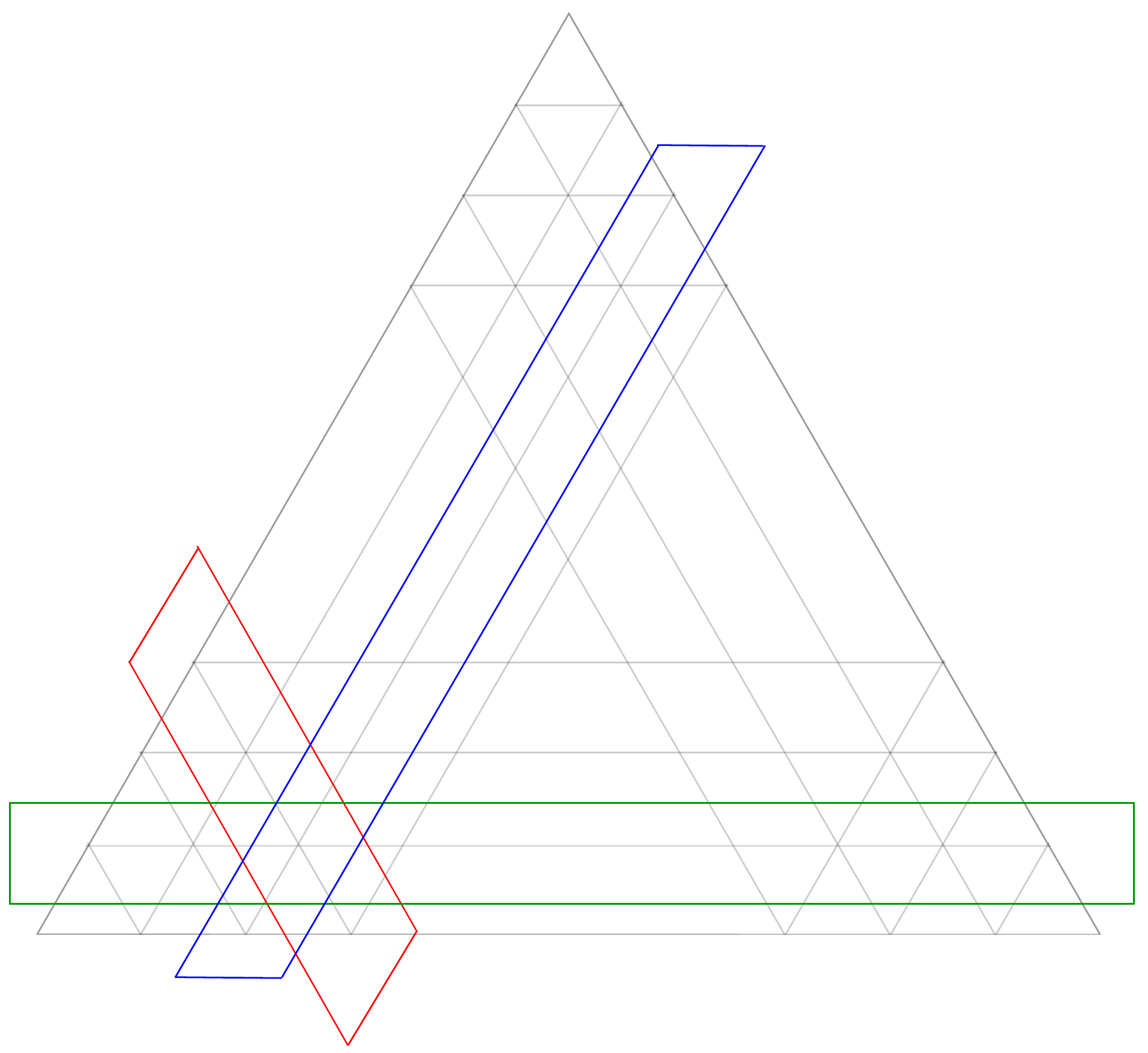}
\put (-432, 291){\makebox[0.7\textwidth][r]{\tiny$(1,0,n-1)$ }}
\put (-392, 291){\makebox[0.7\textwidth][r]{\tiny$(0,1,n-1)$ }}
\put (-452, 263){\makebox[0.7\textwidth][r]{\tiny$(2,0,n-2)$ }}
\put (-412, 263){\makebox[0.7\textwidth][r]{\tiny$(1,1,n-2)$ }}
\put (-372, 263){\makebox[0.7\textwidth][r]{\tiny$(0,2,n-2)$ }}
\put (-472, 235){\makebox[0.7\textwidth][r]{\tiny$(3,0,n-3)$ }}
\put (-432, 235){\makebox[0.7\textwidth][r]{\tiny$(2,1,n-3)$ }}
\put (-392, 235){\makebox[0.7\textwidth][r]{\tiny$(1,2,n-3)$ }}
\put (-352, 235){\makebox[0.7\textwidth][r]{\tiny$(0,3,n-3)$ }}
\put (-530, 116){\makebox[0.7\textwidth][r]{\tiny$(n-3,0,3)$ }}
\put (-550, 88){\makebox[0.7\textwidth][r]{\tiny$(n-2,0,2)$ }}
\put (-510, 88){\makebox[0.7\textwidth][r]{\tiny$(n-3,1,2)$ }}
\put (-570, 60){\makebox[0.7\textwidth][r]{\tiny$(n-1,0,1)$ }}
\put (-530, 60){\makebox[0.7\textwidth][r]{\tiny$(n-2,1,1)$ }}
\put (-490, 60){\makebox[0.7\textwidth][r]{\tiny$(n-3,2,1)$ }}
\put (-550, 32){\makebox[0.7\textwidth][r]{\tiny$(n-1,1,0)$ }}
\put (-510, 32){\makebox[0.7\textwidth][r]{\tiny$(n-2,2,0)$ }}
\put (-470, 32){\makebox[0.7\textwidth][r]{\tiny$(n-3,3,0)$ }}
\put (-295, 116){\makebox[0.7\textwidth][r]{\tiny$(0,n-3,3)$ }}
\put (-315, 88){\makebox[0.7\textwidth][r]{\tiny$(1,n-3,2)$ }}
\put (-275, 88){\makebox[0.7\textwidth][r]{\tiny$(0,n-2,2)$ }}
\put (-335, 60){\makebox[0.7\textwidth][r]{\tiny$(2,n-3,1)$ }}
\put (-295, 60){\makebox[0.7\textwidth][r]{\tiny$(1,n-2,1)$ }}
\put (-255, 60){\makebox[0.7\textwidth][r]{\tiny$(0,n-1,1)$ }}
\put (-355, 32){\makebox[0.7\textwidth][r]{\tiny$(3,n-3,0)$ }}
\put (-315, 32){\makebox[0.7\textwidth][r]{\tiny$(2,n-2,0)$ }}
\put (-275, 32){\makebox[0.7\textwidth][r]{\tiny$(1,n-1,0)$ }}
\put (-559, 140){\makebox[0.7\textwidth][r]{$(3.2)$ }}
\put (-350, 273){\makebox[0.7\textwidth][r]{$(3.3)$ }}
\put (-240, 83){\makebox[0.7\textwidth][r]{$(3.4)$ }}
\caption{Shear and triangle invariants}
\label{triangle}
\end{figure}
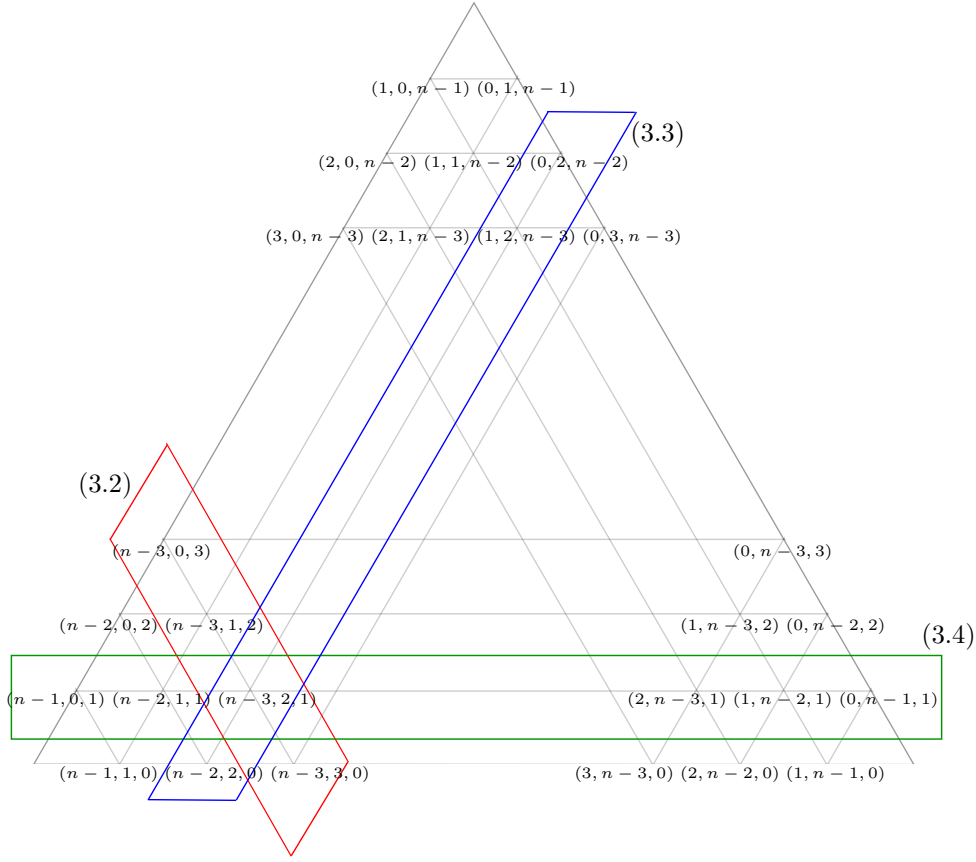

Now, pick any curve $\eta$ in $\Ppc$ and let $P_1$, $P_2$ be the two pairs of pants on either side of $\eta$. Assume without loss of generality that $[A_1]$ and $[A_2]=[A_1^{-1}]$ correspond to $\eta$. Thus, $\lambda_k(A_1)-\lambda_{k+1}(A_1)=\lambda_{n-k}(A_2)-\lambda_{n-k+1}(A_2)$ for all $k=1,\dots,n-1$. This implies the equality
\begin{align*}
&\sigma_{(n-k,k,0),1}+\sigma_{(n-k,0,k),1}+\sum_{i=1}^{k-1}(\tau_{(n-k,i,k-i),1}+\tau'_{(n-k,i,k-i),1})\\
&\hspace{2.4cm}=\sigma_{(k,n-k,0),2}+\sigma_{(k,0,n-k),2}+\sum_{i=1}^{n-k-1}(\tau_{(k,i,n-k-i),2}+\tau'_{(k,i,n-k-i),2}).
\end{align*}
Doing this for each curve in $\Ppc$, we have $(n-1)(3g-3)$ linear equations involving the shear and triangle invariants. These equations are known as the \emph{closed leaf equalities}.

Finally, for each $\eta$ in $\Ppc$, Bonahon-Dreyer also defined $(n-1)$ different numbers specifying how to ``glue" the structures on two pairs of pants together along $\eta$ if the boundary invariants for these two pairs of pants corresponding to $\eta$ are compatible. This gives us another $(3g-3)(n-1)$ \emph{gluing parameters}. (In \cite{BonDre1}, these are known as the \emph{shear invariants along closed leaves}.) However, we will not describe these parameters here as they will not matter for our purposes. 

Putting all of these together, Bonahon-Dreyer specified $(2g-2)(n^2-1)$ numbers associated to the pairs of pants given by $\Ppc$ (the shear and triangle invariants), $(3g-3)(n-1)$ numbers associated to each simple closed curve in $\Ppc$ (the gluing parameters), $(2g-2)(3n-3)$ closed leaf inequalities, and finally $(3g-3)(n-1)$ closed leaf equalities. They then proved (Theorem 2 of \cite{BonDre1}) that one can use this information to obtain a parameterization of $Hit_n(S)$.

\begin{thm}[Bonahon-Dreyer]\label{Bonahon-Dreyer}
The shear invariants, triangle invariants and gluing parameters give a real analytic diffeomorphism from $Hit_n(S)$ to a convex polytope in $\Rbbb^{(g-1)(2n^2+3n-5)}$ of dimension $(2g-2)(n^2-1)$ that is cut out by the $(3g-3)(n-1)$ closed leaf equalities and $(2g-2)(3n-3)$ closed leaf inequalities described above.
\end{thm}

%%%%%%%%%%%%%%%%%%%%%%%%%%%%%%%%%%%%%%%%%%%%%%%%%%
\subsection{A reparameterization of the Hitchin component}\label{A reparameterization of the Hitchin component}
%%%%%%%%%%%%%%%%%%%%%%%%%%%%%%%%%%%%%%%%%%%%%%%%%%
In this section, we give a linear reparameterization of the shear-triangle parameterization. This reparameterization will have exactly $(2g-2)(n^2-1)$ parameters (instead of $(g-1)(2n^2+3n-5)$ parameters), and will be more explicitly analogous to the Fenchel-Nielsen coordinates for $Hit_2(S)$, or the Goldman parameters for $Hit_3(S)$. 

To specify this parameterization, we again choose a pants decomposition $\Ppc$ and the same ideal triangulation $\Tpc$ for $S$ as described in Section \ref{A special ideal triangulation}. On top of that, we choose an orientation on each curve in $\Ppc$. {\bf Henceforth, $\Ppc$ will be an oriented pants decomposition.} %Label the oriented curves in $\Ppc$ by $\{\eta_1,\dots,\eta_{3g-3}\}$.

\begin{notation}
Denote the set of group elements in $\Gamma$ corresponding to oriented closed curves in $\Ppc$ by $\Gamma_\Ppc$.
\end{notation}

We will have three different kinds of parameters. The first kind is the eigenvalue information of the holonomy about each of the oriented simple closed curves in $\Ppc$. More specifically, for any $\eta$ in $\Ppc$, choose any $X$ in $\Gamma$ that corresponds to $\eta$. Then for any $\rho$ in $Hit_n(S)$, the \emph{boundary invariant} corresponding to $\eta$ is
\[\beta_\eta:=\lambda(\rho(X))\in\amf^+.\]
Together, the boundary invariants over all curves in $\Ppc$ take values in $(\amf^+)^{3g-3}$, which is linearly isomorphic to $(\Rbbb^+)^{(3g-3)(n-1)}$. Since we can do this for every $\rho$ in $Hit_n(S)$, we can view these boundary invariants as maps $\beta_\eta: Hit_n(S)\to\amf^+$.

\begin{figure}
\includegraphics[scale=0.55]{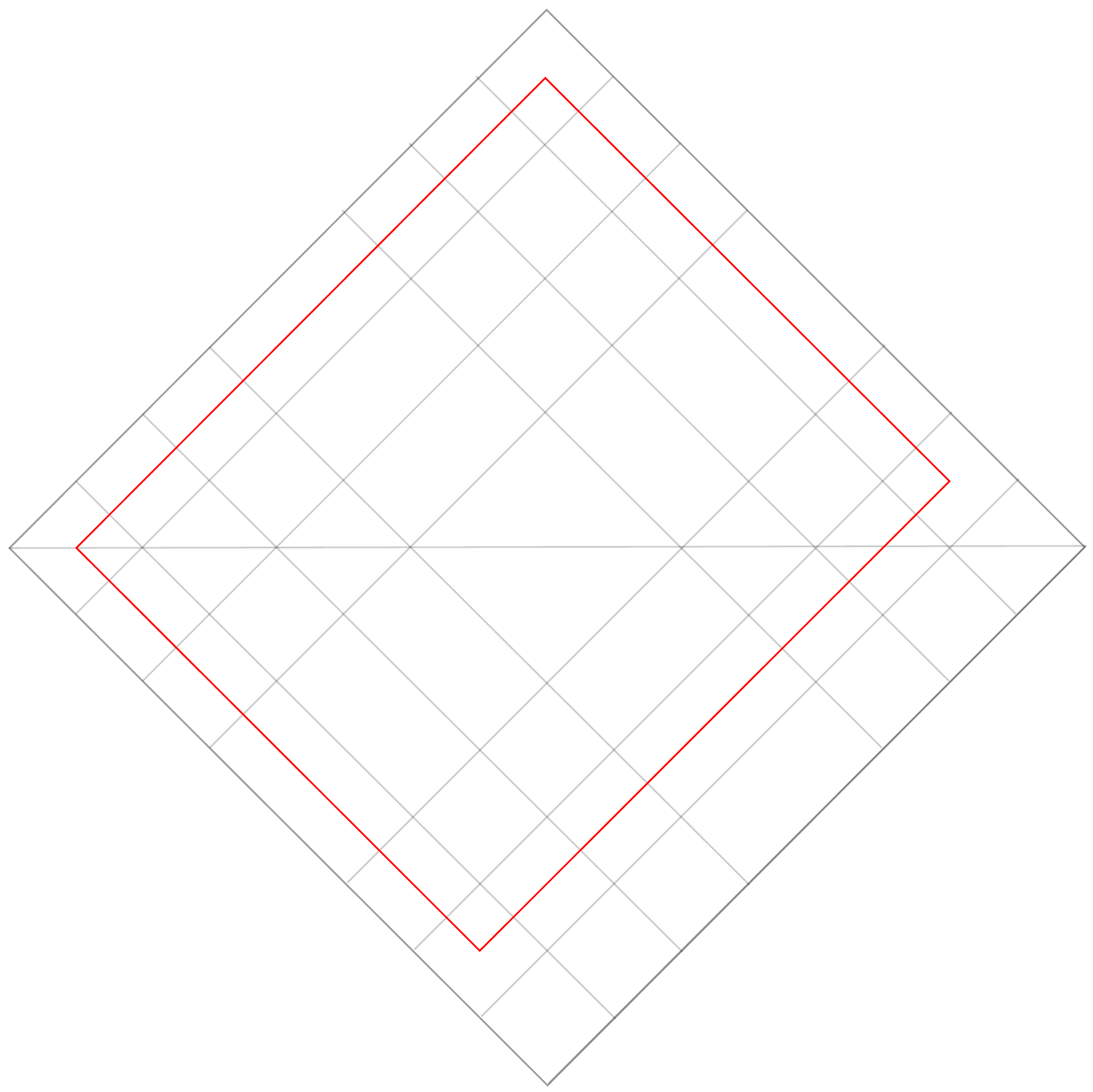}
\put (-397, 270){\makebox[0.7\textwidth][r]{\footnotesize$\sigma_{(1,0,n-1),j}$ }}
\put (-347, 270){\makebox[0.7\textwidth][r]{\footnotesize$\sigma_{(0,1,n-1),j}$ }}
\put (-420, 253){\makebox[0.7\textwidth][r]{\footnotesize$\sigma_{(2,0,n-2),j}$ }}
\put (-373, 253){\makebox[0.7\textwidth][r]{\footnotesize$\tau_{(1,1,n-2),j}$ }}
\put (-325, 253){\makebox[0.7\textwidth][r]{\footnotesize$\sigma_{(0,2,n-2),j}$ }}
\put (-397, 236){\makebox[0.7\textwidth][r]{\footnotesize$\tau_{(2,1,n-3),j}$ }}
\put (-347, 236){\makebox[0.7\textwidth][r]{\footnotesize$\tau_{(1,2,n-3),j}$ }}
\put (-477, 182){\makebox[0.7\textwidth][r]{\footnotesize$\sigma_{(n-2,0,2),j}$ }}
\put (-504, 165){\makebox[0.7\textwidth][r]{\footnotesize$\sigma_{(n-1,0,1),j}$ }}
\put (-455, 165){\makebox[0.7\textwidth][r]{\footnotesize$\tau_{(n-2,1,1),j}$ }}
\put (-477, 148){\makebox[0.7\textwidth][r]{\footnotesize$\sigma_{(n-1,1,0),j}$ }}
\put (-455, 131){\makebox[0.7\textwidth][r]{\footnotesize$\tau_{(1,2,n-3),j}$ }}
\put (-308, 182){\makebox[0.7\textwidth][r]{\footnotesize$\tau_{(1,n-3,2),j}$ }}
\put (-258, 182){\makebox[0.7\textwidth][r]{\footnotesize$\sigma_{(0,n-2,2),j}$ }}
\put (-285, 165){\makebox[0.7\textwidth][r]{\footnotesize$\tau_{(1,n-2,1),j}$ }}
\put (-240, 165){\makebox[0.7\textwidth][r]{\footnotesize$\sigma_{(0,n-1,1),j}$ }}
\put (-308, 148){\makebox[0.7\textwidth][r]{\footnotesize$\sigma_{(2,n-2,0),j}$ }}
\put (-258, 148){\makebox[0.7\textwidth][r]{\footnotesize$\sigma_{(1,n-1,0),j}$ }}
\put (-285, 131){\makebox[0.7\textwidth][r]{\footnotesize$\tau'_{(1,n-2,1),j}$ }}
\put (-330, 131){\makebox[0.7\textwidth][r]{\footnotesize$\tau'_{(2,n-3,1),j}$ }}
\put (-416, 77){\makebox[0.7\textwidth][r]{\footnotesize$\tau'_{(3,1,n-4),j}$ }}
\put (-366, 77){\makebox[0.7\textwidth][r]{\footnotesize$\tau'_{(2,2,n-4),j}$ }}
\put (-396, 59){\makebox[0.7\textwidth][r]{\footnotesize$\tau'_{(2,1,n-3),j}$ }}
\put (-346, 59){\makebox[0.7\textwidth][r]{\footnotesize$\tau'_{(1,2,n-3),j}$ }}
\put (-366, 41){\makebox[0.7\textwidth][r]{\footnotesize$\tau'_{(1,1,n-2),j}$ }}
\caption{Invariants that label points in the red box are the internal parameters.}
\label{parameters}
\end{figure}
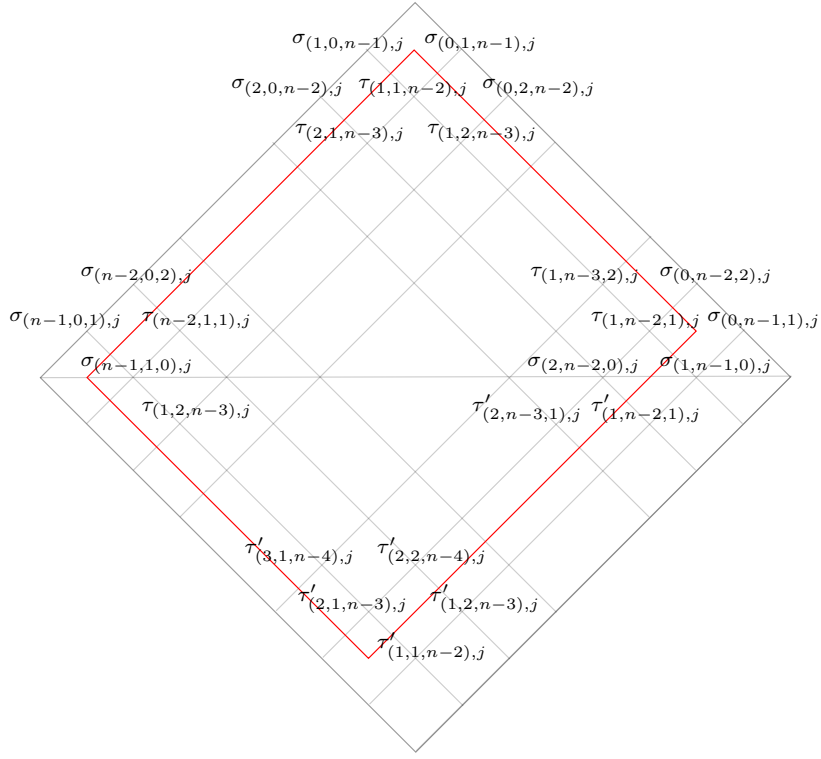

The second kind of parameters are what we will call the \emph{internal parameters}, which are functions associated to each pair of pants $P_j$. In fact, these are a specially chosen subset of the shear and triangle invariants used in the shear-triangle parameterization. For all $j=1,\dots,2g-2$, these parameters are
\begin{enumerate}
\item $\tau_{(x,y,z),j}$ for all positive integers $x,y,z$ such that $x+y+z=n$,
\item $\tau'_{(x,y,z),j}$ for all positive integers $x,y,z$ such that $x+y+z=n$ and $x>1$,
\item $\sigma_{(x,y,0),j}$ for all positive integers $x,y$ such that $x+y=n$ and $x>1$.
\end{enumerate}
One can easily verify that there are $(2g-2)(n-1)(n-2)$ internal parameters. (See Figure \ref{parameters}.)

The third and final kind of parameters are the gluing parameters used in the shear-triangle parameterization. As mentioned before, there are $(3g-3)(n-1)$ of them. We claim that together, the boundary invariants, internal parameters and gluing parameters give us a global parameterization of $Hit_n(S)$.

\begin{prop}\label{parameterization}
The $(3g-3)$ boundary invariants, $(2g-2)(n-1)(n-2)$ internal parameters and $(3g-3)(n-1)$ gluing parameters described above define a real analytic diffeomorphism
\[\Xi:Hit_n(S)\to (\amf^+)^{3g-3}\times\Rbbb^{(2g-2)(n-1)(n-2)}\times\Rbbb^{(3g-3)(n-1)}.\]
\end{prop}

\begin{proof}
Let $ST(n)$ be the convex polytope used to parameterize $Hit_n(S)$ in the shear-triangle parameterization (see Theorem \ref{Bonahon-Dreyer}). We will prove this proposition by showing that the map 
\[\Xi':ST(n)\to  (\amf^+)^{3g-3}\times\Rbbb^{(2g-2)(n-1)(n-2)}\times\Rbbb^{(3g-3)(n-1)}\] 
induced by $\Xi$ is a real-analytic bijection. Observe that Equations (\ref{alpha equation}), (\ref{beta equation}) and (\ref{gamma equation}) imply that $\Xi'$ is the restriction of a linear map to $ST(n)$. Since the dimensions of the domain and range of $\Xi'$ are equal, it is thus sufficient to show that $\Xi'$ is surjective, i.e. we need to show that given a triple of diagonal matrices
\[\big(b_1,b_2,b_3\big)\in(\amf^+)^3\]
and a tuple
\[\Big(\big(t_{(x,y,z)}\big)_{\{(x,y,z)\in\Apc\}},\big(t'_{(x,y,z)}\big)_{\{(x,y,z)\in\Apc,x>1\}},\big(s_{(x,y,0)}\big)_{\{(x,y,0)\in\Cpc,x>1\}}\Big)\in\Rbbb^{(n-1)(n-2)}\]
we can find 
\[\Big(\big(t_{(x,y,z)}\big)_{\{(x,y,z)\in\Apc\}},\big(t'_{(x,y,z)}\big)_{\{(x,y,z)\in\Apc\}},\big(s_{(x,y,z)}\big)_{\{(x,y,z)\in\Cpc\}}\Big)\in\Rbbb^{(n-1)(n+1)}\]
so that
\begin{align}
&b_{k,1}-b_{k+1,1}=s_{(n-k,k,0)}+s_{(n-k,0,k)}+\sum_{i=1}^{k-1}(t_{(n-k,i,k-i)}+t'_{(n-k,i,k-i)}),\label{a equation}\\
&b_{k,2}-b_{k+1,2}=s_{(0,n-k,k)}+s_{(k,n-k,0)}+\sum_{i=1}^{k-1}(t_{(k-i,n-k,i)}+t'_{(k-i,n-k,i)}),\label{b equation}\\
&b_{k,3}-b_{k+1,3}=s_{(k,0,n-k)}+s_{(0,k,n-k)}+\sum_{i=1}^{k-1}(t_{(i,k-i,n-k)}+t'_{(i,k-i,n-k)}).\label{c equation}
\end{align}
where $b_{k,i}$ is the $(k,k)$-th entry of the diagonal matrix $b_i$. Here, Equations (\ref{a equation}), (\ref{b equation}) and (\ref{c equation}) are simply Equations (\ref{alpha equation}), (\ref{beta equation}) and (\ref{gamma equation}) restated using the parameters.

From Equations (\ref{a equation}), (\ref{b equation}), (\ref{c equation}), we can obtain the relation
\begin{eqnarray}\label{bottom line 1}
&&\sum_{k=1}^{n-1}(n-k)(b_{k,1}-b_{k+1,1})+\sum_{k=1}^{n-1}(n-k)(b_{k,2}-b_{k+1,2})\\
&&\hspace{4cm}+\sum_{k=1}^{n-1}(-k)(b_{k,3}-b_{k+1,3})=n\sum_{k=1}^{n-1} s_{(n-k,k,0)}.\nonumber
\end{eqnarray}
To see that this equality holds, observe that $t_{(x,y,z)}$ is a term in the right hand side of 
\begin{itemize}
\item Equation (\ref{a equation}) if and only if $n-k=x$,
\item Equation (\ref{b equation}) if and only if $n-k=y$,
\item Equation (\ref{c equation}) if and only if $n-k=z$.
\end{itemize}
Hence, $t_{(x,y,z)}$ will appear $x$ times in the sum $\displaystyle\sum_{k=1}^{n-1}(n-k)(b_{k,1}-b_{k+1,1})$, $y$ times in the sum $\displaystyle\sum_{k=1}^{n-1}(n-k)(b_{k,2}-b_{k+1,2})$ and $z-n$ times in the sum $\displaystyle\sum_{k=1}^{n-1}(-k)(b_{k,3}-b_{k+1,3})$. Since $x+y+z-n=0$, this implies that $t_{(x,y,z)}$ does not appear as a term on the right hand side of Equation (\ref{bottom line 1}). The same inspection argument for $s_{(x,y,0)}$, $s_{(x,0,z)}$ and $s_{(0,y,z)}$ will yield Equation (\ref{bottom line 1}). Similarly, we can also show that 
\begin{eqnarray}\label{right line 1}
&&\sum_{k=1}^{n-1}(-k)(b_{k,1}-b_{k+1,1})+\sum_{k=1}^{n-1}(n-k)(b_{k,2}-b_{k+1,2})\\
&&\hspace{4cm}+\sum_{k=1}^{n-1}(n-k)(b_{k,3}-b_{k+1,3})=n\sum_{k=1}^{n-1} s_{(0,n-k,k)},\nonumber
\end{eqnarray}
\begin{eqnarray}\label{left line 1}
&&\sum_{k=1}^{n-1}(n-k)(b_{k,1}-b_{k+1,1})+\sum_{k=1}^{n-1}(-k)(b_{k,2}-b_{k+1,2})\\
&&\hspace{4cm}+\sum_{k=1}^{n-1}(n-k)(b_{k,3}-b_{k+1,3})=n\sum_{k=1}^{n-1} s_{(k,0,n-k)}.\nonumber
\end{eqnarray}

Now, observe that from the data we are given, Equation (\ref{bottom line 1}) determines $s_{(1,n-1,0)}$ and Equation (\ref{a equation}) determine $s_{(k,0,n-k)}$ for all $k>1$. By using Equation (\ref{left line 1}), we can also find $s_{(1,0,n-1)}$.

Next, we will show that from the given data, we can also find $s_{(0,n-k,k)}$ for $k=1,\dots,n-1$ and $t'_{(1,n-k-1,k)}$ for $k=1,\dots,n-2$. We will proceed by induction on $k$. For the base case, note that Equation (\ref{b equation}) determines $s_{(0,n-1,1)}$ because we have already found $s_{(1,n-1,0)}$. Then knowing $s_{(0,n-1,1)}$ and $s_{(n-1,0,1)}$ allows us to use Equation (\ref{c equation}) to find $t'_{(1,n-2,1)}$.

For the inductive step, suppose we already know $s_{(0,n-k,k)}$ and $t'_{(1,n-k-1,k)}$ for $k<l$. We need to demonstrate how to find $s_{(0,n-l,l)}$ and $t'_{(1,n-l-1,l)}$. To find $s_{(0,n-l,l)}$, use Equation (\ref{b equation}). Once we have $s_{(0,n-l,l)}$, we can then use Equation (\ref{c equation}) to obtain $t'_{(1,n-l-1,l)}$.
\end{proof}

The reparameterization of $Hit_n(S)$ given in the previous proposition will be called the \emph{modified shear-triangle parameterization}. Choose any Hitchin representation $\rho$. Note that in the above proof, we actually obtained the following identities relating the image of $\Gamma_\Ppc$ under $\lambda\circ\rho$ to the shear and triangle invariants for $\rho$ associated to the triangulation $\Tpc$. To make the notation cleaner, we will keep the $\rho$ dependence implicit and denote $\lambda_k(\rho(X))$ by $\lambda_k(X)$ for any $X\in\Gamma_\Ppc$. For any pair of pants $P_j$ given by $\Ppc$,

\begin{eqnarray}
&&\sum_{k=1}^{n-1}(n-k)(\lambda_k(A_j)-\lambda_{k+1}(A_j))+\sum_{k=1}^{n-1}(n-k)(\lambda_k(B_j)-\lambda_{k+1}(B_j))\label{bottom line}\\
&&\hspace{2.5cm}+\sum_{k=1}^{n-1}(-k)(\lambda_k(C_j)-\lambda_{k+1}(C_j))=n\sum_{k=1}^{n-1} \sigma_{(n-k,k,0),j},\nonumber
\end{eqnarray}
\begin{eqnarray}
&&\sum_{k=1}^{n-1}(-k)(\lambda_k(A_j)-\lambda_{k+1}(A_j))+\sum_{k=1}^{n-1}(n-k)(\lambda_k(B_j)-\lambda_{k+1}(B_j))\label{right line}\\
&&\hspace{2cm}+\sum_{k=1}^{n-1}(n-k)(\lambda_k(C_j)-\lambda_{k+1}(C_j))=n\sum_{k=1}^{n-1} \sigma_{(0,n-k,k),j},\nonumber
\end{eqnarray}
\begin{eqnarray}
&&\sum_{k=1}^{n-1}(n-k)(\lambda_k(A_j)-\lambda_{k+1}(A_j))+\sum_{k=1}^{n-1}(-k)(\lambda_k(B_j)-\lambda_{k+1}(B_j))\label{left line}\\
&&\hspace{2cm}+\sum_{k=1}^{n-1}(n-k)(\lambda_k(C_j)-\lambda_{k+1}(C_j))=n\sum_{k=1}^{n-1} \sigma_{(k,0,n-k),j}.\nonumber
\end{eqnarray}
These equations are a restatement of Equations (\ref{bottom line 1}), (\ref{right line 1}), (\ref{left line 1}) and will be useful for us later.

Recall that we want to study how some geometric properties of Hitchin representations degenerate as we deform in $Hit_n(S)$ along internal sequences. Now that we have a formal description of the modified shear-triangle parameterization, we can formally define these internal sequences. 

\begin{notation}
Define $\pi_j:Hit_n(S)\to\Rbbb^{(n-1)(n-2)}$ to be the projection given by 
\[\pi_j(\rho)=\bigg((\tau_{(x,y,z),j}(\rho))_{x+y+z=n},(\tau'_{(x,y,z),j}(\rho))_{x+y+z=n,x>1},(\sigma_{(x,y,0),j}(\rho))_{x+y=n,x>1}\bigg).\]
\end{notation}

The map $\pi_j$ sends each Hitchin representation to its internal parameters for the $j$-th pair of pants.

\begin{definition}\label{internal sequence}
A sequence $\{\rho_i\}_{i=1}^\infty$ is an \emph{internal sequence} if
\begin{enumerate}
\item There are constants $0<C_1<C_2<\infty$ so that for all $A$ in $\Gamma_\Ppc$, 
\[C_1<\lambda_k(\rho_i(A))-\lambda_{k+1}(\rho_i(A))<C_2\] 
for all $k=1,\dots,n-1$ and all positive integers $i$. In other words, $\lambda(\rho_i(A))$ is uniformly bounded (over all $i$ and $A$ in $\Gamma_\Ppc$) away from the walls of $\amf^+$.
\item For each $j=1,\dots,2g-2$ and for any compact subset $K\subset\Rbbb^{(n-1)(n-2)}$, there is some integer $N$ so that $\pi_j(\rho_i)$ is not in $K$ for $i>N$. 
\end{enumerate}
\end{definition}

One should think of the internal sequences as sequences where we hold the boundary invariants ``essentially fixed" and deform the internal parameters ``as much as possible". In this definition, we do not impose any condition on the gluing parameters because we do not require such a condition for Theorem \ref{main theorem} to hold. 

The rest of the paper will be devoted to the proof of Theorem \ref{main theorem}.

%%%%%%%%%%%%%%%%%%%%%%%%%%%%%%%%%%%%%%%%%%%%%%%%%%%%%%%%%%%%%%%%%%%%%%%%%%%%%%%%%%%%%%%%%%%%%%%%%%%%
\section{Lower bound for lengths of closed curves}\label{Lower bound for lengths of closed curves}
%%%%%%%%%%%%%%%%%%%%%%%%%%%%%%%%%%%%%%%%%%%%%%%%%%%%%%%%%%%%%%%%%%%%%%%%%%%%%%%%%%%%%%%%%%%%%%%%%%%%

In this section, we will fix $\rho$ in $Hit_n(S)$ and a non-identity element $X$ in $\Gamma$. The goal is to first give a combinatorial description of $\rho(X)$, and then use this combinatorial description to obtain the lower bound for $l_\rho(X)$ in Theorem \ref{length lower bound}.

We will denote the Frenet curve corresponding to $\rho$ by $\xi$, and the attracting and repelling fixed points of $X$ by $x^+$ and $x^-$ respectively. Choose an orientation on $\partial_\infty\Gamma$ (drawn as clockwise in all figures), and let $s_0$ and $s_1$ be the two closed subsegments of $\partial_\infty\Gamma$ with endpoints $x^-$ and $x^+$, oriented from $x^-$ to $x^+$ and such that the orientation on $s_0$ agrees with the orientation on $\partial_\infty\Gamma$.

%%%%%%%%%%%%%%%%%%%%%%%%%%%%%%%%%%%%%%%%%%%%%%%%%%
\subsection{The mesh}\label{The mesh}
%%%%%%%%%%%%%%%%%%%%%%%%%%%%%%%%%%%%%%%%%%%%%%%%%%

As mentioned in the introduction, if we choose a hyperbolic metric on $S$, the combinatorial description of $\rho(X)$ needs to capture how the directed geodesic in $S$ associated to $X$ ``winds around" collar neighborhoods of the curves in $\Ppc$ and how it ``crosses between" these collar neighborhoods. The ideal triangulation is good enough to describe the ``crossing", but not the ``winding". Hence, to describe the ``winding", we need to define some additional structure on $\partial_\infty\Gamma$ using $\rho$. We call this additional structure a mesh.

Pick any oriented closed curve $\eta$ in $\Ppc$, and choose any $A$ in $\Gamma_\Ppc$ corresponding to this closed curve. Let $a^-$ and $a^+$ be the repelling and attracting fixed points of $A$ in $\partial_\infty\Gamma$ respectively, and let $r_0$ and $r_1$ be the two closed subsegments of $\partial_\infty\Gamma$ with endpoints $a^-$ and $a^+$, oriented from $a^-$ to $a^+$, and so that the orientation on $r_0$ agrees with the orientation on $\partial_\infty\Gamma$. By reversing the orientation on $\eta$ if necessary, we can assume that there is some $x$ in $r_0\setminus\{a^+\}$ such that $\{x,a^-\}$ lies in $\widetilde{\Tpc}$. Choose any such $x$.

Observe that for all $A,B,C,D$ along $\xi$ in that order, 
\[(A^{(1)},B^{(1)},C^{(1)},D^{(1)})_{A^{(n-1)}\cap D^{(n-1)}}=\prod_{k=1}^{n-1}(A,B,C,D)_{A^{(k-1)}+D^{(n-k-1)}}.\]
Hence, it is an easy consequence of Proposition \ref{cross ratio configuration} that the map 
\begin{align*}
g:\hspace{0.1cm}&r_1\to\Rbbb_{\geq 0}\cup\{\infty\}\\
&z\mapsto\bigg|(\xi(a^+)^{(1)},\xi(x)^{(1)},\xi(z)^{(1)},\xi(a^-)^{(1)})_{\xi(a^+)^{(n-1)}\cap\xi(a^-)^{(n-1)}}\bigg|
\end{align*} 
is also a homeomorphism. Choose $y$ in $r_1\setminus\{a^-\}$ to be the point so that $g(y)$ is minimized subject to the conditions that $\{y,a^+\}$ lies in $\widetilde{\Tpc}$ and $g(y)\geq 1$. Note that if $y'$ is the point in $r_1$ so that $g(y')=1$, then $y$ necessarily lies between $y'$ and $A\cdot y'$. By (8) or Proposition \ref{basic cross ratio} and Proposition \ref{cross ratio and length}, we can then conclude that
\begin{equation}\label{mesh inequality}
1\leq g(y)<g(A\cdot y')=e^{\lambda_1(A)-\lambda_n(A)}.
\end{equation}

Using $x$ and $y$ as described, we can define some additional structure on $\partial_\infty\Gamma$.

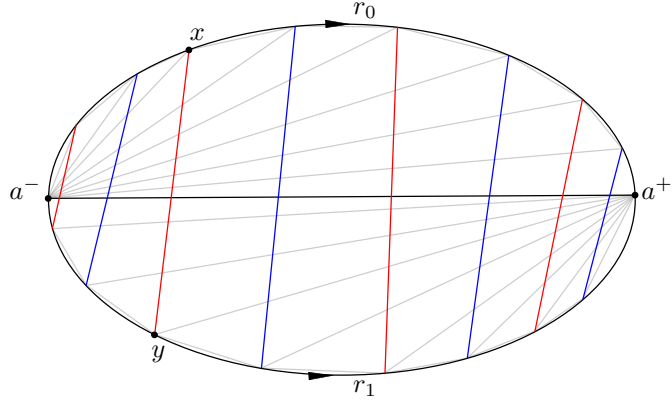
\begin{figure}
\includegraphics[scale=0.6]{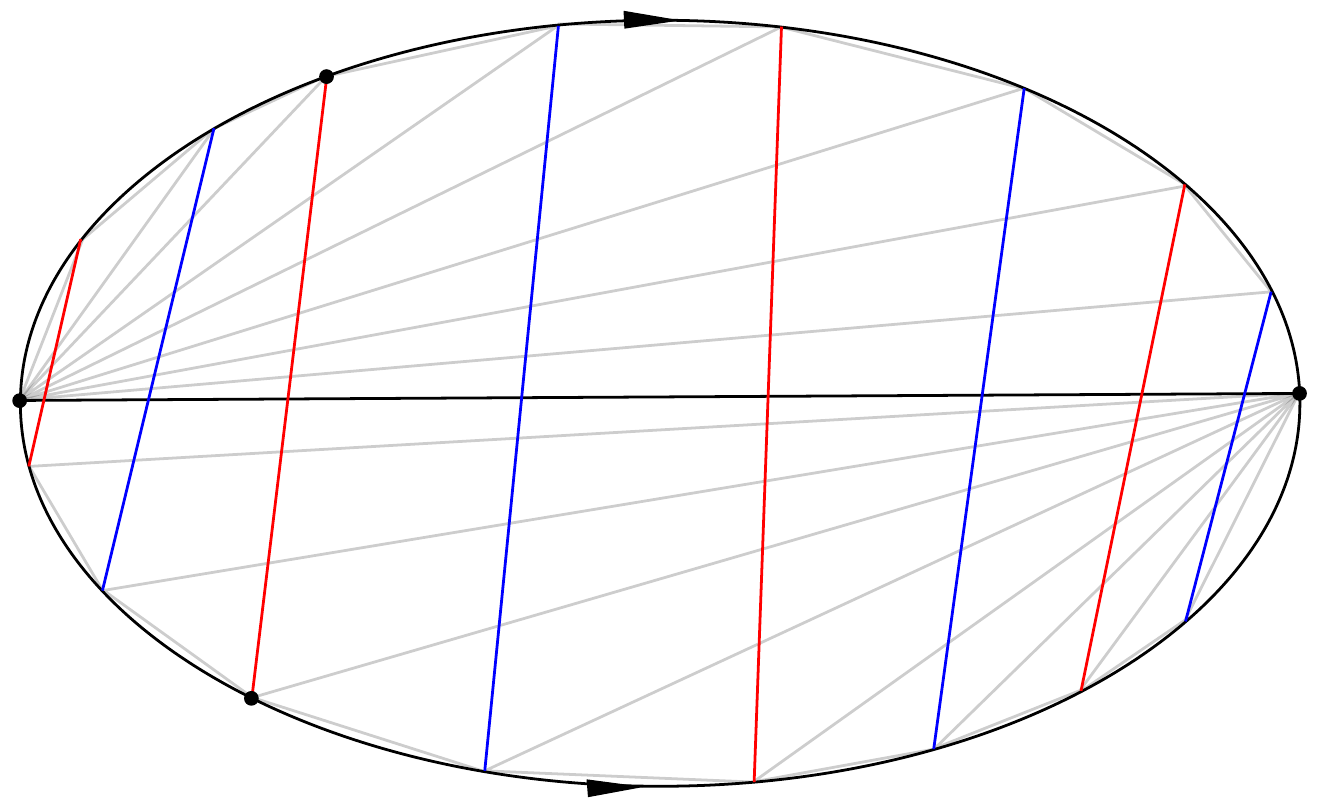}
\put (-350, 140){\makebox[0.7\textwidth][r]{$r_0$ }}
\put (-350, -3){\makebox[0.7\textwidth][r]{$r_1$ }}
\put (-477, 70){\makebox[0.7\textwidth][r]{$a^-$ }}
\put (-238, 70){\makebox[0.7\textwidth][r]{$a^+$ }}
\put (-415, 130){\makebox[0.7\textwidth][r]{$x$ }}
\put (-430, 10){\makebox[0.7\textwidth][r]{$y$ }}
\caption{Two possible meshes for $A$, in blue and red, depending on the choice of $x$.}
\label{mesh}
\end{figure}

\begin{definition}
Let $A$ be an element in $\Gamma_\Ppc$, and choose any $x$ in $r_0$ so that $\{x,a^-\}$ is an edge in $\widetilde{\Tpc}$. Let $y$ be the point in $r_1$ described as above. A \emph{mesh} (see Figure \ref{mesh}) of $A$ is the set of pairs $\{\{A^k\cdot x,A^k\cdot y\}:k\in\Zbbb\}$.
\end{definition}

One can check that if we use any $x'$ in $\langle A\rangle\cdot x$ in place of $x$ and perform this construction, then the mesh we obtain will be the same. This implies that there are only two possible meshes of $A$, because the set 
\[\{x\in\partial_\infty\Gamma:\{x,a^-\}\in\widetilde{\Tpc}\}\]
is the union of two $\langle A\rangle$-orbits. Once and for all, choose one of these meshes, denoted $\Epc_A$, for each $A$ in $\Gamma_\Ppc$, so that if $A'=YAY^{-1}$ for some $Y$ in $\Gamma$, then
\[\Epc_{A'}=Y\cdot\Epc_A.\] 
There is a natural ordering on $\Epc_A$ induced by the action of $A$.

%%%%%%%%%%%%%%%%%%%%%%%%%%%%%%%%%%%%%%%%%%%%%%%%%%
\subsection{Finite combinatorial description of closed curves}\label{Finite combinatorial description of closed curves}
%%%%%%%%%%%%%%%%%%%%%%%%%%%%%%%%%%%%%%%%%%%%%%%%%%
Now, we give a complete description of $\rho(X)$ by finitely many pieces of combinatorial data. This combinatorial data is similar in spirit to the combinatorial data used to describe closed curves in Sections 3.2 and 3.3 of \cite{Zha1}. First, we will use the pair $(x^-,x^+)$ and the subsegments $s_0$, $s_1$ to define several subsets of $\partial_\infty\Gamma$ and $\partial_\infty\Gamma^{[2]}$ that we use to give this combinatorial description. These are summarized in Notation \ref{combinatorial} below. 

\begin{notation}\label{combinatorial}
\begin{itemize}
\item Let $\widetilde{\Ipc}'_X=\widetilde{\Ipc}'$ be the set of edges in $\widetilde{\Tpc}$ that intersect $\{x^-,x^+\}$ and let $\widetilde{\Ipc}_X=\widetilde{\Ipc}$ be the subset of $\widetilde{\Ipc}'$ that are not closed leaves. Observe that both $\widetilde{\Ipc}$ and $\widetilde{\Ipc}'$ are $\langle X\rangle$-invariant, so we can define $\Ipc:=\widetilde{\Ipc}/\langle X\rangle$ and $\Ipc':=\widetilde{\Ipc}'/\langle X\rangle$.
\item A vertex in $\partial_\infty\Gamma$ is a \emph{node} if it is the common vertex of two distinct edges in $\widetilde{\Tpc}$ that intersect $\{x^-,x^+\}$. We call the edge $\{a,b\}$ in $\widetilde{\Ipc}'$ \emph{binodal} if $a$ and $b$ are both nodes. Denote the set of binodal edges in $\widetilde{\Ipc}$ by $\widetilde{\Bpc}_X=\widetilde{\Bpc}$ and let $\Bpc_X=\Bpc:=\widetilde{\Bpc}/\langle X\rangle$.
\item Let $\Vpc_i'$ be the set of vertices of the edges in $\widetilde{\Ipc}'$ that lie in $s_i$.
\end{itemize}
\end{notation}

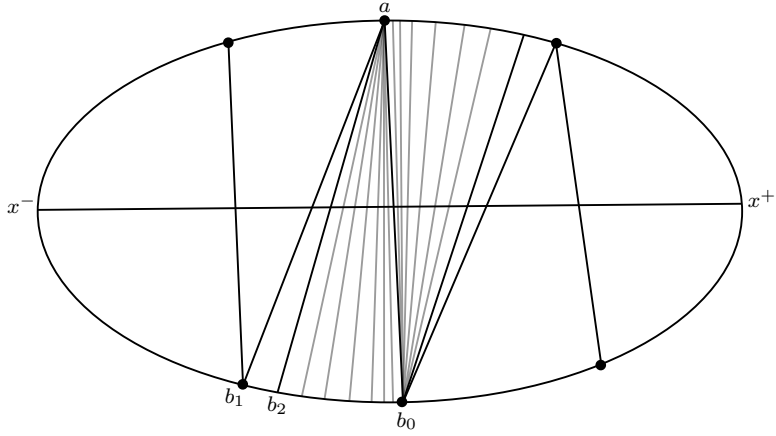
\begin{figure}
\includegraphics[scale=0.9]{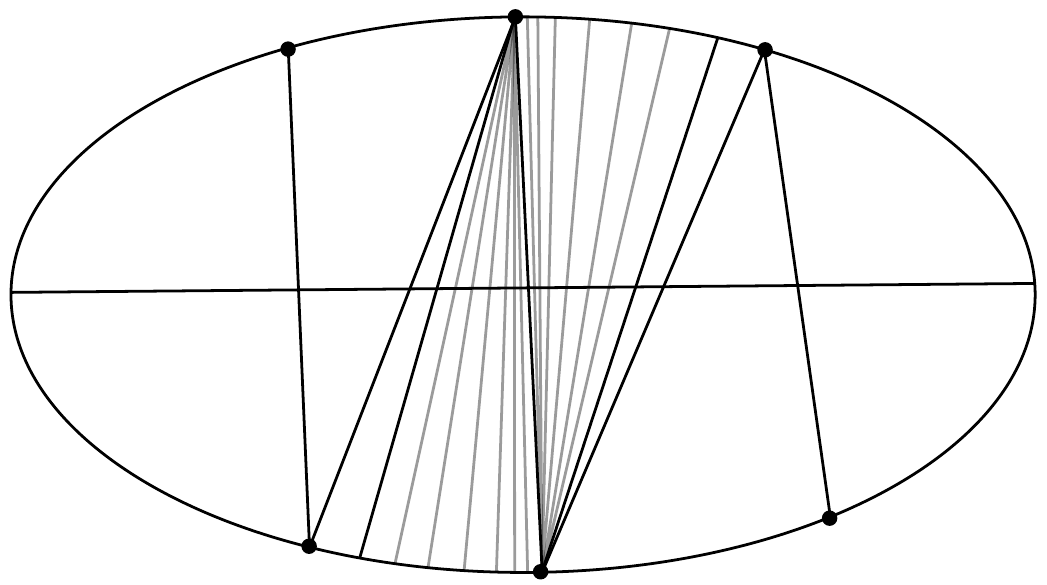}
\put (-385, 151){\makebox[0.7\textwidth][r]{\footnotesize$a$ }}
\put (-239, 77){\makebox[0.7\textwidth][r]{\footnotesize$x^+$ }}
\put (-375, -6){\makebox[0.7\textwidth][r]{\footnotesize$b_0$ }}
\put (-424, 0){\makebox[0.7\textwidth][r]{\footnotesize$b_2$ }}
\put (-440, 3){\makebox[0.7\textwidth][r]{\footnotesize$b_1$ }}
\put (-519, 75){\makebox[0.7\textwidth][r]{\footnotesize$x^-$ }}
\caption{$\Npc_a$ contains the vertices of the grey lines.}
\label{suc}
\end{figure}

Observe that $\Bpc$ is finite, and is empty if and only if $\{x^-,x^+\}$ is a closed leaf in $\widetilde{\Tpc}$. For the rest of this section, we will assume that $\Bpc$ is non-empty. Also, the orientations on $s_0$ and $s_1$ induce orderings $\leq$ on $\Vpc_0'$ and $\Vpc_1'$, which in turn induce an ordering $\preceq$ on $\widetilde{\Ipc}$ defined as follows. Suppose $\{a,b\}$ and $\{a',b'\}$ are edges in $\widetilde{\Ipc}$ so that $a,a'$ lie in $s_0$ and $b,b'$ lie in $s_1$. Then $\{a,b\}\preceq\{a',b'\}$ if and only if $a\leq a'$ and $b\leq b'$. Since the accumulation points of $\widetilde{\Ipc}'$ are exactly the closed leaves, we can define a bijective successor map $\suc:\widetilde{\Ipc}\to\widetilde{\Ipc}$. Moreover, the ordering $\preceq$ induces a cyclic order on $\Ipc$, and the successor map $\suc:\widetilde{\Ipc}\to\widetilde{\Ipc}$ descends to a successor map $\suc:\Ipc\to\Ipc$.

From the way $\widetilde{\Tpc}$ was defined, it is easy to see that every closed leaf in $\widetilde{\Ipc}'$ is binodal. Also, any node is the vertex of exactly two distinct binodal edges, and at most one of these binodal edges is a closed leaf. Thus, for any vertex $a$ of any closed leaf $\{a,b_0\}$ in $\widetilde{\Ipc}'$, there is a unique binodal edge $\{a,b_1\}$ that is not a closed leaf and has $a$ as a vertex. Let $\{a,b_2\}$ be the unique edge in $\widetilde{\Ipc'}$ that is adjacent to $\{a,b_1\}$ and also has $a$ as a vertex. Define 
\[\Npc_a:=\{b\in\partial_\infty\Gamma:\{a,b\}\in\widetilde{\Ipc'}, b\neq b_i\text{ for }i=0,1,2\},\] and let $\Wpc$ be the set of vertices for the closed leaves in $\widetilde{\Ipc}'$ (see Figure \ref{suc}). Then define 
\[\Vpc_i:=\Vpc_i'\setminus\bigcup_{a\in\Wpc}\Npc_a.\]
The main payoff we gain from considering $\Vpc_i$ instead of $\Vpc_i'$ is that $\Vpc_i$ is discrete, which allows us to define bijective successor functions $\suc:\Vpc_i\to\Vpc_i$ for both $i=0,1$.

\begin{figure}
\includegraphics[scale=0.7]{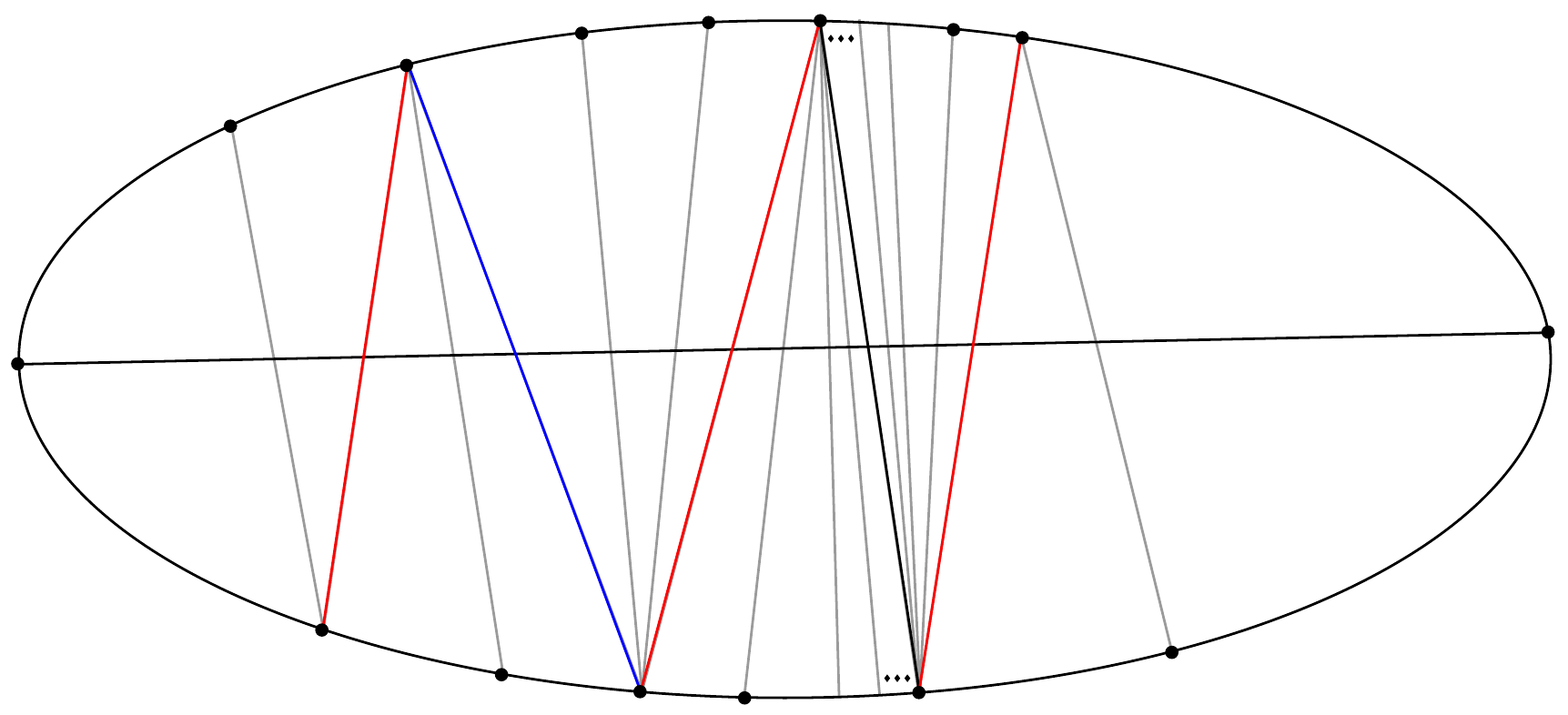}
\put (-595, 77){\makebox[0.7\textwidth][r]{\footnotesize$x^-$ }}
\put (-238, 82){\makebox[0.7\textwidth][r]{\footnotesize$x^+$ }}
\put (-545, 136){\makebox[0.7\textwidth][r]{\footnotesize$\suc^{-1}(a)$ }}
\put (-507, 147){\makebox[0.7\textwidth][r]{\footnotesize$a$ }}
\put (-467, 157){\makebox[0.7\textwidth][r]{\footnotesize$\suc(a)$ }}
\put (-425, 159){\makebox[0.7\textwidth][r]{\footnotesize$\suc^{-1}(a')$ }}
\put (-410, 158){\makebox[0.7\textwidth][r]{\footnotesize$a'$ }}
\put (-367, 157){\makebox[0.7\textwidth][r]{\footnotesize$\suc^{-1}(a'')$ }}
\put (-356, 152){\makebox[0.7\textwidth][r]{\footnotesize$a''$ }}
\put (-528, 13){\makebox[0.7\textwidth][r]{\footnotesize$b$ }}
\put (-480, 0){\makebox[0.7\textwidth][r]{\footnotesize$\suc(b)=\suc^{-1}(b')$ }}
\put (-452, -4){\makebox[0.7\textwidth][r]{\footnotesize$b'$ }}
\put (-420, -5){\makebox[0.7\textwidth][r]{\footnotesize$\suc(b')$ }}
\put (-387, -4){\makebox[0.7\textwidth][r]{\footnotesize$b''$ }}
\put (-325, 5){\makebox[0.7\textwidth][r]{\footnotesize$\suc(b'')$ }}
\caption{$\widetilde{\Ipc'}$ partially drawn. The closed leaf is $\{a',b''\}$, the S-type binodal edges are $\{a,b\}$, $\{a',b'\}$, $\{a'',b''\}$ and the Z-type binodal edge is $\{a,b'\}$.}
\label{binodal}
\end{figure}
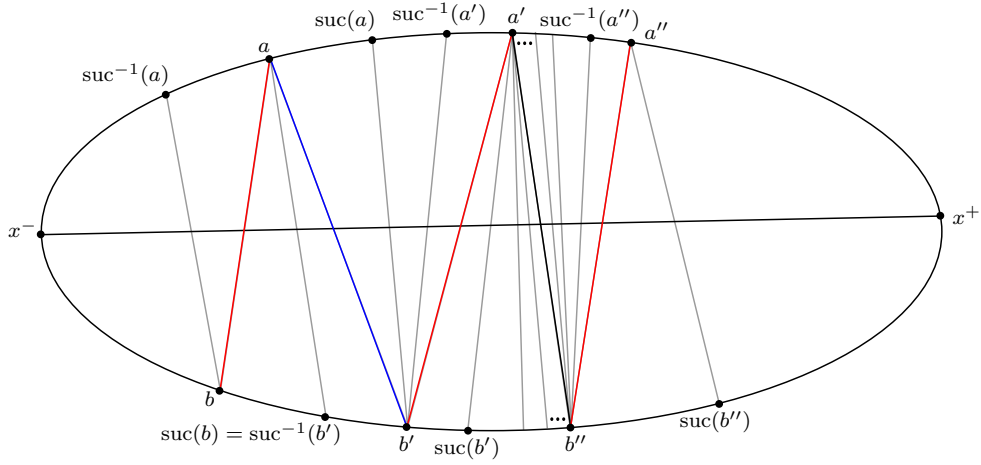

\begin{definition}
Let $\{a,b\}$ be an edge in $\widetilde{\Bpc}$ and assume without loss of generality that $a$ lies in $s_0$ and $b$ lies in $s_1$. We say $\{a,b\}$ is 
\begin{itemize}
\item \emph{Z-type} if $\suc\{a,b\}=\{\suc(a),b\}$ and $\suc^{-1}\{a,b\}=\{a,\suc^{-1}(b)\}$, 
\item \emph{S-type} if $\suc\{a,b\}=\{a,\suc(b)\}$ and $\suc^{-1}\{a,b\}=\{\suc^{-1}(a),b\}$.
\end{itemize}
(See Figure \ref{binodal}.) Let $\widetilde{\Zpc}$ be the edges in $\widetilde{\Bpc}$ that are Z-type and $\widetilde{\Spc}$ be the edges in $\widetilde{\Bpc}$ that are S-type. Since $\widetilde{\Zpc}$ and $\widetilde{\Spc}$ are $\langle X\rangle$-invariant, we can define $\Zpc:=\widetilde{\Zpc}/\langle X\rangle$ and $\Spc:=\widetilde{\Spc}/\langle X\rangle$.
\end{definition}

Note that $\Zpc\cup\Spc=\Bpc$, and the cyclic order on $\Ipc$ induces cyclic orders on $\Zpc$, $\Spc$ and $\Bpc$. Let $e$ and $e'$ be consecutive edges in $\Bpc$ with $e$ preceding $e'$, and observe the following (see Figure \ref{binodal}):
\begin{enumerate}
\item If $e$ and $e'$ are not of the same type, then in $\widetilde{\Bpc}$, there are representatives $\etd$, $\etd'$ of $e$, $e'$ respectively so that $\etd\prec\etd'$ and $\etd$, $\etd'$ share a common vertex.
\item If $e$ and $e'$ are of the same type, then in $\widetilde{\Bpc}$, there are representatives $\etd$, $\etd'$ of $e$, $e'$ respectively so that $\etd\prec\etd'$ and there is exactly one closed leaf between them (in $\widetilde{\Ipc}'$).
\end{enumerate}

If $e$ and $e'$ are not of the same type, choose a pair $\etd$, $\etd'$ as described in (1) and let $A(\etd,\etd')$ be the element in $\Gamma_\Ppc$ that has the common vertex of $\etd$ and $\etd'$ as a fixed point. If $e$, $e'$ are of the same type, choose a pair $\etd$, $\etd'$ as described in (2) and let $A(\etd,\etd')$ be the element in $\Gamma_\Ppc$ whose attracting and repelling fixed points are the initial and terminal vertices of the oriented closed leaf between $\etd$ and $\etd'$. In either case, consider $\Epc_{A(\etd,\etd')}$, the mesh associated to $A(\etd,\etd')$.

\begin{notation}
Let $t(e,e')$ be the signed number of edges in $\Epc_{A(\etd,\etd')}$ that intersect $\{x^+,x^-\}$, where the sign is positive if the ordering on these edges induced by the ordering on $\Epc_{A(e,e')}$ is the same as the ordering induced by the orientation on $s_0$ and $s_1$, and negative otherwise. 
\end{notation}

Observe that $t(e,e')$ does not depend on the choice of $\etd$ and $\etd'$. Cyclically enumerate $\Bpc=\{e_{m+1}=e_1,e_2\dots,e_m\}$, and for each $i=1,\dots,m$, let $T_i$ be the type (Z or S) of $e_i$. Then define the cyclic sequence of tuples
\[\psi_\rho(X)=\psi(X):=\{(\suc^{-1}(e_i),e_i,\suc(e_i),T_i,t(e_i,e_{i+1}))\}_{i=1}^m.\]
This is the combinatorial data we associate to each $X$ in $\Gamma$. If we choose a hyperbolic metric on $S$ and let $\gamma$ be the closed geodesic in $S$ associated to $\rho(X)$, then the cyclic sequence $\{(\suc^{-1}(e_i),e_i,\suc(e_i),T_i)\}_{i=1}^m$ tells us how $\gamma$ ``crosses between" the collar neighborhoods of curves in $\Ppc$ and the cyclic sequence $\{(t(e_i,e_{i+1}))\}_{i=1}^m$ tells us how $\gamma$ ``winds around" these collar neighborhoods.

\begin{figure}
\includegraphics[scale=0.5]{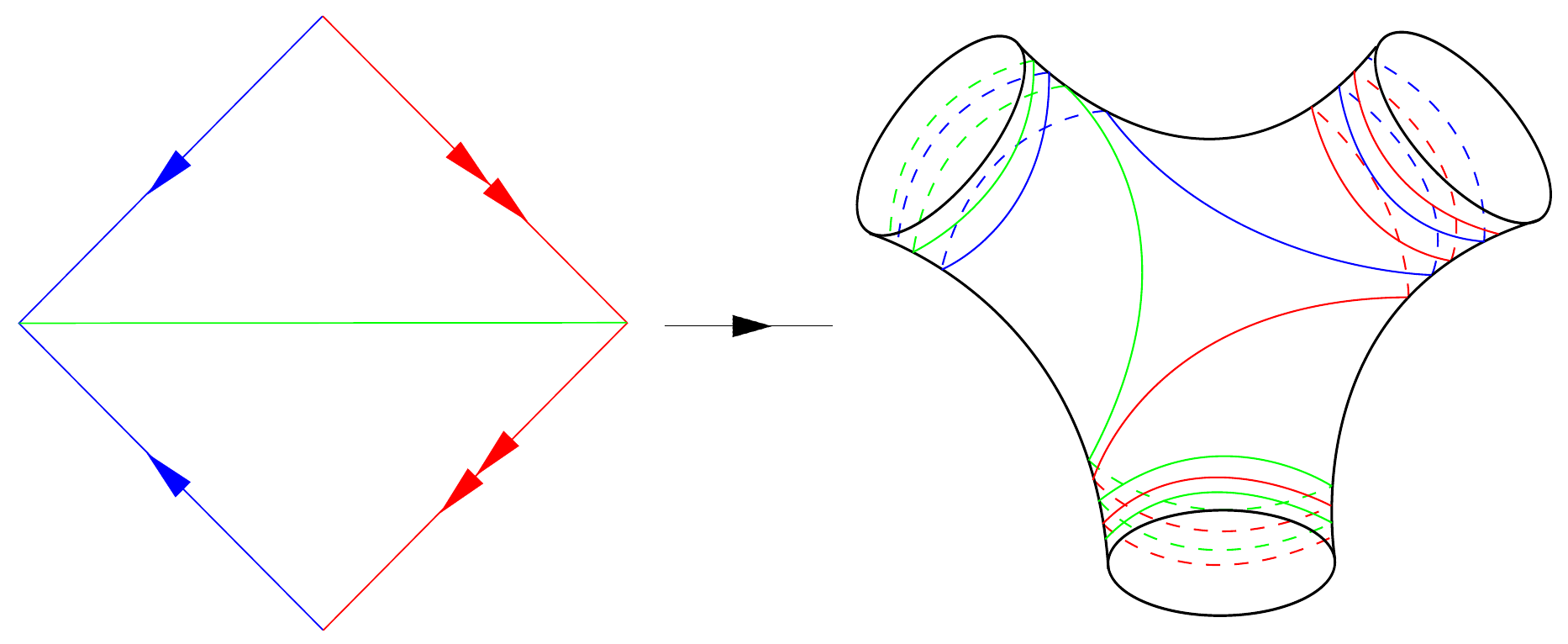}
\put (-516, 53){\makebox[0.7\textwidth][r]{$a$ }}
\put (-405, 53){\makebox[0.7\textwidth][r]{$b$ }}
\put (-460, 110){\makebox[0.7\textwidth][r]{$c$ }}
\put (-460, -5){\makebox[0.7\textwidth][r]{$d$ }}
\put (-460, 59){\makebox[0.7\textwidth][r]{$\etd$ }}
\put (-390, 59){\makebox[0.7\textwidth][r]{$\pi$ }}
\caption{$Q_{\etd}$ is mapped via $\pi$ to a pair of pants.}
\label{Q_e}
\end{figure}
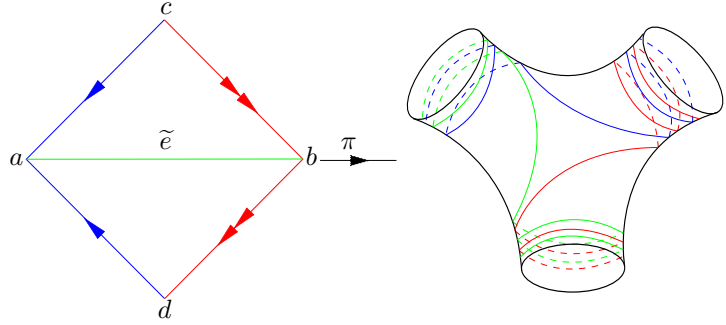

\begin{prop}\label{combinatorial prop}
Let $X_0,X_1$ be elements in $\Gamma$. Then $\psi(X_0)=\psi(X_1)$ if and only if $X_0$ and $X_1$ are conjugate.
\end{prop}

\begin{proof}
It is clear that if $X_0$ and $X_1$ are conjugate, then $\psi(X_0)=\psi(X_1)$. We will now show the converse. Choose a hyperbolic metric on $S$. Then the ideal triangulation $\widetilde{\Tpc}$ can be viewed as a $\Gamma$-invariant ideal triangulation of the Poincar\'e disc $\Dbbb$, so $\Tpc$ is an ideal triangulation of the hyperbolic surface $S$. Also, the union of meshes
\[\widetilde{\Epc}:=\bigcup_{A\in\Gamma_\Ppc}\Epc_A\]
can be viewed as a $\Gamma$-invariant collection of geodesics in $\Dbbb$, so the quotient
\[\Epc:=\widetilde{\Epc}/\Gamma\]
is a collection of $3g-3$ geodesics in the hyperbolic surfaces $S$. Observe that $\gamma$ in $\Epc$ has a lift to $\widetilde{\Epc}$ that lies in $\Epc_A$ if and only if $\gamma$ intersects the closed geodesic corresponding to $A$. Moreover, $\gamma$ intersects $\Ppc$ only at this closed geodesic.

Let $\gamma_{X_0}$, $\gamma_{X_1}$ be the oriented closed geodesics in $S$ that correspond to $X_0$, $X_1$ respectively. It is sufficient to show that if \[\psi(X_0)=\psi(X_1)=\{(\suc^{-1}(e_i),e_i,\suc(e_i),T_i,t(e_i,e_{i+1}))\}_{i=1}^m,\] 
then $\gamma_{X_0}$ and $\gamma_{X_1}$ are homotopic as oriented curves. We will do this by constructing polygons in $S$ along the paths of $\gamma_{X_0}$ and $\gamma_{X_1}$, and show that we can homotope the subsegments of $\gamma_{X_0}$ and $\gamma_{X_1}$ that lie in these polygons relative to the edges of the polygons.

Let $\pi:\Std\to S$ be the covering map. For any non-closed leaf $\etd=\{a,b\}$ in $\widetilde{\Tpc}$, let $c$ and $d$ be points in $\partial\Dbbb$ so that $\{a,c\}$, $\{b,c\}$, $\{a,d\}$, $\{b,d\}$ are in $\widetilde{\Tpc}$. Then let $Q_{\etd}$ be the closed convex quadrilateral in $\Dbbb$ with vertices $a$, $b$, $c$, $d$. Observe that $\pi$ restricted to the interior of $Q_{\etd}$ is injective. (See Figure \ref{Q_e}.)

Pick any $A$ in $\Gamma_\Ppc$ and let $a^-$, $a^+$ be the repelling and attracting fixed points of $A$ respectively. Also, let $r_0$ and $r_1$ be two oriented subsegments of $\partial_\infty\Gamma=\partial\Dbbb$ with endpoints $a^-$ and $a^+$, oriented from $a^-$ to $a^+$, and such that the orientation on $r_0$ agrees with the clockwise orientation on $\partial\Dbbb$. Let $l:=\{b_0,b_1\}$ and $l':=\{b_0',b_1'\}$ be two consecutive geodesics in $\Epc_A$, with $b_0$ and $b_0'$ in $r_0$ and $l$ preceding $l'$. In $r_0$, there is a unique $c_0$ strictly between $b_0$ and $b_0'$ so that $\{c_0,a^-\}$ lies in $\widetilde{\Tpc}$. Similarly, in $r_1$, there is a unique $c_1$ strictly between $b_1$ and $b_1'$ so that $\{c_1,a^+\}$ lies in $\widetilde{\Tpc}$. Let $H_{l,l'}$ be the closed convex hexagon in $\Dbbb$ with vertices $b_0$, $c_0$, $b_0'$, $b_1'$, $c_1$, $b_1$, and observe that $\pi$ restricted to the interior of $H_{l,l'}$ is also injective. (See Figure \ref{H_{l,l'}}.)

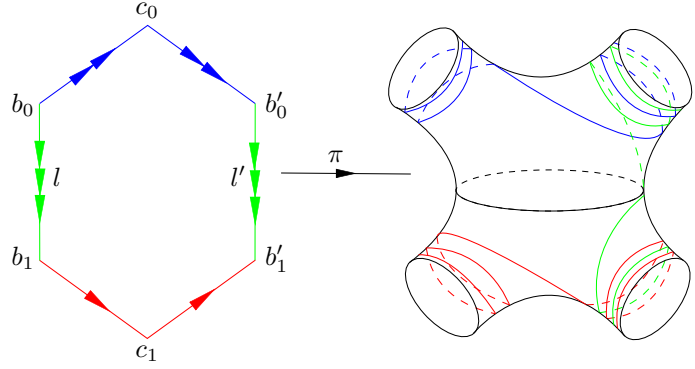
\begin{figure}
\includegraphics[scale=0.48]{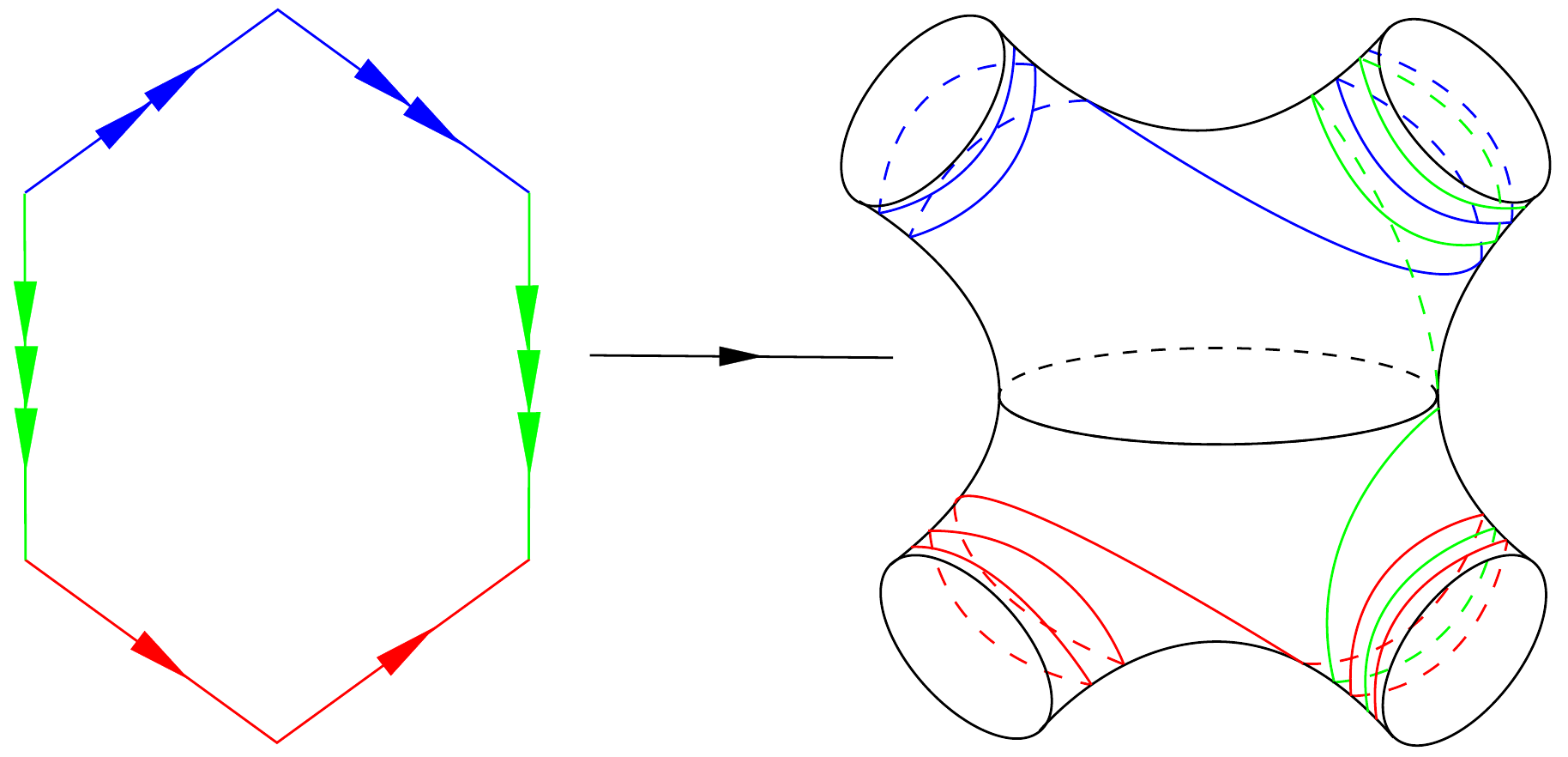}
\put (-500, 30){\makebox[0.7\textwidth][r]{$b_1$ }}
\put (-453, -5){\makebox[0.7\textwidth][r]{$c_1$ }}
\put (-404, 30){\makebox[0.7\textwidth][r]{$b'_1$ }}
\put (-404, 88){\makebox[0.7\textwidth][r]{$b'_0$ }}
\put (-453, 125){\makebox[0.7\textwidth][r]{$c_0$ }}
\put (-500, 88){\makebox[0.7\textwidth][r]{$b_0$ }}
\put (-490, 60){\makebox[0.7\textwidth][r]{$l$ }}
\put (-420, 60){\makebox[0.7\textwidth][r]{$l'$ }}
\put (-383, 69){\makebox[0.7\textwidth][r]{$\pi$ }}
\caption{$H_{l,l'}$ is mapped via $\pi$ to two pairs of pants.}
\label{H_{l,l'}}
\end{figure}

For $j=0,1$, let $\widetilde{\gamma}_{X_j}$ be the axis of $X_j$ and let $\etd_i$ be a lift of $e_i$ that intersects $\widetilde{\gamma}_{X_j}$. Then define the points
\[\ptd_{i,-,X_j}:=\suc^{-1}(\etd_i)\cap\widetilde{\gamma}_{X_j},\ \ptd_{i,+,X_j}:=\suc(\etd_i)\cap\widetilde{\gamma}_{X_j},\]
and let $p_{i,\pm,X_j}=\pi(\ptd_{i,\pm,X_j})$. Let $\alpha_{i,X_j}$ be the oriented closed subsegment of $\gamma_{X_j}$ containing $\pi(\etd_i\cap\widetilde{\gamma}_{X_j})$ and with endpoints $p_{i,-,X_j}$, $p_{i,+,X_j}$, oriented from $p_{i,-,X_j}$ to $p_{i,+,X_j}$. Also, let $\beta_{i,X_j}$ be the closed subsegment of $\gamma_{X_j}$ containing $\pi(\etd_i\cap\widetilde{\gamma}_{X_j})$ and with endpoints $p_{i,-,X_j}$, $p_{i+1,+,X_j}$, oriented from $p_{i,-,X_j}$ to $p_{i+1,+,X_j}$. Observe that $\gamma_{X_j}$ can be written as the cyclic concatenation
\[\alpha_{1,X_j}^{-1}\cdot\beta_{1,X_j}\cdot\alpha_{2,X_j}^{-1}\cdot\beta_{2,X_j}\cdot\dots\cdot\alpha_{m,X_j}^{-1}\cdot\beta_{m,X_j}\]
where $\cdot$ is concatenation and the inverse is reversing the parameterization. Since $\psi(X_0)=\psi(X_1)$, we know that the initial and terminal endpoints of $\alpha_{i,X_0}$ lie on the same edges of $\Tpc$ as those of $\alpha_{i,X_1}$ respectively. For the same reasons, the initial and terminal endpoints of $\beta_{i,X_0}$ lie on the same edges of $\Tpc$ as those of $\beta_{i,X_1}$. It is thus sufficient to show that for all $i=1,\dots,m$, 
\begin{enumerate}
\item $\alpha_{i,X_0}$ is homotopic to $\alpha_{i,X_1}$ and
\item $\beta_{i,X_0}$ is homotopic to $\beta_{i,X_1}$
\end{enumerate}
as oriented curves relative to the edges in $\Tpc$ containing their endpoints.

First, we will show that (1) holds. Observe that $\alpha_0:=\alpha_{i,X_0}$ and $\alpha_1:=\alpha_{i,X_1}$ lie in $\pi(Q_{\etd_i})$ for some lift $\etd_i$ of $e_i$. Also, for each vertex of $\etd_i$, the two edges of $Q_{\etd_i}$ adjacent to this vertex are mapped via $\pi$ to the same edge in $\Tpc$. (See Figure \ref{Q_e}.) Since we know $e_i$ is the same type (Z or S) for both $X_0$ and $X_1$, the lifts $\widetilde{\alpha}_0$, $\widetilde{\alpha}_1$ of $\alpha_0$, $\alpha_1$ respectively that lie in $Q_{\etd_i}$ have their initial endpoints in a common edge of $Q_{\etd_i}$ and their terminal endpoints in a common edge of $Q_{\etd_i}$. It is thus clear that (1) holds.

To show that (2) holds, further partition each $\beta_j:=\beta_{i,X_j}$ as follows. Let $\{q_{1,j},\dots,q_{|t(e_i,e_{i+1})|,j}\}$ be the $|t(e_i,e_{i+1})|$ points of intersection of $\beta_j$ with the mesh $\Epc_{A(e_i,e_{i+1})}$, ordered according to the orientation of $\beta_j$. For $k=0,\dots,|t(e_i,e_{i+1})|$, let $\beta_{k,j}$ be the subsegment of $\beta_j$ with endpoints 
\begin{itemize}
\item $p_{i,-,X_j}$ and $p_{i+1,+,X_j}$ if $t(e_i,e_{i+1})=0$, oriented from $p_{i,-,X_j}$ to $p_{i+1,+,X_j}$,
\item $p_{i,-,X_j}$ and $q_{1,j}$ if $|t(e_i,e_{i+1})|>0$ and $k=0$, oriented from $p_{i,-,X_j}$ to $q_{1,j}$,
\item $q_{k,j}$ and $q_{k+1,j}$ if $|t(e_i,e_{i+1})|>0$ and $0<k<|t(e_i,e_{i+1})|$, oriented from $q_{k,j}$ to $q_{k+1,j}$,
\item $q_{|t(e_i,e_{i+1})|,j}$ and $p_{i+1,+,X_j}$ if $|t(e_i,e_{i+1})|>0$ and $k=|t(e_i,e_{i+1})|$, oriented from $q_{|t(e_i,e_{i+1})|,j}$ to $p_{i+1,+,X_j}$.
\end{itemize}
We now need to show that for $k=0,\dots,|t(e_i,e_{i+1})|$, the segments $\beta_{k,0}$ and $\beta_{k,1}$ are homotopic relative to the edges in $\Epc_{A(e_i,e_{i+1})}$ and $\Tpc$ that contain their endpoints. Observe that $\beta_{k,0}$ and $\beta_{k,1}$ lie in $\pi(H_{l,l'})$ for any consecutive pair $l$, $l'$ in $\Epc_{A(e_i,e_{i+1})}$, with $l$ preceding $l'$. 

Consider the lift $\widetilde{\beta}_{0,j}$ of $\beta_{0,j}$ that lies in $H_{l,l'}$. Since the initial endpoint of $\widetilde{\beta}_{0,j}$ is $\ptd_{i,-,X_j}$, the triple $(\suc^{-1}(e_i),e_i,\suc(e_i))$ determines the edge of $H_{l,l'}$ that $\ptd_{i,-,X_j}$ lies in. Observe then that 
\begin{itemize}
\item if $t(e_i,e_{i+1})<0$, the terminal endpoint $q_{1,j}$ of $\widetilde{\beta}_{0,j}$ lies in $l$, 
\item if $t(e_i,e_{i+1})>0$, the terminal endpoint $q_{1,j}$ of $\widetilde{\beta}_{0,j}$ lies in $l'$,
\item if $t(e_i,e_{i+1})=0$, the triple $(\suc^{-1}(e_{i+1}),e_{i+1},\suc(e_{i+1}))$ determines the edge of $H_{l,l'}$ containing $\ptd_{i,-,X_j}$, which is the terminal endpoint of $\widetilde{\beta}_{0,j}$.
\end{itemize}
In any case, this proves that $\widetilde{\beta}_{0,0}$ and $\widetilde{\beta}_{0,1}$ have initial endpoints on the same edge of $H_{l,l'}$ and terminal endpoints on the same edge in $H_{l,l'}$. Similar arguments show the same for $\widetilde{\beta}_{k,0}$ and $\widetilde{\beta}_{k,1}$ for $k=1,\dots,|t(e_i,e_{i+1})|$, so (2) holds.
\end{proof}

%%%%%%%%%%%%%%%%%%%%%%%%%%%%%%%%%%%%%%%%%%%%%%%%%%
\subsection{Crossing and winding $(p)$-subsegments of $X$}
%%%%%%%%%%%%%%%%%%%%%%%%%%%%%%%%%%%%%%%%%%%%%%%%%%
For the rest of this section, we will use the combinatorial description $\psi(X)$ of $\rho(X)$ to obtain a lower bound for $l_{\rho}(X)$. Let $H$ be the plane $\xi(x^-)^{(1)}+\xi(x^+)^{(1)}$ in $\Rbbb^n$. The next two definitions describe two kinds of subsegments of $\Pbbb(H)$ that we will use to obtain our lower bound. These are the crossing $(p)$-subsegments and the winding $(p)$-subsegments

\begin{definition}
Let $\etd=\{a,b\}$ be an element in $\widetilde{\Bpc}$. Assume without loss of generality that $a$ lies in $s_0$ and $b$ lies in $s_1$. For $p=0,\dots, n-1$, define the projective points $L_{p,+}(\etd)$, $L_p(\etd)$, $L_{p,-}(\etd)$ as follows:
\begin{itemize}
\item $L_p(\etd):=\Pbbb(\xi(a)^{(p)}+\xi(b)^{(n-p-1)})\cap \Pbbb(H)$
\item If $\etd$ is in $\widetilde{\Zpc}$, let 
\begin{align*}
&L_{p,+}(\etd):=\Pbbb(\xi(\suc(a))^{(p)}+\xi(b)^{(n-p-1)})\cap \Pbbb(H),\\ 
&L_{p,-}(\etd):=\Pbbb(\xi(a)^{(p)}+\xi(\suc^{-1}(b))^{(n-p-1)})\cap \Pbbb(H).
\end{align*} 
\item If $\etd$ is in $\widetilde{\Spc}$, let 
\begin{align*}
&L_{p,+}(\etd):=\Pbbb(\xi(a)^{(p)}+\xi(\suc(b))^{(n-p-1)})\cap \Pbbb(H),\\ 
&L_{p,-}(\etd):=\Pbbb(\xi(\suc^{-1}(a))^{(p)}+\xi(b)^{(n-p-1)})\cap \Pbbb(H).
\end{align*} 
\end{itemize}
\end{definition}

\begin{definition}
Let $\etd=\{a,b\}$ and $\etd'=\{a',b'\}$ be two consecutive elements in $\widetilde{\Bpc}$, with $\etd$ preceding $\etd'$. 
\begin{itemize}
\item The \emph{crossing $(p)$-subsegment} of $\Pbbb(H)$ corresponding to $\etd$, denoted by $c_p(\etd)$, is the closed subsegment of $\Pbbb(H)$ containing $L_p(\etd)$ with endpoints $L_{p,-}(\etd)$ and $L_{p,+}(\etd)$.
\item The \emph{winding $(p)$-subsegment} of $\Pbbb(H)$ corresponding to $\etd$ and $\etd'$, denoted by $w_p(\etd,\etd')$, is the closed subsegment of $\Pbbb(H)$ containing $L_p(\etd)$ and $L_p(\etd')$, with endpoints $L_{p,-}(\etd)$ and $L_{p,+}(\etd')$.
\end{itemize}
\end{definition}

The crossing $(p)$-subsegments are generalizations of the crossing segments used in \cite{Zha1} while the winding $(p)$-subsegments are analogous to the winding and pants-changing segments used in \cite{Zha1}.

If we apply (2) of Lemma \ref{important lemma} to this setting for a fixed $p=0,\dots,n-1$, we see that there is some subsegment $\gamma$ of $\Pbbb(H)$ with endpoints $\xi(x^-)^{(1)}$, $\xi(x^+)^{(1)}$ that contains $L_{p,+}(\etd)$, $L_{p,-}(\etd)$ for all $\etd$ in $\widetilde{\Bpc}$. In particular, $c_p(\etd)$ and $w_p(\etd)$ are well-defined and lie in $\gamma$ for all $\etd$ in $\widetilde{\Bpc}$. In fact, we can say more. If we orient $\gamma$ from $\xi(x^-)^{(1)}$ to $\xi(x^+)^{(1)}$, then we also have orientations induced on $c_p(\etd)$ and $w_p(\etd)$. Part (2) of Lemma \ref{important lemma} then tells us that the orientation on $c_p(\etd)$ is from $L_{p,-}(\etd)$ to $L_{p,+}(\etd)$ and the induced orientation on $w_p(\etd,\etd')$ is from $L_{p,-}(\etd)$ to $L_{p,+}(\etd')$. 

Next, we define a notion of length for subsegments of $\gamma$. This gives us a notion of length for each crossing $(p)$-subsegment and winding $(p)$-subsegment, which we will use to obtain a lower bound for $l_\rho(X)$.

\begin{definition}
Let $\gamma$ be a subsegment of $\Pbbb(H)$ with endpoints $\xi(x^-)^{(1)}$ and $\xi(x^+)^{(1)}$. Let $y$, $z$ be two points in $\gamma$ so that $\xi(x^-)^{(1)},y,z,\xi(x^+)^{(1)}$ lie on $\gamma$ in that order. Then let $\eta$ be the closed subsegment of $\gamma$ with endpoints $y$, $z$. The length of $\eta$, denoted $l(\eta)$, is given by
\[l(\eta):=\log(\xi(x^-)^{(1)},y,z,\xi(x^+)^{(1)}).\]
\end{definition}

We will now obtain a lower bound for the length of $X$ in terms of the lengths of the crossing $(p)$-subsegments and winding $(p)$-subsegments of $H$. Choose an edge $\ztd$ in $\widetilde{\Zpc}$ and let $\ztd'=X\cdot\ztd$. Observe that the set of elements in $\widetilde{\Zpc}$ between $\ztd$ and $\ztd'$ is finite, so we can enumerate them according to the ordering on $\widetilde{\Zpc}$. In other words, the set of elements between $\ztd$ and $\ztd'$ can be written as
\[\{\ztd_1,\dots,\ztd_{|\Zpc|+1}\}\]
where 
\[\ztd=\ztd_1\prec\ztd_2\prec\dots\prec\ztd_{|\Zpc|+1}=\ztd'.\]

\begin{lem}\label{crossing lower bound 1}
Fix any $p=0,\dots,n-1$. Then
\[l_\rho(X)\geq\sum_{i=1}^{|\Zpc|}l(c_p(\ztd_i)).\]
\end{lem}

\begin{proof}
By Proposition \ref{cross ratio and length}, we know 
\[\log\bigg(\frac{\omega_n}{\omega_1}\bigg)=\log(\xi(x^-)^{(1)},L_{p,-}(\ztd_1),L_{p,-}(\ztd_{|\Zpc|+1}),\xi(x^+)^{(1)}).\]
Moreover, since $\suc(\ztd_i)\preceq\suc^{-1}(\ztd_{i+1})$, we can use (2) of Lemma \ref{important lemma} to see that
\[(\xi(x^-)^{(1)},L_{p,+}(\ztd_i),L_{p,-}(\ztd_{i+1}),\xi(x^+)^{(1)})\geq1\] 
for $i=1,\dots,|\Zpc|$. Appying (8) of Proposition \ref{basic cross ratio}, we then have 
\begin{eqnarray*}
&&(\xi(x^-)^{(1)},L_{p,-}(\ztd_1),L_{p,-}(\ztd_{|\Zpc|+1}),\xi(x^+)^{(1)})\\
&=&\prod_{i=1}^{|\Zpc|}(\xi(x^-)^{(1)},L_{p,-}(\ztd_i),L_{p,+}(\ztd_i),\xi(x^+)^{(1)})\cdot(\xi(x^-)^{(1)},L_{p,+}(\ztd_i),L_{p,-}(\ztd_{i+1}),\xi(x^+)^{(1)})\\
&\geq&\prod_{i=1}^{|\Zpc|}(\xi(x^-)^{(1)},L_{p,-}(\ztd_i),L_{p,+}(\ztd_i),\xi(x^+)^{(1)}).
\end{eqnarray*}
Taking the logarithm gives us the lemma.
\end{proof}

Similarly, if we choose an edge $\std$ in $\widetilde{\Spc}$ and let $\std'=X\cdot\std$, then we can label the edges in $\widetilde{\Spc}$ between $\std$ and $\std'$ by
\[\{\std_1,\dots,\std_{|\Spc|+1}\}\]
where 
\[\std=\std_1\prec\std_2\prec\dots\prec\std_{|\Spc|+1}=\std'.\]
The same proof as in Lemma \ref{crossing lower bound 1} will give the following lemma.

\begin{lem}\label{crossing lower bound 2}
Fix any $p=0,\dots,n-1$. Then
\[l_\rho(X)\geq\sum_{i=1}^{|\Spc|}l(c_p(\std_i)).\]
\end{lem}

We now want a similar lower bound for $l_\rho(X)$ in terms of the lengths of the winding $(p)$-subsegments. As before, let $\btd$ be any edge in $\widetilde{\Bpc}$, let $\btd'=X\cdot\btd$ and label the edges in $\widetilde{\Bpc}$ between $\btd$ and $\btd'$ by
\[\{\btd_1,\dots,\btd_{|\Bpc|+1}\}\]
where 
\[\btd=\btd_1\prec\btd_2\prec\dots\prec\btd_{|\Bpc|+1}=\btd'.\]
Also, define
\begin{align*}
&\widetilde{\Dpc_1}:=\{(\btd_i,\btd_{i+1}):\btd_i,\btd_{i+1}\text{ that are of the same type}\},\\
&\widetilde{\Dpc_2}:=\{(\btd_i,\btd_{i+1}):\btd_i,\btd_{i+1}\text{ that are not of the same type}\}.
\end{align*}

\begin{lem}\label{winding lower bound}
Fix $p=0,\dots,n-1$. Then
\[3\cdot l_\rho(X)\geq\sum_{i=1}^{|\Bpc|}l(w_p(\btd_i,\btd_{i+1})).\]
\end{lem}

\begin{proof}
Label the set $\{L_{p,-}(\btd_i), L_p(\btd_i),L_{p,+}(\btd_i), L_{p,-}(\btd_{i+1}), L_p(\btd_{i+1}),L_{p,+}(\btd_{i+1})\}$ by $\{d_{i,1},\dots,d_{i,k_i}\}$ with $d_{i,1},\dots,d_{i,k_i}$ arranged in this order along $w_p(\btd_i,\btd_{i+1})$. For each $j=1,\dots,k_i-1$, let $s_{i,j}$ be the closed subsegment of $w_p(\btd_i,\btd_{i+1})$ with endpoints $d_{i,j}$ and $d_{i,j+1}$, and let 
\[\Spc:=\{s_{i,j}:i=1,\dots,|\Bpc|;j=1,\dots,k_i-1\}.\]
Note that either $s_{i,j}=s_{i',j'}$ or $s_{i,j}\cap s_{i',j'}$ is at most a single point. Also, it is clear that 
\[\bigcup_{j=1}^{k_i-1}s_{i,j}=w_p(\btd_i,\btd_{i+1}).\] 

Observe also that any point in the subsegment between $L_{p,-}(\btd_1)$ and $L_{p,+}(\btd_{|\Bpc|+1})$ is contained in the interior of at most three different winding $(p)$-segments. Hence, 
\begin{align*}
\sum_{i=1}^{|\Bpc|}l(w_p(\btd_i,\btd_{i+1}))&=\sum_{i=1}^{|\Bpc|}\sum_{j=1}^{k_i-1}l(s_{i,j})\\
&\leq 3\sum_{s\in\Spc}l(s)\\
&= 3\cdot l_\rho(X).
\end{align*}
\end{proof}

Now, we can give a lower bound for $l_\rho(X)$ in terms of the lengths of the crossing $(p)$-subsegments and the winding $(p)$-subsegments. 

\begin{prop}\label{segment lower bound}
Fix $p=0,\dots,n-1$. Then
\begin{eqnarray*}l_\rho(X)&\geq&\frac{1}{11n}\sum_{p=0}^{n-1}\bigg(\sum_{i=1}^{|\Zpc|}l(c_p(\ztd_i))+\sum_{i=1}^{|\Spc|}l(c_p(\std_i))+\sum_{\widetilde{\Dpc}_1}l(w_p(\btd_i,\btd_{i+1}))\\
\hspace{1cm}&&\hspace{4cm}+\sum_{\widetilde{\Dpc}_2}(l(w_1(\btd_i,\btd_{i+1}))+l(w_{n-2}(\btd_i,\btd_{i+1})))\bigg).
\end{eqnarray*}
\end{prop}

\begin{proof}
Lemma \ref{winding lower bound} implies that 
\[6\cdot l_\rho(X)\geq\sum_{\widetilde{\Dpc}_2}(l(w_1(\btd_i,\btd_{i+1}))+l(w_{n-2}(\btd_i,\btd_{i+1})))\]
and for all $p=0,\dots,n-1$,
\[3\cdot l_\rho(X)\geq\sum_{\widetilde{\Dpc}_1}l(w_p(\btd_i,\btd_{i+1})).\]

Sum these two inequalities with the inequalities in Lemma \ref{crossing lower bound 1} and Lemma \ref{crossing lower bound 2}, and then take average over $p$.
\end{proof}

%%%%%%%%%%%%%%%%%%%%%%%%%%%%%%%%%%%%%%%%%%%%%%%%%%%
\subsection{Lower bound for the length of a closed curve}\label{Lower bound for the length of a closed curve}
%%%%%%%%%%%%%%%%%%%%%%%%%%%%%%%%%%%%%%%%%%%%%%%%%%%
Let $\etd,\etd'$ in $\widetilde{\Bpc}$ be any consecutive pair with $\etd$ preceding $\etd'$, and let $e$, $e'$ be the equivalence classes in $\Bpc$ containing $\etd$ and $\etd'$ respectively. We now want to define numbers $K=K(\rho)$ and $L=L(\rho)$ which depend only on $\rho$, so that 
\[\frac{1}{n}\sum_{p=0}^{n-1}l(c_p(\etd))\geq K,\]  
\[\frac{1}{n}\sum_{p=0}^{n-1}l(w_p(\etd,\etd'))\geq\max\{0,|t(e,e')|-2\}\cdot L\] 
when $(\etd,\etd')$ is in $\widetilde{\Dpc}_1$, and  
\[l(w_1(\etd,\etd'))+l(w_{n-2}(\etd,\etd'))\geq\max\{0,|t(e,e')|-2\}\cdot L\] 
when $(\etd,\etd')$ is in $\widetilde{\Dpc}_2$. These estimates, together with Proposition \ref{segment lower bound}, will allow us to obtain a lower bound for $l_\rho(X)$ in Theorem \ref{length lower bound}. Let us start with $K$.

\begin{notation}
Consider any triangle $\{\{a,b\},\{b,c\},\{c,a\}\}$ in $\widetilde{\Tpc}$. For $p=1,\dots,n-1$, consider the collection of subspaces
\[\Mpc^\xi_p(a,b,c)=\Mpc_p(a,b,c):=\{\xi(a)^{(p-r)}+\xi(b)^{(n-p-1)}+\xi(c)^{(r-1)}:r=1,\dots,p\}\]
\end{notation}

\begin{lem}\label{Z length inequality}
Let $\etd=\{a,b\}$ be an edge in $\widetilde{\Zpc}$, and suppose $a$ lies in $s_0$ and $b$ lies in $s_1$. Then the following hold.
\begin{enumerate}
\item For all $p=1,\dots,n-1$ and for all $M$ in $\Mpc_p(a,b,\suc(a))$, we have
\[l(c_p(\etd))\geq\log(\xi(\suc^{-1}(b)),\xi(a),\xi(\suc(a)),\xi(b))_M.\]
\item For all $p=0,\dots,n-2$ and for all $M$ in $\Mpc_{n-p-1}(b,a,\suc^{-1}(b))$, we have 
\[l(c_p(\etd))\geq\log(\xi(\suc^{-1}(b)),\xi(a),\xi(\suc(a)),\xi(b))_M.\]
\end{enumerate}
\end{lem}

\begin{proof}
Proof of (1). Let $s'_0$, $s''_0$ be the closed subinterval of $s_0$ with endpoints $\suc(a)$ and $x^+$, $x^-$ and $a$ respectively. By (2) of Lemma \ref{important lemma}, we know that for any $p=1,\dots,n-1$ and any $r=1,\dots,p$, there exists $a'$ in $s'_0$ and $a''$ in $s''_0$ (see Figure \ref{Ztype}) so that 
\begin{eqnarray*}
L_{p,+}(\etd)&=&\Pbbb(\xi(b)^{(n-p-1)}+\xi(\suc(a))^{(p)})\cap\Pbbb(H)\\
&=&\Pbbb(\xi(a)^{(p-r)}+\xi(b)^{(n-p-1)}+\xi(\suc(a))^{(r-1)}+\xi(a')^{(1)})\cap \Pbbb(H),\\
L_{p,-}(\etd)&=&\Pbbb(\xi(\suc^{-1}(b))^{(n-p-1)}+\xi(a)^{(p)})\cap \Pbbb(H)\\
&=&\Pbbb(\xi(a)^{(p-r)}+\xi(b)^{(n-p-1)}+\xi(\suc(a))^{(r-1)}+\xi(a'')^{(1)})\cap \Pbbb(H).
\end{eqnarray*}
These imply that 
\begin{eqnarray*}
&&(\xi(x^-)^{(1)},L_{p,-}(\etd),L_{p,+}(\etd),\xi(x^+)^{(1)})\\
&=&(\xi(x^-)^{(1)},\xi(a'')^{(1)},\xi(a')^{(1)},\xi(x^+)^{(1)})_{\xi(a)^{(p-r)}+\xi(b)^{(n-p-1)}+\xi(\suc(a))^{(r-1)}}\\
&\geq&(\xi(\suc^{-1}(b)),\xi(a),\xi(\suc(a)),\xi(b))_{\xi(a)^{(p-r)}+\xi(b)^{(n-p-1)}+\xi(\suc(a))^{(r-1)}}
\end{eqnarray*}
where the final inequality is a consequence of Proposition \ref{useful cross ratio inequalities}. 

\begin{figure}
\includegraphics[scale=0.45]{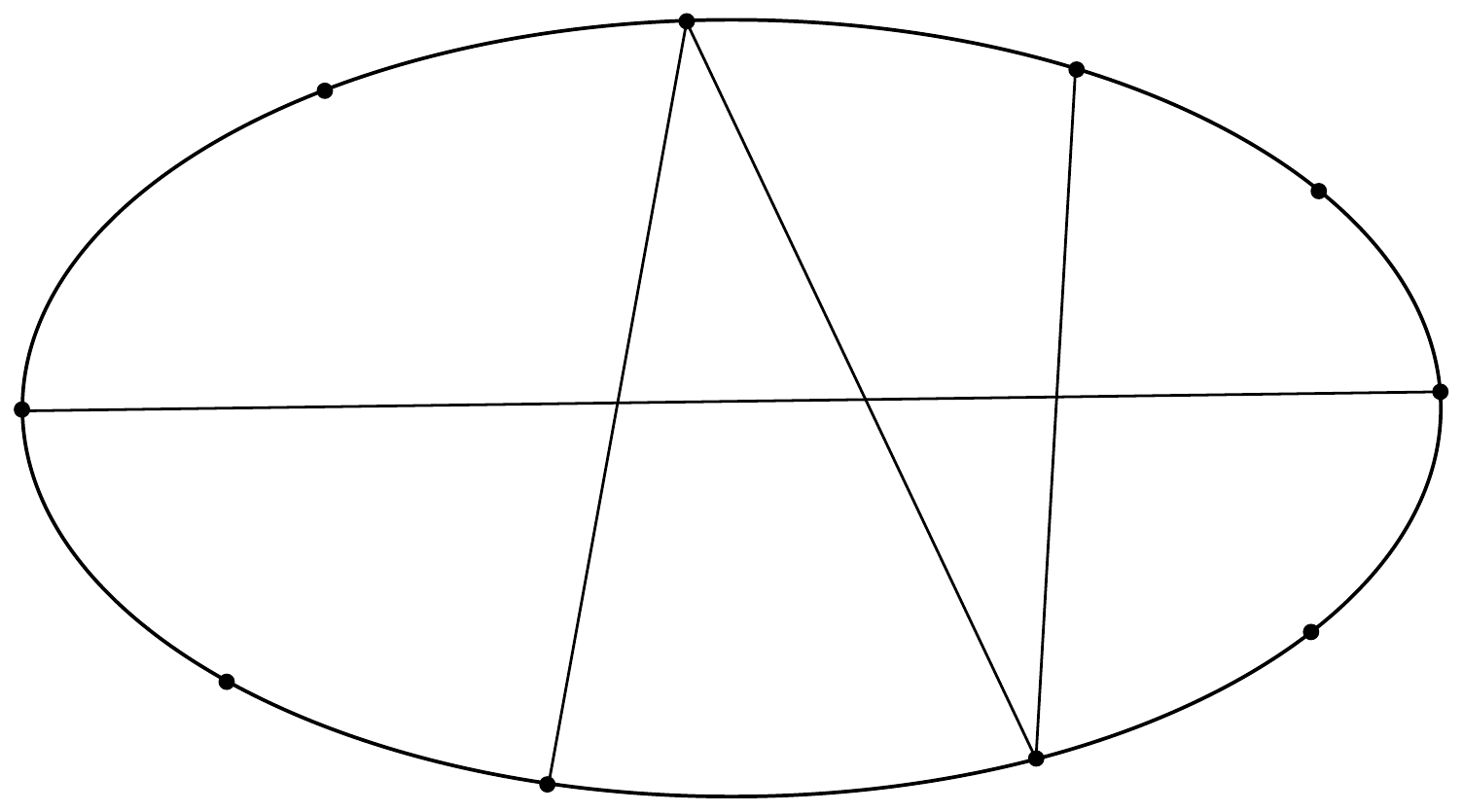}
\put (-442, 51){\makebox[0.7\textwidth][r]{$x^-$ }}
\put (-236, 53){\makebox[0.7\textwidth][r]{$x^+$ }}
\put (-398, 98){\makebox[0.7\textwidth][r]{$a''$ }}
\put (-353, 108){\makebox[0.7\textwidth][r]{$a$ }}
\put (-273, 102){\makebox[0.7\textwidth][r]{$\suc(a)$ }}
\put (-257, 82){\makebox[0.7\textwidth][r]{$a'$ }}
\put (-410, 6){\makebox[0.7\textwidth][r]{$b''$ }}
\put (-360, -7){\makebox[0.7\textwidth][r]{$\suc^{-1}(b)$ }}
\put (-302, -2){\makebox[0.7\textwidth][r]{$b$ }}
\put (-261, 15){\makebox[0.7\textwidth][r]{$b'$ }}
\caption{}
\label{Ztype}
\end{figure}

Proof of (2). Similarly, let $s'_1$, $s''_1$ be the closed subinterval of $s_1$ with endpoints $b$ and $x^+$, $x^-$ and $\suc^{-1}(b)$ respectively. For $p=0,\dots,n-2$ and any $r=1,\dots,n-p-1$, there exists $b'$ in $s'_1$ and $b''$ in $s''_1$ (see Figure \ref{Ztype}) so that 
\begin{eqnarray*}
L_{p,+}(\etd)&=&(\xi(b)^{(n-p-1)}+\xi(\suc(a))^{(p)})\cap\Pbbb(H)\\
&=&(\xi(a)^{(p)}+\xi(b)^{(n-p-r-1)}+\xi(\suc^{-1}(b))^{(r-1)}+\xi(b')^{(1)})\cap\Pbbb(H),\\
L_{p,-}(\etd)&=&(\xi(\suc^{-1}(b))^{(n-p-1)}+\xi(a)^{(p)})\cap\Pbbb(H)\\
&=&(\xi(a)^{(p)}+\xi(b)^{(n-p-r-1)}+\xi(\suc^{-1}(b))^{(r-1)}+\xi(b'')^{(1)})\cap\Pbbb(H).
\end{eqnarray*}
Hence, we have
\begin{eqnarray*}
&&(\xi(x^-)^{(1)},L_{p,-}(\etd),L_{p,+}(\etd),\xi(x^+)^{(1)})\\
&=&(\xi(x^-)^{(1)},\xi(b'')^{(1)},\xi(b')^{(1)},\xi(x^+)^{(1)})_{\xi(a)^{(p)}+\xi(b)^{(n-p-r-1)}+\xi(\suc^{-1}(b))^{(r-1)}}\\
&\geq&(\xi(a),\xi(\suc^{-1}(b)),\xi(b),\xi(\suc(a)))_{\xi(a)^{(p)}+\xi(b)^{(n-p-r-1)}+\xi(\suc^{-1}(b))^{(r-1)}}\\
&=&(\xi(\suc^{-1}(b)),\xi(a),\xi(\suc(a)),\xi(b))_{\xi(a)^{(p)}+\xi(b)^{(n-p-r-1)}+\xi(\suc^{-1}(b))^{(r-1)}}.
\end{eqnarray*}
where the last equality follows from (6) and (7) of Proposition \ref{basic cross ratio}.
\end{proof}

A proof similar to the one for Lemma \ref{Z length inequality} gives the following lemma.

\begin{lem}\label{S length inequality}
Let $\etd=\{a,b\}$ be an edge in $\widetilde{\Spc}$, and suppose $a$ lies in $s_0$ and $b$ lies in $s_1$. Then we have the following.
\begin{enumerate}
\item For all $p=1,\dots,n-1$ and for all $M\in\Mpc_p(a,b,\suc^{-1}(a))$, we have
\[l(c_p(\etd))\geq\log(\xi(b),\xi(\suc^{-1}(a)),\xi(a),\xi(\suc(b)))_M.\]
\item For all $p=0,\dots,n-2$ and for all $M\in\Mpc_{n-p-1}(b,a,\suc(b))$, we have 
\[l(c_p(\etd))\geq\log(\xi(b),\xi(\suc^{-1}(a)),\xi(a),\xi(\suc(b)))_M.\]
\end{enumerate}
\end{lem}

Now, we will define the quantity $K$. For any $[a,b]$ in $\Tpc$ that is not a closed leaf, choose a lift $\{a,b\}$ in $\widetilde{\Tpc}$ of $[a,b]$. Let $c$, $d$ be points in $\partial_\infty\Gamma$ such that $\{a,c\}$, $\{b,c\}$, $\{a,d\}$, $\{b,d\}$ lie in $\widetilde{\Tpc}$. For $p=1,\dots,n-2$ define
\begin{align*}
&K_p':=\max\{\log(\xi(d),\xi(a),\xi(c),\xi(b))_M:M\in\Mpc_p(a,b,c)\cup\Mpc_{n-p-1}(b,a,d)\}\\
&K_p'':=\max\{\log(\xi(b),\xi(d),\xi(a),\xi(c))_M:M\in\Mpc_p(a,b,d)\cup\Mpc_{n-p-1}(b,a,c)\}
\end{align*} and define 
\begin{align*}
&K_0':=\max\{\log(\xi(d),\xi(a),\xi(c),\xi(b))_M:M\in\Mpc_{n-1}(b,a,d)\}\\
&K_0'':=\max\{\log(\xi(b),\xi(d),\xi(a),\xi(c))_M:M\in\Mpc_{n-1}(b,a,c)\}\\
&K_{n-1}':=\max\{\log(\xi(d),\xi(a),\xi(c),\xi(b))_M:M\in\Mpc_{n-1}(a,b,c)\}\\
&K_{n-1}'':=\max\{\log(\xi(b),\xi(d),\xi(a),\xi(c))_M:M\in\Mpc_{n-1}(a,b,d)\}
\end{align*}
Finally, define
\[K[a,b]:=\min\bigg\{\frac{1}{n}\sum_{p=0}^{n-1}K_p',\frac{1}{n}\sum_{p=0}^{n-1}K_p''\bigg\}.\]
Note that if we switch the roles of $a$ with $b$, then the quantities $K_p'$ and $K_{n-p-1}''$ are switched. Also, switching $c$ with $d$ causes the quantities $K_p'$ and $K_p''$ to be switched. Thus, permuting $a$ and $b$ or permuting $c$ and $d$ leaves $K[a,b]$ invariant. Moreover, the $PSL(n,\Rbbb)$ invariance of the cross ratio implies that $K[a,b]$ does not depend on the choice of lift $\{a,b\}$ of $[a,b]$. This allows us to define 
\[K(\rho)=K:=\min_{[a,b]\in\Qpc}K[a,b].\]
Recall that $\Qpc=\Tpc\setminus\Ppc$ as defined in Section \ref{A special ideal triangulation}.

\begin{prop}\label{K lower bound}
For any $\etd$ in $\widetilde{\Bpc}$, 
\[\frac{1}{n}\sum_{p=0}^{n-1}l(c_p(\etd))\geq K.\]
\end{prop}

\begin{proof}
Let $\etd=[a,b]$ with $a$ in $s_0$ and $b$ in $s_1$. By Lemma \ref{Z length inequality}, we see that when $\etd$ is in $\Zpc$, by taking $d=\suc^{-1}(b)$, $c=\suc(a)$, we have $l(c_p(\etd))\geq K_p'$ for all $p=0,\dots,n-1$. Thus,  
\[\frac{1}{n}\sum_{p=0}^{n-1}l(c_p(\etd))\geq\frac{1}{n}\sum_{p=0}^{n-1}K_p'\geq K.\]
Similarly, by Lemma \ref{S length inequality}, we see that when $\etd$ is in $\Spc$, by taking $d=\suc^{-1}(a)$, $c=\suc(b)$, we have 
\[\frac{1}{n}\sum_{p=0}^{n-1}l(c_p(\etd))\geq\frac{1}{n}\sum_{p=0}^{n-1}K_p''\geq K.\]
\end{proof}

Next, denote
\[L(\rho):=\min\Big\{\frac{l_\rho(A)}{n}:A\in\Gamma_\Ppc\Big\}.\]
We will show how $L(\rho)$ provides a lower bound for winding $(p)$-subsegments. 

For the remainder of this section, let $\etd$ and $\etd'$ be consecutive elements in $\widetilde{\Bpc}$ with $\etd$ preceding $\etd'$, and let $e$, $e'$ be the equivalence classes in $\Bpc$ that contain $\etd$ and $\etd'$. Also, let $a^-$ and $a^+$ be the repelling and attracting fixed points for $A=A(\etd,\etd')$ respectively.

\begin{lem}\label{not same type}
If $e$ and $e'$ are not of the same type, then 
\[l(w_1(\etd,\etd'))+l(w_{n-2}(\etd,\etd'))\geq\max\{0,|t(e,e')|-1\}\cdot L(\rho).\]
\end{lem}

\begin{proof}
Let $r_0$, $r_1$ be the oriented subsegments of $\partial_\infty\Gamma$ with endpoints $a^-$, $a^+$, oriented from $a^-$ to $a^+$, so that the orientation on $r_0$ agrees with the clockwise orientation on $\partial_\infty\Gamma$. Since $e$ and $e'$ are not of the same type, either both of $x^-$, $x^+$ lie in $r_0$ or both of $x^-$, $x^+$ lie in $r_1$. If the former holds, we will show that $l(w_1(\etd,\etd'))\geq\max\{0,|t(e,e')|-1\}\cdot L(\rho)$ and if the latter holds, we will show that $l(w_{n-2}(\etd,\etd'))\geq\max\{0,|t(e,e')|-1\}\cdot L(\rho)$. 

By taking inverses, we can assume without loss of generality that $a^-$, $x^-$, $x^+$, $a^+$ lie on $r_0$ in that clockwise order. In this case, $t(e,e')\geq 0$, and we need to show that $l(w_1(\etd,\etd'))\geq\max\{0,t(e,e')-1\}\cdot L(\rho)$. Clearly, this holds when $t(e,e')\leq 1$ so we will assume that $t(e,e')\geq 2$. Let $c$ be the point in $r_0$ so that $\{c,c'\}$ is an edge in $\Epc_A$ and $A^{-1}\cdot c$, $x^-$, $c$ lie in $r_0$ in that order. This then implies that $A^{t(e,e')-1}\cdot c$, $x^+$, $A^{t(e,e')}\cdot c$ lie in $r_0$ in that order. (See Figure \ref{different type}.) For any $p=0,\dots,n-2$, define
\begin{align*}
&\alpha_p:=\Pbbb(\xi(a^-)^{(p)}+\xi(a^+)^{(n-p-2)}+\xi(c)^{(1)})\cap\Pbbb(H),\\
&\beta_p:=\Pbbb(\xi(a^-)^{(p)}+\xi(a^+)^{(n-p-2)}+\xi(A^{t(e,e')-1}\cdot c)^{(1)})\cap\Pbbb(H).
\end{align*}

\begin{figure}
\includegraphics[scale=0.7]{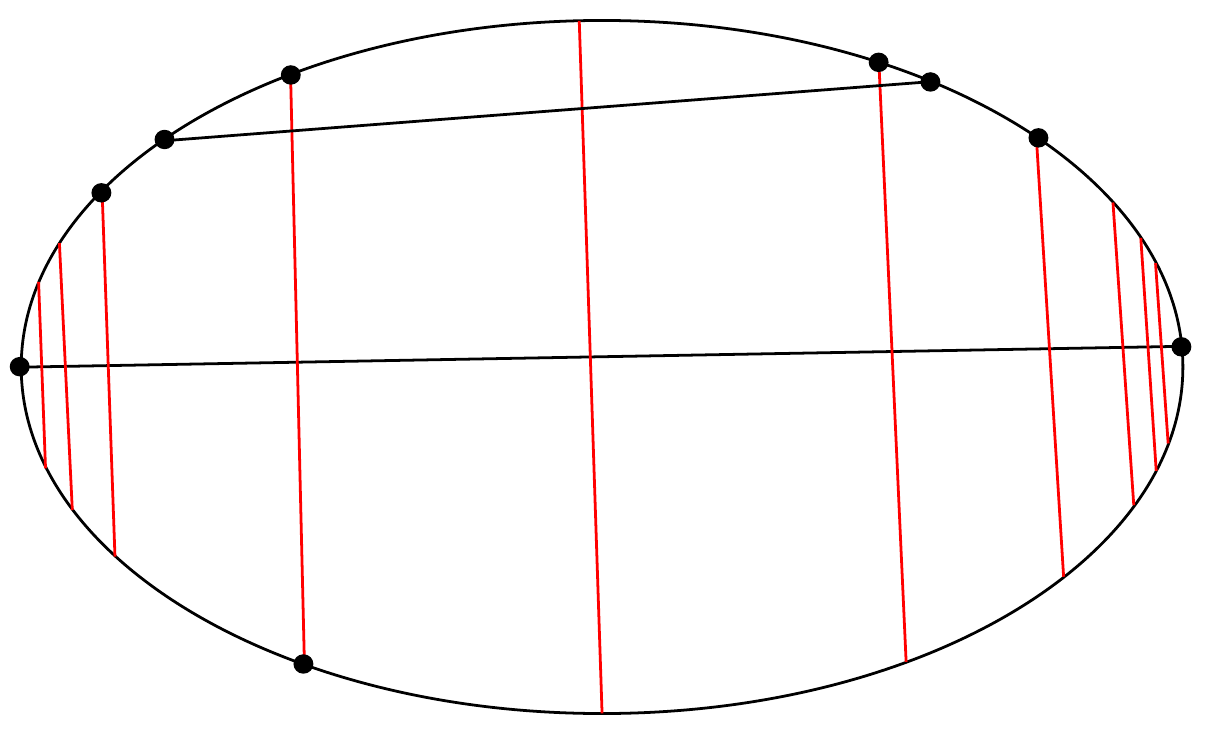}
\put (-432, 3){\makebox[0.7\textwidth][r]{$c'$ }}
\put (-490, 72){\makebox[0.7\textwidth][r]{$a^-$ }}
\put (-470, 114){\makebox[0.7\textwidth][r]{$A^{-1}\cdot c$ }}
\put (-450, 125){\makebox[0.7\textwidth][r]{$x^-$ }}
\put (-432, 137){\makebox[0.7\textwidth][r]{$c$ }}
\put (-265, 140){\makebox[0.7\textwidth][r]{$A^{t(e,e')-1}\cdot c$ }}
\put (-288, 132){\makebox[0.7\textwidth][r]{$x^+$ }}
\put (-238, 120){\makebox[0.7\textwidth][r]{$A^{t(e,e')}\cdot c$ }}
\put (-237, 77){\makebox[0.7\textwidth][r]{$a^+$ }}
\caption{$e$, $e'$ are not of the same type.}
\label{different type}
\end{figure}
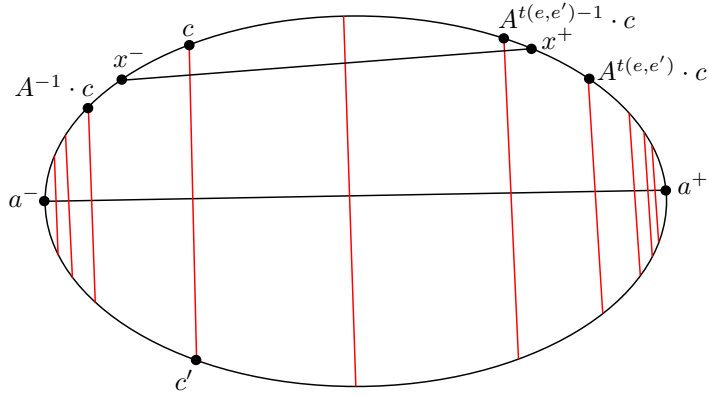

Now, let $b$ be the vertex of $\widetilde{\Tpc}$ that lies on $r_0$ between $A^{-1}\cdot c$ and $c$, so that $\{a^-,b\}$ is an edge of $\widetilde{\Tpc}$. Observe that all the edges in $\widetilde{\Tpc}$ with $a^-$ as an endpoint has its other endpoint in $\langle A\rangle\cdot c\cup\langle A\rangle\cdot b$. Hence, if $b$ lies between $A^{-1}\cdot c$ and $x^-$, then 
\[\etd=\{a^-,c\},\ \suc^{-1}(\etd)=\{b,c\},\ \suc(\etd)=\{a^-,A\cdot b\},\]
which implies that $L_{1,-}(\etd)=\Pbbb(\xi(b)^{(n-2)}+\xi(c)^{(1)})\cap\Pbbb(H)$. On the other hand, if $b$ lies between $x^-$ and $c$, then
\[\etd=\{a^-,b\},\ \suc^{-1}(\etd)=\{A^{-1}\cdot c,b\},\ \suc(\etd)=\{a^-,c\}, \]
which implies that $L_{1,-}(\etd)=\Pbbb(\xi(b)^{(1)}+\xi(A^{-1}\cdot c)^{(n-2)})\cap\Pbbb(H)$.

By a similar reasoning, if $A^{t(e,e')}\cdot b$ lies between $A^{t(e,e')-1}\cdot c$ and $x^+$,  then $L_{1,+}(\etd')=\Pbbb(\xi(A^{t(e,e')}\cdot b)^{(1)}+\xi(A^{t(e,e')}\cdot c)^{(n-2)})\cap\Pbbb(H)$ and if $A^{t(e,e')}\cdot b$ lies between $x^+$ and $A^{t(e,e')}\cdot c$, then $L_{1,+}(\etd')=\Pbbb(\xi(A^{t(e,e')}\cdot b)^{(n-2)}+\xi(A^{t(e,e')-1}\cdot c)^{(1)})\cap\Pbbb(H)$.

In any case, by (2) of Lemma \ref{important lemma}, we see that $L_{1,-}(\etd)$, $\alpha_p$, $\beta_p$, $L_{1,+}(\etd)$ lie in the same subsegment of $\Pbbb(H)$ with endpoints $\xi(x^-)^{(1)}$, $\xi(x^+)^{(1)}$, in that order. Proposition \ref{cross ratio configuration} thus implies that for all $p=0,\dots,n-2$, we have
\[l(w_1(\etd,\etd'))\geq\log(\xi(x^-)^{(1)},\alpha_p,\beta_p,\xi(x^+)^{(1)}).\]

Also, one can compute that
\begin{eqnarray*}
&&(\xi(x^-)^{(1)},\alpha_p,\beta_p,\xi(x^+)^{(1)})\\
&=&(\xi(x^-)^{(1)},\alpha_p,\beta_p,\xi(x^+)^{(1)})_{\xi(a^-)^{(p)}+\xi(a^+)^{(n-p-2)}}\\
&\geq&(\xi(a^-),\xi(c),\xi(A^{t(e,e')-1}\cdot c),\xi(a^+))_{\xi(a^-)^{(p)}+\xi(a^+)^{(n-p-2)}}\\
&=&\big(e^{\lambda_{n-p-1}(A)-\lambda_{n-p}(A)}\big)^{t(e,e')-1},
\end{eqnarray*}
where the inequality is a consequence of Proposition \ref{useful cross ratio inequalities}, and the last equality is a special case of Proposition \ref{cross ratio and length}.

By taking the product of these inequalities over $p=0,\dots,n-2$, and then taking logarithm, we obtain
\[(n-1)\cdot l(w_1(\etd,\etd'))\geq\big(t(e,e')-1\big)\cdot\big(\lambda_1(A)-\lambda_n(A)\big).\]
This implies the lemma.
\end{proof}

\begin{lem}\label{same type}
If $e$ and $e'$ are of the same type, then 
\[\frac{1}{n}\sum_{p=0}^{n-1}l(w_p(\etd,\etd'))\geq\max\{0,|t(e,e')|-2\}\cdot L(\rho).\]
\end{lem}

\begin{proof}
This inequality clearly holds when $t(e,e')=-2,-1,0,1,2$, so for the rest of the proof, we will assume that $|t(e,e')|\geq 3$. As before, let $r_0$, $r_1$ be the two subsegments of $\partial_\infty\Gamma$ with endpoints $a^-$ and $a^+$, oriented from $a^-$ to $a^+$, and such that the orientation on $r_0$ agrees with the orientation on $\partial_\infty\Gamma$. 

By taking inverses, we can assume without loss of generality that $x^-$ lies in $r_0$ and $x^+$ lies in $r_1$. Let $\{c_0,c_1\}$ be the edge in $\Epc_{A}$ so that $c_0$ lies in $s_0$ and $A^{-1}\cdot c_0$ lies in $s_1$. This implies that $A^{t(e,e')-1}\cdot c_1$ lies in $s_1$ and $A^{t(e,e')}\cdot c_1$ lies in $s_0$. For any $p\in\{0,\dots,n-1\}$, let 
\begin{align*}
&\alpha_p:=\Pbbb(\xi(A^{-1}\cdot c_0)^{(n-p-1)}+\xi(c_0)^{(p)})\cap\Pbbb(H),\\
&\beta_p:=\Pbbb(\xi(A^{t(e,e')-1}\cdot c_1)^{(n-p-1)}+\xi(A^{t(e,e')}\cdot c_1)^{(p)})\cap\Pbbb(H).
\end{align*}
Using a similar argument as in the proof of Lemma \ref{not same type}, we have
\[l(w_p(\etd,\etd'))\geq\log(\xi(x^-)^{(1)},\alpha_p,\beta_p,\xi(x^+)^{(1)}).\]

Choose a normalization so that $\xi(a^-)^{(n-k+1)}\cap\xi(a^+)^{(k)}=[e_k]$ for all $k=1,\dots,n$ and $\xi(c_0)^{(1)}=[e_1+\dots+e_n]$. In this normalization, $\rho(A)$ is the projectivization of a diagonal matrix and 
\[\xi(c_1)^{(1)}=\bigg[\sum_{i=1}^n\delta_ie_i\bigg]\]
for some real numbers $\delta_i$. 

From here, the proof will proceed in two cases, depending on the sign of $t(e,e')$.

\begin{figure}
\includegraphics[scale=0.7]{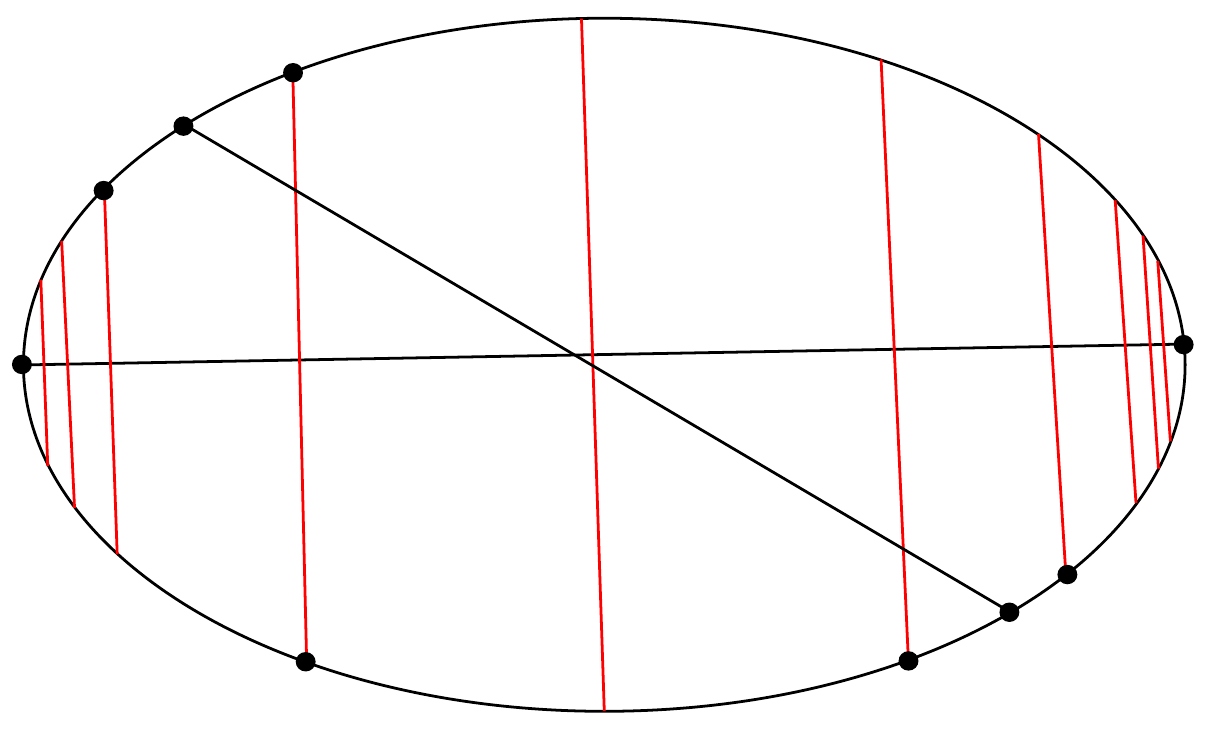}
\put (-429, 7){\makebox[0.7\textwidth][r]{$c_1$ }}
\put (-490, 72){\makebox[0.7\textwidth][r]{$a^-$ }}
\put (-470, 115){\makebox[0.7\textwidth][r]{$A^{-1}\cdot c_0$ }}
\put (-450, 125){\makebox[0.7\textwidth][r]{$x^-$ }}
\put (-432, 138){\makebox[0.7\textwidth][r]{$c_0$ }}
\put (-265, 3){\makebox[0.7\textwidth][r]{$A^{t(e,e')-1}\cdot c_1$ }}
\put (-278, 16){\makebox[0.7\textwidth][r]{$x^+$ }}
\put (-230, 25){\makebox[0.7\textwidth][r]{$A^{t(e,e')}\cdot c_1$ }}
\put (-237, 77){\makebox[0.7\textwidth][r]{$a^+$ }}
\caption{$e$, $e'$ are of the same type and $t(e,e')>0$.}
\label{t>0}
\end{figure}
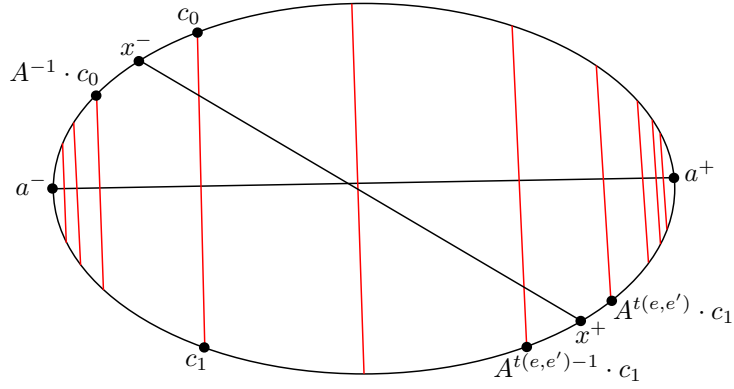

{\bf Case 1:} Suppose that $t(e,e')>0$. (See Figure \ref{t>0}.) For $p=1,\dots,n-1$, define
\begin{align*}
&\alpha_p':=\Pbbb(\xi(a^-)^{(n-p-1)}+\xi(a^+)^{(p-1)}+\xi(c_0)^{(1)})\cap\Pbbb(H),\\
&\beta_p':=\Pbbb(\xi(a^-)^{(n-p-1)}+\xi(a^+)^{(p)})\cap\Pbbb(H).
\end{align*}
By (2) of Lemma \ref{important lemma}, we see that $\alpha_p$, $\alpha_p'$, $\beta_p'$, $\beta_p$ lie on $\Pbbb(H)$ in that order. Hence, we can apply Proposition \ref{useful cross ratio inequalities} to see that
\begin{eqnarray*}
&&(\xi(x^-)^{(1)},\alpha_p,\beta_p,\xi(x^+)^{(1)})\\
&\geq&(\xi(x^-)^{(1)},\alpha_p',\beta_p',\xi(x^+)^{(1)})\\
&=&(\xi(x^-)^{(1)},\alpha_p',\beta_p',\xi(x^+)^{(1)})_{\xi(a^-)^{(n-p-1)}+\xi(a^+)^{(p-1)}}\\
&\geq&(\xi(A^{-1}\cdot c_0),\xi(c_0),\xi(a^+),\xi(A^{t(e,e')-1}\cdot c_1))_{\xi(a^-)^{(n-p-1)}+\xi(a^+)^{(p-1)}}\\
&=&\frac{1}{1-e^{\lambda_{p+1}(A)-\lambda_p(A)}}\cdot\bigg(1-\frac{\delta_p}{\delta_{p+1}}\cdot\Big(e^{\lambda_p(A)-\lambda_{p+1}(A)}\Big)^{t(e,e')-1}\bigg).
\end{eqnarray*}

Observe that for any integer $k$, (1) of Proposition \ref{useful cross ratio inequalities} says that 
\[(\xi(A^{-1}\cdot c_0),\xi(c_0),\xi(a^+),\xi(A^{k-1}\cdot c_1))_{\xi(a^-)^{(n-p-1)}+\xi(a^+)^{(p-1)}}>1.\] 
This then implies that $\frac{\delta_{p}}{\delta_{p+1}}<0$.
Thus, we have
\[(\xi(x^-)^{(1)},\alpha_p,\beta_p,\xi(x^+)^{(1)})\geq\bigg|\frac{\delta_p}{\delta_{p+1}}\bigg|\cdot\Big(e^{\lambda_p(A)-\lambda_{p+1}(A)}\Big)^{t(e,e')-1}.\]
By doing this for all $p=1,\dots,n-1$ and taking the product, we get that
\[\prod_{p=1}^{n-1}(\xi(x^-)^{(1)},\alpha_p,\beta_p,\xi(x^+)^{(1)})\geq\bigg|\frac{\delta_1}{\delta_n}\bigg|\cdot\Big(e^{\lambda_1(A)-\lambda_n(A)}\Big)^{t(e,e')-1},\]
which implies 
\begin{equation}\label{type 2 inequality 1}
\sum_{p=1}^{n-1} l(w_p(\etd,\etd'))\geq\log\bigg|\frac{\delta_1}{\delta_n}\bigg|+\big(t(e,e')-1\big)\cdot\big(\lambda_1(A)-\lambda_n(A)\big).
\end{equation}

{\bf Case 2:} Suppose that $t(e,e')<0$. (See Figure \ref{t<0}.) For all $0\leq p\leq n-2$, define
\begin{align*}
&\alpha_p'':=\Pbbb(\xi(a^-)^{(n-p-2)}+\xi(a^+)^{(p)}+\xi(A^{-1}\cdot c_0)^{(1)})\cap\Pbbb(H),\\
&\beta_p'':=\Pbbb(\xi(a^-)^{(n-p-1)}+\xi(a^+)^{(p)})\cap\Pbbb(H).
\end{align*}

As before, (2) of Lemma \ref{important lemma}, implies that $\alpha_p$, $\alpha_p''$, $\beta_p''$, $\beta_p$ lie on $\Pbbb(H)$ in that order, thus allowing us to use Proposition \ref{cross ratio configuration} and Proposition \ref{useful cross ratio inequalities} to compute

\begin{figure}
\includegraphics[scale=0.7]{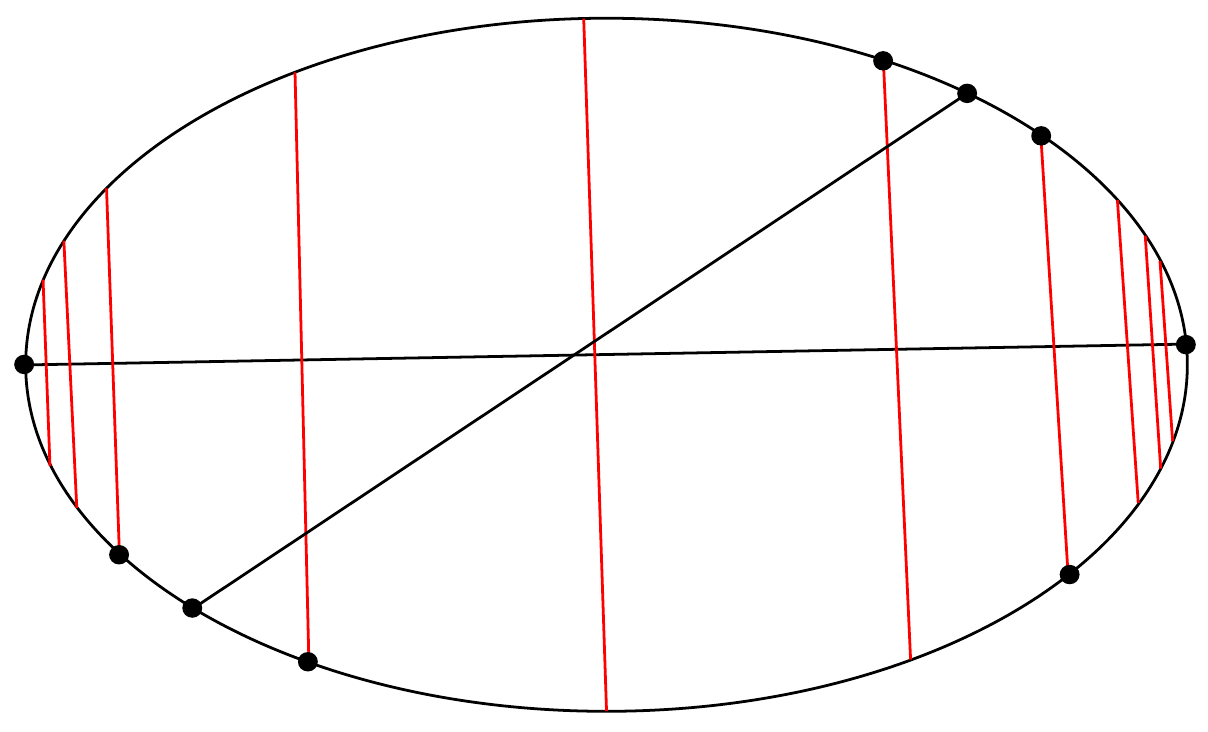}
\put (-425, 5){\makebox[0.7\textwidth][r]{$A^{t(e,e')}\cdot c_1$ }}
\put (-454, 16){\makebox[0.7\textwidth][r]{$x^+$ }}
\put (-465, 27){\makebox[0.7\textwidth][r]{$A^{t(e,e')-1}\cdot c_1$ }}
\put (-490, 72){\makebox[0.7\textwidth][r]{$a^-$ }}
\put (-287, 140){\makebox[0.7\textwidth][r]{$A^{-1}\cdot c_0$ }}
\put (-281, 130){\makebox[0.7\textwidth][r]{$x^-$ }}
\put (-272, 123){\makebox[0.7\textwidth][r]{$c_0$ }}
\put (-237, 77){\makebox[0.7\textwidth][r]{$a^+$ }}
\put (-265, 27){\makebox[0.7\textwidth][r]{$c_1$ }}
\caption{$e$, $e'$ are of the same type and $t(e,e')<0$.}
\label{t<0}
\end{figure}

\begin{eqnarray*}
&&(\xi(x^-)^{(1)},\alpha_p,\beta_p,\xi(x^+)^{(1)})\\
&\geq&(\xi(x^-)^{(1)},\alpha_p'',\beta_p'',\xi(x^+)^{(1)})\\
&=&(\xi(x^-)^{(1)},\alpha_p'',\beta_p'',\xi(x^+)^{(1)})_{\xi(a^-)^{(n-p-2)}+\xi(a^+)^{(p)}}\\
&\geq&(\xi(c_0),\xi(A^{-1}\cdot c_0),\xi(a^-),\xi(A^{t(e,e')}\cdot c_1))_{\xi(a^-)^{(n-p-2)}+\xi(a^+)^{(p)}}\\
&=&\frac{1}{1-e^{\lambda_{p+2}(A)-\lambda_{p+1}(A)}}\cdot\bigg(1-\frac{\delta_{p+2}}{\delta_{p+1}}\cdot\Big(e^{\lambda_{p+2}(A)-\lambda_{p+1}(A)}\Big)^{t(e,e')+1}\bigg).
\end{eqnarray*}
As before, $\frac{\delta_{p+2}}{\delta_{p+1}}<0$, so
\[(\xi(x^-)^{(1)},\alpha_p,\beta_p,\xi(x^+)^{(1)})\geq\bigg|\frac{\delta_{p+2}}{\delta_{p+1}}\bigg|\cdot\Big(e^{\lambda_{p+2}(A)-\lambda_{p+1}(A)}\Big)^{t(e,e')+1}.\]
If we take the product of all these inequalities for $p=0,\dots,n-2$ and then take logarithm, we obtain the inequality
\begin{equation}\label{type 2 inequality 2}
\sum_{p=0}^{n-2} l(w_p(\etd,\etd'))\geq\log\bigg|\frac{\delta_n}{\delta_1}\bigg|+\big(-t(e,e')-1\big)\cdot\big(\lambda_1(A)-\lambda_n(A)\big)
\end{equation}

By the way we defined a mesh (see (\ref{mesh inequality})), we have that 
\[1\leq\bigg|\frac{\delta_1}{\delta_n}\bigg|=\bigg|(\xi(a^+)^{(1)},\xi(c_0)^{(1)},\xi(c_1)^{(1)},\xi(a^-)^{(1)})_{\xi(a^+)^{(n-1)}\cap\xi(a^-)^{(n-1)}}\bigg|< e^{\lambda_1(A)-\lambda_n(A)}.\] 
Thus, both inequalities (\ref{type 2 inequality 1}) and (\ref{type 2 inequality 2}) imply that 
\[\sum_{p=0}^{n-1} l(w_p(\etd,\etd'))\geq\big(|t(e,e')|-2\big)\cdot\big(\lambda_1(A)-\lambda_n(A)\big).\]
The inequality in the lemma follows immediately from this.
\end{proof}

In order to emphasize that the lower bound in Theorem \ref{length lower bound} depends only on the combinatorial description $\psi(X)$ of $X$, we will use the notation 
\[r(\psi(X)):=|\Bpc|\text{ and }s(\psi(X)):=\underset{(e,e')\in\widetilde{\Dpc}_1\cup\widetilde{\Dpc}_2}{\sum}\max\{0,|t(e,e')|-2\}.\]
As a corollary of the estimates in Proposition \ref{segment lower bound}, Proposition \ref{K lower bound}, Lemma \ref{not same type} and Lemma \ref{same type}, we obtain the following theorem.

\begin{thm}\label{length lower bound}
Pick any $\rho$ in $Hit_n(S)$ and any $X$ in $\Gamma$ such that $r(\psi(X))\neq 0$. Then
\[l_{\rho}(X)\geq r(\psi(X))\cdot\frac{K(\rho)}{11}+s(\psi(X))\cdot\frac{L(\rho)}{11}.\]
\end{thm}

%%%%%%%%%%%%%%%%%%%%%%%%%%%%%%%%%%%%%%%%%%%%%%%%%%
%%%%%%%%%%%%%%%%%%%%%%%%%%%%%%%%%%%%%%%%%%%%%%%%%%
\section{Degeneration along internal sequences}\label{Degeneration along internal sequences}
%%%%%%%%%%%%%%%%%%%%%%%%%%%%%%%%%%%%%%%%%%%%%%%%%%
%%%%%%%%%%%%%%%%%%%%%%%%%%%%%%%%%%%%%%%%%%%%%%%%%%
In this section, we will use the analysis in Section \ref{Lower bound for lengths of closed curves} to prove Theorem \ref{main theorem}. 

%%%%%%%%%%%%%%%%%%%%%%%%%%%%%%%%%%%%%%%%%%%%%%%%%%
\subsection{Proof of (1) of main theorem}
%%%%%%%%%%%%%%%%%%%%%%%%%%%%%%%%%%%%%%%%%%%%%%%%%%
Observe that $X$ is an element of $\Gamma$ such that $X\neq A^k$ for any $A$ in $\Gamma_\Ppc$ and any integer $k$ if and only if $r(\psi(X))\neq 0$. Theorem \ref{length lower bound} then implies that for any $\rho$ in $Hit_n(S)$, $K(\rho)$ is a lower bound for $\Theta(\rho)$. Thus, to prove (1) of Theorem \ref{main theorem}, it is sufficient to prove the following theorem.

\begin{thm}\label{K to infinity}
Let $\{\rho_i\}_{i=1}^\infty$ be an internal sequence in $Hit_n(S)$. Then 
\[\lim_{i\to\infty}K(\rho_i)=\infty.\]
\end{thm}

We start by using the closed leaf equalities (\ref{alpha equation}), (\ref{beta equation}) and (\ref{gamma equation}) to prove the following lemma.

\begin{lem}\label{triangle linear relations}
Let $\{\rho_i\}_{i=1}^\infty$ be an internal sequence such that the shear and triangle invariants converge (possibly to $\infty$ or $-\infty$) along $\{\rho_i\}_{i=1}^\infty$. Then for any $j=1,\dots,2g-2$, one of the following hold:
\begin{enumerate}
\item There is a pair of shear invariants, call them $\sigma_{1,[a_j^-,b_j^-]}$ and $\sigma_{2,[a_j^-,b_j^-]}$ for the edge $[a_j^-,b_j^-]$, so that 
\[\lim_{i\to\infty}\sigma_{1,[a_j^-,b_j^-]}(\rho_i)=\infty\text{ and }\lim_{i\to\infty}\sigma_{2,[a_j^-,b_j^-]}(\rho_i)=-\infty.\]
\item There is some number $z_0\in\{1,\dots,n-2\}$ so that there are triangle invariants $\tau_1,\tau_2$ in 
$\{\tau_{(x,y,z),j}:z=z_0\}\cup\{\tau'_{(x,y,z),j}:z=z_0\}$ satisfying
\[\lim_{i\to\infty}\tau_1(\rho_i)=\infty\text{ and }\lim_{i\to\infty}\tau_2(\rho_i)=-\infty,\]
and the triangle invariants in $\{\tau_{(x,y,z),j}:z<z_0\}\cup\{\tau'_{(x,y,z),j}:z<z_0\}$ 
are bounded above and below along $\{\rho_i\}_{i=1}^\infty$.
\end{enumerate}
\end{lem}

\begin{proof}
First, we prove the claim that if (1) does not hold, then there must be some triangle invariant $\tau_2$ in $\Apc_j\cup\Apc_j'$ so that 
\[\lim_{i\to\infty}\tau_2(\rho_i)=-\infty.\]
Suppose that every triangle invariant in $\Apc_j\cup\Apc_j'$ is bounded below by a real number when evaluated along $\{\rho_i\}_{i=1}^\infty$. Equations (\ref{alpha equation}), (\ref{beta equation}), (\ref{gamma equation}) and the definition of an internal sequence then imply that there is some shear invariant that is converging to $-\infty$ along $\{\rho_i\}_{i=1}^\infty$. Assume without loss of generality that this shear invariant is associated to the edge $[b_j^-,c_j^-]$, and denote it by $\sigma_{2,[b_j^-,c_j^-]}$. 

By Equation (\ref{right line}), we know that there is some other shear invariant $\sigma_{1,[b_j^-,c_j^-]}$ associated to the edge $[b_j^-,c_j^-]$ so that 
\[\lim_{i\to\infty}\sigma_{1,[b_j^-,c_j^-]}(\rho_i)=\infty.\]
Then, by Equation (\ref{gamma equation}) and the assumption that every triangle invariant for $P_j$ is bounded below by a real number when evaluated along $\{\rho_i\}_{i=1}^\infty$, we can deduce that there is some shear invariant $\sigma_{2,[c_j^-,a_j^-]}$ associated to the edge $[c_j^-,a_j^-]$ so that 
\[\lim_{i\to\infty}\sigma_{2,[c_j^-,a_j^-]}(\rho_i)=-\infty.\]
Equation (\ref{left line}) now implies that there is some other shear invariant $\sigma_{1,[c_j^-,a_j^-]}$ associated to the edge $[c_j^-,a_j^-]$ so that 
\[\lim_{i\to\infty}\sigma_{1,[c_j^-,a_j^-]}(\rho_i)=\infty.\]
Using the same arguments as above, Equations (\ref{alpha equation}) and (\ref{bottom line}) together imply that there are shear invariants $\sigma_{1,[a_j^-,b_j^-]}$ and $\sigma_{2,[a_j^-,b_j^-]}$ associated to the edge $[b_j^-,c_j^-]$ so that 
\[\lim_{i\to\infty}\sigma_{1,[a_j^-,b_j^-]}(\rho_i)=\infty\text{ and }\lim_{i\to\infty}\sigma_{2,[a_j^-,b_j^-]}(\rho_i)=-\infty.\]

We have thus proven that under the hypothesis that every triangle invariant in $\Apc_j\cup\Apc_j'$ is bounded below by a real number when evaluated along $\{\rho_i\}_{i=1}^\infty$, the three edges $[a_j^-,b_j^-]$, $[b_j^-,c_j^-]$ and $[a_j^-,c_j^-]$ each have a pair shear invariants associated to them with the property that one of them converges to $-\infty$ while the other converges to $\infty$. In particular, (1) holds. The contrapositive of this is the claim.

Next, we prove that if (1) does not hold, then (2) must hold. Suppose (1) does not hold. By Equation (\ref{bottom line}), we see that every shear invariant for $[a_j^-,b_j^-]$ is bounded below along $\{\rho_i\}_{i=1}^\infty$ if and only if every shear invariant for $[a_j^-,b_j^-]$ is bounded above along $\{\rho_i\}_{i=1}^\infty$. Hence, the shear invariants for the edge $[a_j^-,b_j^-]$ are bounded both above and below.

Also, we can assume without loss of generality that $\tau_2$ in the above claim lies in 
\[\{\tau_{(x,y,z),j}:z=z_0\}\cup\{\tau'_{(x,y,z),j}:z=z_0\}\] 
for some $z_0$ so that the triangle invariants in 
\[\{\tau_{(x,y,z),j}:z<z_0\}\cup\{\tau'_{(x,y,z),j}:z<z_0\}\]
are bounded below along $\{\rho_i\}_{i=1}^\infty$. Equations (\ref{alpha equation}) and (\ref{beta equation}) then imply respectively that the sequences 
\[\{\sigma_{(n-z_0,0,z_0),j}(\rho_i)\}_{i=1}^\infty\text{ and }\{\sigma_{(0,n-z_0,z_0),j}(\rho_i)\}_{i=1}^\infty\] are bounded above. Thus, by Equation (\ref{gamma equation}), there exists a triangle invariant $\tau_1$ in $\{\tau_{(x,y,z),j}:z=z_0\}\cup\{\tau'_{(x,y,z),j}:z=z_0\}$ with 
\[\lim_{i\to\infty}\tau_1(\rho_i)=\infty.\]

Finally, suppose for contradiction that there is some $z_0'<z_0$ with the property that there is some triangle invariant $\tau_1'$ in $\{\tau_{(x,y,z),j}:z=z_0'\}\cup\{\tau'_{(x,y,z),j}:z=z_0'\}$ such that
\[\lim_{i\to\infty}\tau_1'(\rho_i)=\infty.\]
We can assume that $z_0'$ is the minimal such number. A similar proof as the one given above then implies that there is some $\tau_2'$ in $\{\tau_{(x,y,z),j}:z=z_0'\}\cup\{\tau'_{(x,y,z),j}:z=z_0'\}$ with 
\[\lim_{i\to\infty}\tau_2'(\rho_i)=-\infty.\]
However, this contradicts the definition of $z_0$.
\end{proof}

Armed with Lemma \ref{triangle linear relations}, we are now ready to prove Theorem \ref{K to infinity}.

\begin{proof}[Proof of Theorem \ref{K to infinity}]
For any subsequence of $\{\rho_i\}_{i=1}^\infty$, choose a further subsequence, denoted $\{\rho_{i_k}\}_{k=1}^\infty$, so that the shear and triangle invariants converge (possibly to $\infty$ or $-\infty$) along $\{\rho_{i_k}\}_{i=1}^\infty$. It is sufficient to show that 
\[\lim_{k\to\infty}K(\rho_{i_k})=\infty.\]
For the rest of the proof, we will simplify notation by relabeling the sequence $\{\rho_{i_k}\}_{k=1}^\infty$ as $\{\rho_i\}_{i=1}^\infty$.

Let $\{a,b\}$ be any edge in $\widetilde{\Tpc}$ that is not a closed leaf, and let $c,d$ be the unique pair of points in $\partial_\infty\Gamma$ such that $\{a,c\}$, $\{b,c\}$, $\{a,d\}$, $\{b,d\}$ are all edges in $\widetilde{\Tpc}$. It is sufficient to show that both of the following hold:
\begin{enumerate}[(a)]
\item There is a sequence $\{M_i\}_{i=1}^\infty$ of $(n-2)$-dimensional subspaces in $\Rbbb^n$ so that $M_i$ is an element of  
\[\bigg(\bigcup_{p=1}^{n-1}\Mpc^{\xi_i}_p(a,b,c)\bigg)\cup\bigg(\bigcup_{p=0}^{n-2}\Mpc^{\xi_i}_{n-p-1}(b,a,d)\bigg)\]
for all $i$, and
\[\lim_{i\to\infty}\log(\xi_i(d),\xi_i(a),\xi_i(c),\xi_i(b))_{M_i}=\infty.\]
\item There is a sequence $\{M_i\}_{i=1}^\infty$ of $(n-2)$-dimensional subspaces in $\Rbbb^n$ so that $M_i$ is an element of 
\[\bigg(\bigcup_{p=1}^{n-1}\Mpc^{\xi_i}_p(a,b,d)\bigg)\cup\bigg(\bigcup_{p=0}^{n-2}\Mpc^{\xi_i}_{n-p-1}(b,a,c)\bigg)\]
for all $i$, and
\[\lim_{i\to\infty}\log(\xi_i(b),\xi_i(d),\xi_i(a),\xi_i(c))_{M_i}=\infty.\]
\end{enumerate}

By Lemma \ref{symmetry}, we can assume without loss of generality that the edge $a=a_j^-$, $b=b_j^-$, $c=c_j^-$ and $d=A_j\cdot c_j^-$ for some pair of pants $P_j$. In this setting, either (1) or (2) of Lemma \ref{triangle linear relations} must hold.

Suppose (1) holds. Let $p,q\in\{1,\dots,n-2\}$ be such that 
\[\lim_{i\to\infty}\sigma_{(q,n-q,0),j}(\rho_i)=\infty\text{ and }\lim_{i\to\infty}\sigma_{(p,n-p,0),j}(\rho_i)=-\infty.\] 
Then
\begin{eqnarray*}
&&(\xi_i(d),\xi_i(a),\xi_i(c),\xi_i(b))_{\xi_i(a)^{(p-1)}+\xi_i(b)^{(n-p-1)}}\\
&=&1-(\xi_i(a),\xi_i(d),\xi_i(c),\xi_i(b))_{\xi_i(a)^{(p-1)}+\xi_i(b)^{(n-p-1)}}\\
&=&1-\frac{1}{(\xi_i(a),\xi_i(c),\xi_i(d),\xi_i(b))_{\xi_i(a)^{(p-1)}+\xi_i(b)^{(n-p-1)}}}\\
&=&1+e^{-\sigma_{(p,n-p,0),j}(\rho_i)}\to\infty\text{ as }i\to\infty
\end{eqnarray*}
and
\begin{eqnarray*}
&&(\xi_i(b),\xi_i(d),\xi_i(a),\xi_i(c))_{\xi_i(a)^{(q-1)}+\xi_i(b)^{(n-q-1)}}\\
&=&1-(\xi_i(b),\xi_i(d),\xi_i(c),\xi_i(a))_{\xi_i(a)^{(q-1)}+\xi_i(b)^{(n-q-1)}}\\
&=&1+e^{\sigma_{(q,n-q,0),j}(\rho_i)}\to\infty\text{ as }i\to\infty
\end{eqnarray*}
so (a) and (b) hold. 

Next, suppose that (2) holds. By Theorem \ref{Bonahon-Dreyer} and Lemma \ref{triangle lemma}, we know that there is some representation $\rho$ in $Hit_n(S)$ with corresponding Frenet curve $\xi$ such that under a suitably chosen normalization, the following hold.
\begin{itemize}
\item $\xi_i(a_j^-)^{(k)}=\xi(a_j^-)^{(k)}$ and $\xi_i(b_j^-)^{(k)}=\xi(b_j^-)^{(k)}$ for all $k=1,\dots,n-1$ and for all $i$, 
\item $\xi_i(c_j^-)^{(1)}=\xi(c_j^-)^{(1)}$ for all $i$,
\item $\underset{i\to\infty}{\lim}\xi_i(c_j^-)^{(k)}=\xi(c_j^-)^{(k)}$ and $\underset{i\to\infty}{\lim}\xi_i(A_j\cdot c_j^-)^{(k)}=\xi(A_j\cdot c_j^-)^{(k)}$ for all $k=1,\dots,z_0$.
\end{itemize}
In particular, for any $(x,y,z_0)$ in $\Apc$, the triple
\begin{align}\label{hyperplane triple 1}
&\lim_{i\to\infty}\xi_i(A_j\cdot c_j^-)^{(1)}+\xi_i(a_j^-)^{(x-1)}+\xi_i(b_j^-)^{(y-1)}+\xi_i(c_j^-)^{(z_0)},\nonumber\\
&\lim_{i\to\infty}\xi_i(a_j^-)^{(x)}+\xi_i(b_j^-)^{(y-1)}+\xi_i(c_j^-)^{(z_0)}\text{ and }\\
&\lim_{i\to\infty}\xi_i(a_j^-)^{(x-1)}+\xi_i(b_j^-)^{(y)}+\xi_i(c_j^-)^{(z_0)}\nonumber
\end{align}
and the triple
\begin{align}\label{hyperplane triple 2}
&\lim_{i\to\infty}\xi_i(c_j^-)^{(1)}+\xi_i(a_j^-)^{(x-1)}+\xi_i(b_j^-)^{(y-1)}+\xi_i(A_j\cdot c_j^-)^{(z_0)},\nonumber\\
&\lim_{i\to\infty}\xi_i(a_j^-)^{(x)}+\xi_i(b_j^-)^{(y-1)}+\xi_i(A_j\cdot c_j^-)^{(z_0)}\text{ and }\\
&\lim_{i\to\infty}\xi_i(a_j^-)^{(x-1)}+\xi_i(b_j^-)^{(y)}+\xi_i(A_j\cdot c_j^-)^{(z_0)}\nonumber
\end{align}
are both pairwise distinct triples of hyperplanes. 

By Proposition \ref{triple ratio escaping 1}, we see that $\displaystyle\lim_{i\to\infty}\tau_{(x,y,z_0),j}(\rho_i)=\infty$ if and only if 
\[\lim_{i\to\infty}\xi_i(a_j^-)^{(x-1)}+\xi_i(b_j^-)^{(y-1)}+\xi_i(c_j^-)^{(z_0+1)}=\lim_{i\to\infty}\xi_i(a_j^-)^{(x)}+\xi_i(b_j^-)^{(y-1)}+\xi_i(c_j^-)^{(z_0)}\]
and $\displaystyle\lim_{i\to\infty}\tau_{(x,y,z_0),j}(\rho_i)=-\infty$ if and only if 
\[\lim_{i\to\infty}\xi_i(a_j^-)^{(x-1)}+\xi_i(b_j^-)^{(y-1)}+\xi_i(c_j^-)^{(z_0+1)}=\lim_{i\to\infty}\xi_i(a_j^-)^{(x-1)}+\xi_i(b_j^-)^{(y)}+\xi_i(c_j^-)^{(z_0)}.\]
Then (4) of Proposition \ref{basic cross ratio} together with the fact that the triple of hyperplanes (\ref{hyperplane triple 1}) are pairwise distinct imply that if $M={\xi_i(a_j^-)^{(x-1)}+\xi_i(b_j^-)^{(y-1)}+\xi_i(c_j^-)^{(z_0)}}$, then
\begin{eqnarray*}
\lim_{i\to\infty}\tau_{(x,y,z_0),j}(\rho_i)=\infty&\iff&\lim_{i\to\infty}(\xi_i(b_j^-),\xi_i(A_j\cdot c_j^-),\xi_i(a_j^-),\xi_i(c_j^-))_M=\infty.\\
\lim_{i\to\infty}\tau_{(x,y,z_0),j}(\rho_i)=-\infty&\iff&\lim_{i\to\infty}(\xi_i(A_j\cdot c_j^-),\xi_i(a_j^-),\xi_i(c_j^-),\xi_i(b_j^-))_M=\infty,
\end{eqnarray*}
 
The same argument, using the hyperplanes (\ref{hyperplane triple 2}) in place of (\ref{hyperplane triple 1}), proves that if $M=\xi_i(a_j^-)^{(x-1)}+\xi_i(b_j^-)^{(y-1)}+\xi_i(A\cdot c_j^-)^{(z_0)}$, then
\begin{eqnarray*}
\lim_{i\to\infty}\tau'_{(x,y,z_0),j}(\rho_i)=-\infty&\iff&\lim_{i\to\infty}(\xi_i(A_j\cdot c_j^-),\xi_i(a_j^-),\xi_i(c_j^-),\xi_i(b_j^-))_M=\infty.\\
\lim_{i\to\infty}\tau'_{(x,y,z_0),j}(\rho_i)=\infty&\iff&\lim_{i\to\infty}(\xi_i(b_j^-),\xi_i(A_j\cdot c_j^-),\xi_i(a_j^-),\xi_i(c_j^-))_M=\infty,
\end{eqnarray*}
Hence, (a) and (b) also hold.
\end{proof}

%%%%%%%%%%%%%%%%%%%%%%%%%%%%%%%%%%%%%%%%%%%%%%%%%%
\subsection{Proof of (2) of main theorem}
%%%%%%%%%%%%%%%%%%%%%%%%%%%%%%%%%%%%%%%%%%%%%%%%%%
Let $h:\Gamma\to\Zbbb_{\geq 0}$ be the function given by $h(X)=|\Bpc_X|$. Observe that $h$ is invariant under conjugation because the triangulation $\widetilde{\Tpc}$ is $\Gamma$-invariant. Hence, we can define 
\[\Gamma_0:=h^{-1}(0),\Gamma_1:=h^{-1}(\Zbbb_+),\]
and let $[\Gamma_i]$ be the set of conjugacy classes in $\Gamma_i$ for $i=0,1$. Observe also that a conjugacy class $[X]$ lies in $[\Gamma_0]$ if and only if $X=\id$ or $X=A^k$ for some integer $k$ and some $A$ in $\Gamma_\Ppc$.

Given a representation $\rho$ in $Hit_n(S)$, we want to produce an upper bound for $\htop(\rho)$. To do so, we will first find upper bounds for the size of the sets
\[\{[X]\in[\Gamma_0]:l_\rho(X)\leq T\}\text{ and }\{[X]\in[\Gamma_1]:l_\rho(X)\leq T\}\]
For some fixed $T>0$. These will give an upper bound on the size of 
\[\{[X]\in[\Gamma]:l_\rho(X)\leq T\},\] 
which we can then use to control $\htop(\rho)$.

\begin{lem}\label{Gamma_0}
Let $T>0$ and $\rho$ be a representation in $Hit_n(S)$. Then
\[|\{[X]\in[\Gamma_0]:l_\rho(X)\leq T\}|\leq(6g-6)\bigg\lfloor\frac{T}{L(\rho)}\bigg\rfloor+1,\]
where $g$ is the genus of $S$. Recall that $L(\rho)$ is the minimum of $l_\rho(A)$ over all $A$ in $\Gamma_\Ppc$, divided by $n$.
\end{lem}

\begin{proof}
Choose group elements $A_1,\dots, A_{3g-3}$ in $\Gamma$ corresponding to the $3g-3$ oriented simple closed curves in $\Ppc$. Observe that any conjugacy class in $[\Gamma_0]$ has a unique representative of the form $A_i^k$ for some $i=1,\dots,3g-3$ and some integer $k$. Moreover, for any representation $\rho$ in $Hit_n(S)$, we have
\[l_\rho(A_i^k)=|k|\cdot l_\rho(A_i)\geq|k|\cdot L(\rho).\] 
These observations imply that
\begin{eqnarray*}
|\{[X]\in[\Gamma_0]:l_\rho(X)\leq T\}|&\leq&\bigg|\bigg\{A_i^k\in\Gamma:i=1,\dots,3g-3;|k|\leq\frac{T}{L(\rho)}\bigg\}\bigg|\\
&=&(3g-3)\bigg(2\bigg\lfloor\frac{T}{L(\rho)}\bigg\rfloor\bigg)+1.
\end{eqnarray*}
\end{proof}

\begin{lem}\label{Gamma_1}
Let $T>0$ and $\rho$ be a representation in $Hit_n(S)$. Then
\[|\{[X]\in[\Gamma_1]:l_\rho(X)\leq T\}|\leq\sum_{a=1}^{\lfloor\frac{11T}{K(\rho)}\rfloor}\frac{(120g-120)^a}{a}\cdot{\lfloor\frac{11T-a\cdot K(\rho)}{L(\rho)}\rfloor+a\choose a},\]
where $g$ is the genus of $S$.
\end{lem}

\begin{proof}
Let $\Psi_\rho:=\{\psi_\rho(X):X\in\Gamma_1\}$, where $\psi_\rho(X)$ is the combinatorial data defined in Section \ref{Finite combinatorial description of closed curves}. By Proposition \ref{combinatorial prop}, the map $\psi_\rho:\Gamma_1\to\Psi_\rho$ descends to a bijection $\widehat{\psi}_\rho:[\Gamma_1]\to\Psi_\rho$. Hence, Theorem \ref{length lower bound} implies that 
\begin{eqnarray}\label{fixing binodal number}
|\{[X]\in[\Gamma_1]:l_\rho(X)\leq T\}|&\leq&|\{\sigma\in\Psi_\rho:r(\sigma)\cdot K(\rho)+s(\sigma)\cdot L(\rho)\leq 11T\}|\nonumber\\
&=&\sum_{a=1}^{\lfloor\frac{11T}{K(\rho)}\rfloor}\bigg|\bigg\{\sigma\in\Psi_\rho:r(\sigma)=a, s(\sigma)\leq\bigg\lfloor\frac{11T-a\cdot K(\rho)}{L(\rho)}\bigg\rfloor\bigg\}\bigg|.
\end{eqnarray}

For any cyclic sequence $\sigma=\{(\suc^{-1}(e_i),e_i,\suc(e_i),T_i,t(e_i,e_{i+1}))\}_{i=1}^{r(\sigma)}\in\Psi_\rho$, let $\pi_0(\sigma)$ be the cyclic sequence
\[\pi_0(\sigma):=\{(\suc^{-1}(e_i),e_i,\suc(e_i),T_i)\}_{i=1}^{r(\sigma)}\] 
and define $\Psi_{\rho,0}:=\{\pi_0(\sigma):\sigma\in\Psi_\rho\}$. For any $i$, let $j\in\{1,\dots,2g-2\}$ be the number such that $e_i$ lies in $\Qpc_j$. There are exactly two other edges in $\Qpc_j$, call them $e_i'$ and $e_i''$. This means that there are at most four possibilities for what $(\suc^{-1}(e_i),e_i,\suc(e_i),T_i)$ can be, namely 
\[(e_i'',e_i,e_i',Z), (e_i'',e_i,e_i',S), (e_i',e_i,e_i'',Z)\text{ or }(e_i',e_i,e_i'',S).\]
Hence, for any $a\geq 1$,
\begin{equation}\label{binodal possibilities}
|\{\pi_0(\sigma)\in\Psi_{\rho,0}:r(\sigma)=a\}|\leq \frac{(4\cdot(6g-6))^a}{a}.
\end{equation}

Next, we make two easy observations. First, consider the map $f:\Zbbb\to\Zbbb_{\geq 0}$ given by $f(a)=\max\{0,|a|-2\}$. Then for any non-negative integer $b$, observe that $|f^{-1}(b)|\leq 5$. Second, observe that for all positive integers $a$ and $k$,
\[\Big|\{(t_1,\dots, t_a)\in(\Zbbb_{\geq 0})^a:\sum_{i=1}^at_i\leq k\}\Big|={k+a\choose a}.\] 
These two observations, together with the inequality (\ref{binodal possibilities}), allow us to conclude that 
\begin{eqnarray*}
&&\bigg|\bigg\{\sigma\in\Psi_\rho:r(\sigma)=a, s(\sigma)\leq\bigg\lfloor\frac{11T-a\cdot K(\rho)}{L(\rho)}\bigg\rfloor\bigg\}\bigg|\\
&\leq&\frac{(4\cdot(6g-6))^a}{a}\cdot5^a\cdot{\lfloor\frac{11T-a\cdot K(\rho)}{L(\rho)}\rfloor+a\choose a}\\
&=& \frac{(120g-120)^a}{a}\cdot{\lfloor\frac{11T-a\cdot K(\rho)}{L(\rho)}\rfloor+a\choose a}.
\end{eqnarray*}
The above inequality together with inequality (\ref{fixing binodal number}) imply the lemma.
\end{proof}

\begin{prop}\label{final estimate}
Let $\rho$ be a representation in $Hit_n(S)$. Then
\[\htop(\rho)\leq11\limsup_{T\to\infty}\frac{1}{T}\log{\lfloor\frac{T-Q\cdot K(\rho)}{L(\rho)}\rfloor+Q\choose Q}+\frac{11\log(120g-120)}{K(\rho)},\]
where $Q\in\{0,1,\dots,\lfloor\frac{T}{K(\rho)}\rfloor\}$ is the integer so that 
\[{\lfloor\frac{T-Q\cdot K(\rho)}{L(\rho)}\rfloor+Q\choose Q}=\max_{a\in\{0,1,\dots,\lfloor\frac{T}{K(\rho)}\rfloor\}}{\lfloor\frac{T-a\cdot K(\rho)}{L(\rho)}\rfloor+a\choose a}.\]
\end{prop}

\begin{proof}
Since $[\Gamma]=[\Gamma_0]\cup[\Gamma_1]$, Lemma \ref{Gamma_0} and Lemma \ref{Gamma_1} imply that for sufficiently large $T$,
\begin{eqnarray*}
&&\frac{1}{T}\log|\{[X]\in[\Gamma]:l_\rho(X)\leq T\}|\\
&\leq&\frac{1}{T}\log\bigg((6g-6)\bigg\lfloor\frac{T}{L(\rho)}\bigg\rfloor+1+\sum_{a=1}^{\lfloor\frac{11T}{K(\rho)}\rfloor}\frac{(120g-120)^a}{a}\cdot{\lfloor\frac{11T-a\cdot K(\rho)}{L(\rho)}\rfloor+a\choose a}\bigg)\\
&\leq&\frac{1}{T}\log\bigg((6g-6)\bigg\lfloor\frac{T}{L(\rho)}\bigg\rfloor+1\bigg)+\frac{1}{T}\log\bigg(\frac{(120g-120)^{\lfloor\frac{11T}{K(\rho)}\rfloor}}{{\lfloor\frac{11T}{K(\rho)}\rfloor}}\bigg)\\
&&+\frac{1}{T}\log\bigg\lfloor\frac{11T}{K(\rho)}\bigg\rfloor+\frac{1}{T}\log{\lfloor\frac{11T-R\cdot K(\rho)}{L(\rho)}\rfloor+R\choose R}\\
&=&\frac{1}{T}\log\bigg((6g-6)\bigg\lfloor\frac{T}{L(\rho)}\bigg\rfloor+1\bigg)+\frac{\lfloor\frac{11T}{K(\rho)}\rfloor}{T}\log(120g-120)\\
&&+\frac{1}{T}\log{\lfloor\frac{11T-R\cdot K(\rho)}{L(\rho)}\rfloor+R\choose R}
\end{eqnarray*}
where $R\in\{0,1,\dots,\lfloor\frac{11T}{K(\rho)}\rfloor\}$ is the integer so that 
\[{\lfloor\frac{11T-R\cdot K(\rho)}{L(\rho)}\rfloor+R\choose R}=\max_{a\in\{0,1,\dots,\lfloor\frac{11T}{K(\rho)}\rfloor\}}{\lfloor\frac{11T-a\cdot K(\rho)}{L(\rho)}\rfloor+a\choose a}.\]

Since
\[\lim_{T\to\infty}\frac{1}{T}\log\bigg((6g-6)\bigg\lfloor\frac{T}{L(\rho)}\bigg\rfloor+1\bigg)=0,\]
we have 
\begin{eqnarray*}
&&\limsup_{T\to\infty}\frac{1}{T}\log|\{[X]\in[\Gamma]:l_\rho(X)\leq T\}|\\
&\leq&\limsup_{T\to\infty}\frac{1}{T}\log{\lfloor\frac{11T-R\cdot K(\rho)}{L(\rho)}\rfloor+R\choose R}+\frac{11\log(120g-120)}{K(\rho)}\\
&=&11\limsup_{T\to\infty}\frac{1}{T}\log{\lfloor\frac{T-Q\cdot K(\rho)}{L(\rho)}\rfloor+Q\choose Q}+\frac{11\log(120g-120)}{K(\rho)}.\\
\end{eqnarray*}
\end{proof}

By Proposition \ref{final estimate}, to finish the proof of (2) of Theorem \ref{main theorem}, it is now sufficient to prove the following proposition.

\begin{prop}
Let $\{\rho_i\}_{i=1}^\infty$ be an internal sequence in $Hit_n(S)$. Then
\[\lim_{i\to\infty}\limsup_{T\to\infty}\frac{1}{T}\log{\lfloor\frac{T-Q\cdot K(\rho_i)}{L(\rho_i)}\rfloor+Q\choose Q}=0,\]
where $Q\in\{0,1,\dots,\lfloor\frac{T}{K(\rho_i)}\rfloor\}$ is the integer so that 
\[{\lfloor\frac{T-Q\cdot K(\rho_i)}{L(\rho_i)}\rfloor+Q\choose Q}=\max_{a\in\{0,1,\dots,\lfloor\frac{T}{K(\rho_i)}\rfloor\}}{\lfloor\frac{T-a\cdot K(\rho_i)}{L(\rho_i)}\rfloor+a\choose a}.\]
\end{prop}

\begin{proof}
By the definition of an internal sequences, we know that the sequence $\{L(\rho_i)\}_{i=1}^\infty$ is bounded away from $0$ and $\infty$. This means in particular that 
\[L_0:=\inf_{i}\{L(\rho_i)\}\]
is a positive number. Also, by Theorem \ref{K to infinity}, we know that $\displaystyle\lim_{i\to\infty}K(\rho_i)=\infty$. This proposition thus follows if we can prove the following statement: \\
For any pair of sequences of positive numbers $\{T_j\}_{j=1}^\infty$ and $\{K_i\}_{i=1}^\infty$ such that $\displaystyle\lim_{i\to\infty}T_j=\infty$ and $\displaystyle\lim_{j\to\infty}K_i=\infty$, we have that 
\[\lim_{i\to\infty}\limsup_{j\to\infty}\frac{1}{T_j}\log{\lfloor\frac{T_j-Q_{i,j}\cdot K_i}{L_0}\rfloor+Q_{i,j}\choose Q_{i,j}}=0,\]
where $Q_{i,j}\in\{0,1,\dots,\lfloor\frac{T_j}{K_i}\rfloor\}$ is the integer so that 
\[{\lfloor\frac{T_j-Q_{i,j}\cdot K_i}{L_0}\rfloor+Q_{i,j}\choose Q_{i,j}}=\max_{a\in\{0,1,\dots,\lfloor\frac{T_j}{K_i}\rfloor\}}{\lfloor\frac{T_j-a\cdot K_i}{L_0}\rfloor+a\choose a}.\]
This is exactly Proposition 3.29 of \cite{Zha1}.
\end{proof}

\end{document}